\documentclass[11pt]{amsart}
\usepackage{amssymb,amsmath}
\usepackage{mathrsfs,enumitem}
\usepackage[all]{xy}
\usepackage{color}

\theoremstyle{plain}
\newtheorem{thm}{Theorem}[subsection]
\newtheorem{prop}[thm]{Proposition}
\newtheorem{lem}[thm]{Lemma}
\newtheorem{cor}[thm]{Corollary}
\newtheorem{defn}[thm]{Definition}

\newtheorem*{hyp}{Rough corank condition}
\theoremstyle{definition}
\newtheorem{rem}[thm]{Remark}
\newtheorem{ex}{Example}
\newtheorem{cex}{Counterexample}
\numberwithin{equation}{subsection}

\newcommand{\eps}{\varepsilon}
\renewcommand{\phi}{\varphi}
\renewcommand{\d}{\partial}
\newcommand{\dd}{\mathrm{d}}

\newcommand{\R}{\mathbf{R}}

\newcommand{\N}{\mathbf{N}}
\newcommand{\D}{C^{\infty}_0}
\renewcommand{\S}{\mathscr{S}}
\newcommand{\supp}{\mathop{\rm supp}}
\def\q{{\frac{1}{\vert B\vert}}}
\def\2q{{\frac{2}{\|B\|}}}
\def\|{{\Vert}}
\begin{document}
\title[Regularity of Fourier Integral Operators]{Global and local regularity of Fourier integral operators on weighted and unweighted spaces}
\author{David Dos Santos Ferreira}
\address{Universit\'e Paris 13, Cnrs, umr 7539 Laga, 99 avenue Jean-Baptiste Cl\'ement, F-93430 Villetaneuse, France}
\email{ddsf@math.univ-paris13.fr}
\thanks{D.~DSF. is partially supported by ANR grant Equa-disp.}
\author{Wolfgang Staubach}
\address{Department of Mathematics and the Maxwell Institute for Mathematical
Sciences, Heriot-Watt University, Edinburgh EH14 4AS, United Kingdom}
\email{W.Staubach@hw.ac.uk}
\thanks{W.~S. is partially supported by the Engineering and Physical Sciences Research Council First Grant Scheme, reference number EP/H051368/1.}

\begin{abstract}
We investigate the global continuity on $L^p$ spaces with $p\in [1,\infty]$ of Fourier integral operators with smooth and rough amplitudes and/or phase functions subject to certain non-degeneracy conditions. We initiate the investigation of the continuity of smooth and rough Fourier integral operators on weighted $L^{p}$ spaces, $L_{w}^p$ with $1< p < \infty$ and $w\in A_{p},$ (i.e. the Muckenhoput weights), and establish weighted norm inequalities for operators with rough and smooth amplitudes and phase functions satisfying a suitable rank condition. These results are then applied to prove weighted and unweighted estimates for the commutators of Fourier integral operators with functions of bounded mean oscillation BMO, then to some estimates on weighted Triebel-Lizorkin spaces, and finally to global unweighted and local weighted estimates for the solutions of the Cauchy problem for $m$-th and second order hyperbolic partial differential equations on $\mathbf{R}^n .$
\end{abstract}
\maketitle
\setcounter{tocdepth}{1}
\tableofcontents

\section*{Introduction and main results}
A Fourier integral operator is an operator that can be written locally in the form
\begin{align}\label{Intro:Fourier integral operator}
   T_{a, \varphi}u(x)= (2\pi)^{-n} \int_{\R^n} e^{i\phi(x,\xi)} a(x,\xi) \widehat{u}(\xi) \, \dd\xi,
\end{align}
where $a(x,\xi)$ is the {\it{amplitude}} and $\varphi(x,\xi)$ is the {\it{phase function}}.
Historically, a systematic study of smooth Fourier integral operators with amplitudes in $C^{\infty}(\mathbf{R}^{n} \times \mathbf{R}^n)$ with $|\partial^{\alpha}_{\xi} \partial^{\beta}_{\xi} a(x,\xi)| \leq C_{\alpha \beta} (1+|\xi|)^{m-\varrho|\alpha|+\delta|\beta|}$ (i.e. $a(x, \xi)\in S^{m}_{\varrho, \delta}$), and phase functions in $C^{\infty}(\mathbf{R}^{n} \times \mathbf{R}^n\setminus 0)$ homogenous of degree $1$ in the frequency variable $\xi$ and with non-vanishing determinant of the mixed Hessian matrix (i.e. {\it{non-degenerate phase functions}}), was initiated in the classical paper of L. H\"ormander \cite{H3}. Furthermore, I. Eskin \cite{Esk} (in the case $a(x,\xi)\in S^{0}_{1,0}$) and H\"ormander \cite{H3} (in the case $a(x,\xi) \in S^{0}_{\varrho, 1-\varrho},$ $\frac{1}{2} <\varrho \leq 1$) showed the local $L^2$ boundedness of Fourier integral operators with non-degenerate phase functions. Later on, H\"ormander's local $L^2$ result was extended by R. Beals \cite{RBE} and A. Greenleaf and G. Uhlmann \cite{GU} to the case of amplitudes in $S^{0}_{\frac{1}{2},\frac{1}{2}}.$ \\

After the pioneering investigations by M. Beals \cite{Be}, the optimal results concerning local continuity properties of smooth Fourier integral operators (with non-degenerate and homogeneous phase functions) in $L^{p}$ for $1\leq p\leq \infty$, were obtained in the seminal paper of A. Seeger, C. D. Sogge and E.M. Stein \cite{SSS}. This also paved the way for further investigations by G. Mockenhaupt, Seeger and Sogge in \cite{MSS1} and \cite{MSS2}, see also \cite{So} and \cite{So1}. In these investigations the boundedness, from $L^{p}_{\text{comp}}$ to $L^{p}_{\text{loc}}$ and from $L^{p}_{\text{comp}}$ to $L^{q}_{\text{loc}}$ of smooth Fourier integral operators with non-degenerate phase functions have been established, and furthermore it was shown that the maximal operators associated with certain Fourier integral operators (and in particular constant and variable coefficient hypersurface averages) are $L^p$ bounded.\\

In the context of H\"ormander type amplitudes and non-degenerate homogeneous phase functions which are most frequently used in applications in partial differential equations, it has been comparatively small amount of activity concerning global $L^p$ boundedness of Fourier integral operators. Among these, we would like to mention the global $L^2$ boundedness of Fourier integral operators with homogeneous phases in $C^{\infty}(\mathbf{R}^{n} \times \mathbf{R}^n \setminus 0)$ and amplitudes in the H\"ormander class $S^{0}_{0,0},$ due to D. Fujiwara \cite{Fuji}; the global $L^2$ boundedness of operators with inhomogeneous phases in $C^{\infty}(\mathbf{R}^{n} \times \mathbf{R}^n)$ and amplitudes in $S^{0}_{0,0},$ due to K. Asada and D. Fujiwara \cite{AF}; the global $L^p$ boundedness of operators with smooth amplitudes in the so called $\mathbf{SG}$ classes, due to E. Cordero, F. Nicola and L. Rodino in \cite{CNR}; and finally, S. Coriasco and M. Ruzhansky's global $L^{p}$ boundedness of Fourier integral operators \cite{CR}, with smooth amplitudes in a suitable subclass of the H\"ormander class $S^{0}_{1,0},$ where certain decay of the amplitudes in the spatial variables are assumed. We should also mention that before the appearance of the paper \cite{CR}, M. Ruzhansky and M. Sugimoto had already proved in \cite{Ruz 2} certain weighted $L^2$ boundedness (with some power weights) as well as the global unweighted $L^2$ boundedness of Fourier integrals operators with phases in $C^{\infty}(\mathbf{R}^{n} \times \mathbf{R}^{n})$ that are not necessarily homogeneous in the frequency variables, and amplitudes that are in the class $S^{0}_{0,0}.$ In all the aforementioned results, one has assumed ceratin bounds on the derivatives of the phase functions and also a stronger non-degeneracy condition than the one required in the local $L^p$ estimates. \\

In this paper, we shall take all these results as our point of departure and make a systematic study of the global $L^p$ boundedness of Fourier integral operators with amplitudes in $S^{m}_{\varrho, \delta}$ with $\varrho$ and $\delta$ in $[0,1],$ which cover all the possible ranges of $\varrho$'s and $\delta$'s. Furthermore we initiate the study of weighted norm inequalities for Fourier integral operators with weights in the $A_p$ class of Muckenhoupt and use our global unweighted $L^p$ results to prove a sharp weighted $L^p$ boundedness theorem for Fourier integral operators. The weighted results in turn will be used to establish the validity of certain vector-valued inequalities and more importantly to prove the weighted and unweighted boundedness of commutators of Fourier integral operators with functions of bounded mean oscillation BMO. Thus, all the results of this paper are connected and each chapter uses the results of the previous ones. This has been reflected in the structure of the paper and the presentation of the results.\\

Concerning the specific conditions that are put in this paper on the phase functions, it has been known at least since the appearance of the papers \cite{Fuji}, \cite{AF}, \cite{Ruz 2} and \cite{CR}, that one has to assume stronger conditions, than mere non-degeneracy, on the phase function in order to obtain global $L^{p}$ boundedness results. In fact it turns out that the assumption on the phase function, referred to in this paper as the {\it{strong non-degeneracy condition}}, which requires a nonzero lower bound on the modulus of the determinant of the mixed Hessian of the phase, is actually necessary for the validity of global regularity of Fourier integral operators, see section \ref{necessity of strong non-degeneracy}. Furthermore, we also introduce the class $\Phi^k$ of homogeneous (of degree 1) phase functions with a specific control over the derivatives of orders greater than or equal to $k,$ and assume our phases to be strongly non-degenerate and belong to $\Phi^k$ for some $k$. At first glance, these conditions might seem restrictive, but fortunately they are general enough to accommodate the phase functions arising in the study of hyperbolic partial differential equations and will still apply to the most generic phases in practical applications.\\

Concerning our choice of amplitudes, there are some features that set our investigations apart from those made previously, for example partly motivated by the investigation of C.E. Kenig and the second author of the present paper \cite{KS}, of the $L^{p}$ boundedness of the so called {\it{pseudo-pseudodifferential operators}}, we consider the global and local $L^p$ boundedness of Fourier integral operators when the amplitude $a(x,\xi)$ belongs to the class $L^{\infty}S^{m}_{\varrho}$, wherein $a(x,\xi)$ behaves in the spatial variable $x$ like an $L^\infty$ function, and in the frequency variable $\xi,$ the behaviour is that of a symbol in the H\"ormander class $S^{m}_{\varrho,0}.$\\

It is worth mentioning that the conditions defining classes $\Phi^k$, $L^{\infty}S^{m}_{\varrho}$ and the assumption of strong non-degeneracy make the global results obtained here natural extensions of the local boundedness results of Seeger, Sogge and Stein's in \cite{SSS}. Apart from the obvious local to global generalizations, this is because on one hand, our methods can handle the singularity of the phase function in the frequency variables at the origin and therefore the usual assumption that $\xi\neq 0$ in the support of the amplitude becomes obsolete. On the other hand, we do not require any regularity (and therefore no decay of derivatives) in the spatial variables of the amplitudes. Therefore, our amplitudes are close to, and in fact are spatially non-smooth versions of those in the Seeger-Sogge-Stein's paper \cite{SSS}. Indeed, in \cite{SSS} the authors although dealing with spatially smooth amplitudes, assume neither any decay in the spatial variables nor the vanishing of the amplitude in a neighbourhood of the singularity of the phase function.\\

There are several steps involved in the proof of the results of the paper and there are discussions about various conditions that we impose on the operators as well as some motivations for the study of rough operators. Moreover, giving examples and counterexamples when necessary, we have strived to give motivations for our assumptions in the statements of the theorems. Here we will not mention all the results that have been proven in this paper, instead we chose to highlight those that are the main outcomes of our investigations.\\

\noindent In Chapter 1, we set up the notations and give the definitions of the classes of amplitudes, phase functions and weights that will be used througt the paper. We also include the tools that we need in proving our global boundedness results, in this chapter. We close the chapter with a discussion about the connections between rough amplitudes and global boundedness of Fourier integral operators.\\

Chapter 2 is devoted to the investigation of the global boundedness of Fourier integral operators with smooth or rough phases, and smooth or rough amplitudes.
To achieve our global boundedness results, we split the operators in low and high frequency parts and show the boundedness of each and one of them separately. In proving the $L^p$ boundedness of the low frequency portion, see Theorem \ref{general low frequency boundedness for rough Fourier integral operator}, we utilise Lemma \ref{main low frequency estim} which yields a favourable kernel estimate for the low frequency operator, thereafter we use the phase reduction of Lemma \ref{phase reduction} to bring the operator to a canonical form, and finally we use the $L^p$ boundedness of the non-smooth substation in Corollary \ref{cor:main substitution estim} to conclude the proof. Thus, we are able to establish the global $L^p$ boundedness of the frequency-localised portion of the operator, for all $p$'s in $[1,\infty]$ simultaneously.\\

The global boundedness of the high frequency portion of the Fourier integral operator needs to be investigated in three steps. First we show the $L^1-L^1$ boundedness then we proceed to the $L^2-L^2$ boundedness and finally we prove the $L^{\infty}-L^{\infty}$ boundedness.\\

In order to show the $L^1$ boundedness of Theorem \ref{Intro:L1Thm}, we use a semi-classical reduction from Subsection \ref{Semiclasical reduction subsec} and Lemma \ref{Lp:semiclassical}, which will be used throughout the paper. Thereafter we use the semiclassical version of the Seeger-Sogge-Stein decomposition which was introduced in a microlocal form in \cite{SSS}.\\

For our global $L^2$ boundedness result, we also consider amplitudes with $\varrho<\delta$ which appear in the study of Fourier integral operators on nilpotent Lie groups and also in scattering theory. In Theorem \ref{Calderon-Vaillancourt for FIOs}, we show a global $L^2$ boundedness result for operators with smooth H\"ormander class amplitudes in $S^{\min(0,\frac{n}{2}(\varrho-\delta))}_{\varrho,\delta},$ $\varrho\in [0,1],$ $\delta \in [0,1).$  Also, in Theorem \ref{Intro:L2Thm} we prove the $L^2$ boundedness of operators with amplitudes belonging to $S^{m}_{\varrho,1},$ with $m<\frac{n}{2}(\varrho-1)$. In both of these theorems, the phase functions are assumed to satisfy the strong non-degeneracy condition and both of these results are sharp. These can be viewed as extensions of the celebrated Calder\'on-Vaillancourt theorem \cite{CV} to the case of Fourier integral operators. Indeed, the phase function of a pseudodifferential operator, which is $\langle x, \xi \rangle$ is in class $\Phi^2$ and satisfies the strong non-degeneracy condition and therefore our $L^2$ boundedness result completes the $L^2$ boundedness theory of smooth Fourier integral operators with homogeneous non-degenerate phases.\\
Finally, in Theorem \ref{Intro:LinftyThm} we prove the sharp global $L^{\infty}$ boundedness of Fourier integral operators, where in the proof we follow almost the same line of argument as in the proof of the $L^1$ boundedness case, but to obtain the sharp result which we desire, we make a more detailed analysis of the kernel estimates which bring us beyond the result implied by the mere utilisation of the Seeger-Sogge-Stein decomposition. Furthermore, in this case, no non-degeneracy assumption on the phase is required.
Our results above are summerised in the following global $L^p$ boundedness theorem, see Theorem \ref{main L^p thm for smooth Fourier integral operators}:
\subsection*{A. Global $L^{p}$ boundedness of smooth Fourier integral operators.}
Let $T$ be a Fourier integral operator given by \ref{Intro:Fourier integral operator} with amplitude $a \in S^m_{\varrho, \delta}$ and a strongly non-degenerate phase function $\varphi(x,\xi) \in \Phi^2.$ Setting $\lambda:=\min(0,n(\varrho-\delta)),$ suppose that either of the following conditions hold:
\begin{enumerate}
\item[(a)]  $1 \leq p \leq 2,$ $0\leq \varrho \leq 1,$ $0\leq \delta \leq 1,$ and
                      $$ m<n(\varrho -1)\bigg (\frac{2}{p}-1\bigg)+\big(n-1\big)\bigg(\frac{1}{2}-\frac{1}{p}\bigg)+ \lambda\bigg(1-\frac{1}{p}\bigg); $$ or
\item[(b)]  $2 \leq p \leq \infty,$ $0\leq \varrho \leq 1,$ $0\leq \delta \leq 1,$ and
                      $$ m<n(\varrho -1)\bigg (\frac{1}{2}-\frac{1}{p}\bigg)+ (n-1)\bigg(\frac{1}{p}-\frac{1}{2}\bigg) +\frac{\lambda}{p};$$ or
\item[(c)]   $p=2,$ $0\leq \varrho\leq 1,$ $0\leq \delta<1,$ and
$$m= \frac{\lambda}{2}.$$
\end{enumerate}
Then there exists a constant $C>0$ such that $ \Vert Tu\Vert_{L^{p}} \leq C \Vert u\Vert_{L^{p}}.$
For Fourier integral operators with rough amplitudes we show in Theorem \ref{main L^p thm for Fourier integral operators with smooth phase and rough amplitudes} the following:

\subsection*{B. Global $L^{p}$ boundedness of rough Fourier integral operators.}
\noindent Let $T$ be a Fourier integral operator given by \eqref{Intro:Fourier integral operator} with amplitude $a \in L^{\infty}S^m_{\varrho}, 0\leq \varrho \leq 1$ and a strongly non-degenerate phase function $\varphi(x,\xi) \in \Phi^2.$ Suppose that either of the following conditions hold:
      \begin{enumerate}
\item[(a)]  $1 \leq p \leq 2$ and
                      $$ m<\frac{n}{p}(\varrho -1)+\big(n-1\big)\bigg(\frac{1}{2}-\frac{1}{p}\bigg); $$ or
\item[(b)]  $2 \leq p \leq \infty$ and
                      $$ m<\frac{n}{2}(\varrho -1)+ (n-1)\bigg(\frac{1}{p}-\frac{1}{2}\bigg).$$
\end{enumerate}
Then there exists a constant $C>0$ such that $ \Vert Tu\Vert_{L^{p}} \leq C \Vert u\Vert_{L^{p}}.$
We also extend both of the results above, i.e. the $L^p-L^p$ regularity of smooth and rough operators, to the $L^p-L^q$ regularity, in Theorem \ref{main LpLq thm for Fourier integral operators}.\\
After proving the global regularity of Fourier integral operators with smooth phase functions, we turn to the problem of local and global boundedness of operators which are merely bounded in the spatial variables in both their phases and amplitudes.
A motivation for this investigation stems from the study of maximal estimates involving Fourier integral operators, where a standard stopping time argument (here referred to as linearisation) reduces the problem to a Fourier integral operator with a non-smooth phase and sometimes also a non-smooth amplitude. For instance, estimates for the maximal spherical average operator
\begin{align*}
   A u(x) = \sup_{t \in [0,1]} \Big| \int_{S^{n-1}} u(x+t\omega) \, \dd\sigma(\omega) \Big|
\end{align*}
which is directly related to the rough Fourier integral operator
$$ T u(x) =(2\pi)^{-n} \int_{\R^n} a(x,\xi) e^{it(x)|\xi|+i\langle x,\xi \rangle} \widehat{u}(\xi) \, \dd\xi$$
where $t(x)$ is a measurable function in $x$, with values in $[0,1]$ and $a(x,\xi) \in L^{\infty} S_{1}^{-\frac{n-1}{2}}.$ Here, the phase function of the Fourier integral operator is $\varphi(x,\xi)=it(x)|\xi|+i\langle x,\xi \rangle$ which is merely an $L^{\infty}$ function in the spatial variables $x$, but is smooth outside the origin in the frequency variables $\xi.$  As we shall see later, according to Definition \ref{defn of rough phases}, this phase belongs to the class $L^{\infty} \Phi^{2}.$\\
In our investigation of local or global $L^p$ boundedness of the rough Fourier integral operators above for $p\neq 2$, the results obtained are similar to those of the local results for amplitudes in the class $S^{m}_{1,0}$ obtained in \cite{MSS1}, \cite{MSS2} and \cite{So1} for \eqref{Intro:LinWave}. However, we consider more general classes of amplitudes (i.e. the $S^{m}_{\varrho, \delta}$ class) and also require only measurability and boundedness in the spatial variables (i.e. the $L^{\infty}S^{m}_{\varrho}$ class).
The main results in this context are the $L^2$ boundedness results which apart from the case of Fourier integral operators in dimension $n=1,$ yield a  problem of considerable difficulty in case one desires to prove a $L^2$ regularity result under the sole assumption of rough non-degeneracy, see Definition \ref{defn of rough nondegeneracy}.

Using the geometric conditions (imposed on the phase functions) which are the rough analogues of the non-degeneracy and corank conditions for smooth phases (the rough corank condition \ref{Rough corank condition}), we are able to prove a local $L^2$ boundedness result with a certain loss of derivatives depending on the rough corank of the phase. More explicitly we prove in Theorem \ref{Intro:L2ThmDeg}:

\subsection*{C. Local $L^2$ boundedness of Fourier integral operators with rough amplitudes and phases.}
\noindent Let $T$ be a Fourier integral operator given by \eqref{Intro:Fourier integral operator} with amplitude $a \in L^{\infty}S^m_\varrho$ and phase function $\phi \in L^{\infty}\Phi^2$.
      Suppose that the phase satisfies the rough corank condition \ref{Rough corank condition}, then  $T$ can be extended as a bounded operator from $L^{2}_{\rm comp}$
      to $L^{2}_{\rm loc}$ provided $m<-\frac{n+k-1}{4}+\frac{(n-k)(\varrho-1)}{2}.$\\
\noindent Despite the lack of sharpness in the above theorem, the proof is rather technical. However, in case $n=1$ this theorem can be improved to yield a local $L^2$ boundedness result with $m<\frac{-1}{2},$ and if the assumptions on the phase function are also strengthen with a Lipschitz condition on the $\xi$ derivatives of order $2$ and higher of the phase, then the above theorem holds with a loss $m<-\frac{k}{2}+\frac{(n-k)(\varrho-1)}{2}.$\\

In Chapter 3 we turn to the problem of weighted norm inequalities for Fourier integral operators. To our knowledge this question has never been investigated previously in the context of Muckenhoupt's $A_p$ weights which are the most natural class of weights in harmonic analysis. Here we start this investigation by establishing sharp boundedness results for a fairly wide class of Fourier integral operators, somehow modeled on the parametrices of hyperbolic partial differential equations. One notable feature of our investigation is that we also prove the results for Fourier integral operators whose phase functions and amplitudes are only bounded and measurable in the spatial variables and exhibit suitable symbol type behavior in the frequency variables.\\
As before, we begin by discussing the weighted estimates for the low frequency portion of the Fourier integral operators which can be handled by Lemma \ref{main low frequency estim}. As a matter of fact, the weighted $L^p$ boundedness of low frequency parts of Fourier integral operators is merely an analytic issue involving the right decay rates of the phase function and does not involve any rank condition on the phase. The situation in the high frequency case is entirely different. Here, there is also a significant distinction between the weighted a and unweighted case, in the sense that, if one desires to prove sharp weighted estimates, then a rank condition on the phase function is absolutely crucial. This fact has been discussed in detail in Section \ref{counterexamples in weighted setting}, where one finds various examples, including one related to the wave equation, and counterexamples which will lead us to the correct condition on the phase. Then we will proceed with the local high frequency and global high frequency boundedness estimates. As a rule, in the investigation of boundedness of Fourier integral operators, the local estimates require somewhat milder conditions on the phase functions compared to the global estimates and our case study of the weighted norm inequalities here is no exception to this rule. Furthermore, we are able to formulate the local weighted boundedness results in an invariant way involving the canonical relation of the Fourier integral operator in question. Our main results in this context are contained in Theorem \ref{endpointweightedboundthm}:

\subsection*{D. Weighted $L^{p}$ boundedness of Fourier integral operators.}
Let $a(x,\xi)\in L^{\infty}S^{-\frac{n+1}{2}\varrho+n(\varrho -1)}_{\varrho}$ and $\varrho \in [0,1].$ Suppose that either
\begin{enumerate}
\item[(a)] $a(x,\xi)$ is compactly supported in the $x$ variable and the phase function $\varphi(x,\xi)\in C^{\infty} (\mathbf{R}^ n \times \mathbf{R}^{n}\setminus 0)$, is positively homogeneous of degree $1$ in $\xi$ and satisfies, $\det\partial^2_{x\xi}\varphi(x,\xi) \neq 0$ as well as $\mathrm{rank}\,\partial^2_{\xi\xi}\varphi(x,\xi) =n-1$; or
\item[(b)] $\varphi (x,\xi)-\langle x,\xi\rangle \in L^{\infty}\Phi^1,$ $\varphi $ satisfies the rough non-degeneracy condition as well as $|\mathrm{det}_{n-1} \partial^2_{\xi\xi}\varphi(x,\xi)|\geq c>0$.
\end{enumerate}
Then the operator $T_{a,\varphi}$ is bounded on $L^{p}_{w}$ for $p\in (1,\infty)$ and all $w\in A_p$. Furthermore, for $\varrho=1$ this result is sharp.\\
Here, it is worth mentioning that in the non-endpoint case, i.e. if $a(x,\xi)\in L^{\infty}S^{m}_{\varrho}$ with $m<-\frac{n+1}{2}\varrho+n(\varrho -1),$ we can prove a result that requires no non-degeneracy assumption on the phase function.
The proof of these statements are long and technical and use several steps involving careful kernel estimates, uniform pointwise estimates on certain oscillatory integrals, unweighted local and global $L^p$ boundedness, interpolation, and extrapolation.\\

In Chapter 4 we are motivated by the fact that weighted norm inequalities with $A_p$ weights can be used as an efficient tool in proving vector valued inequalities and also boundedness of commutators of operators with {\it{functions of bounded mean oscillation}} BMO. Therefore, we start the chapter by showing boundedness of certain Fourier integral operators in {\it{weighted Triebel-Lizorkin spaces}} (see \eqref{Tribliz definition}). This is based on a vector valued inequality for Fourier integral operators.\\
But more importantly we prove for the first time, in Theorems \ref{Commutator estimates for FIO} and \ref{k-th commutator estimates}, the boundedness and weighted boundedness of BMO commutators of Fourier integral operators, namely
\subsection*{E. $L^{p}$ boundedness of BMO commutators of Fourier integral operators.}
\noindent Suppose either
\begin{enumerate}
  \item [(a)] $T\in I^{m}_{\varrho, \mathrm{comp}}(\mathbf{R}^{n}\times \mathbf{R}^{n}; \mathcal{C})$ with $\frac{1}{2} \leq \varrho\leq 1$ and $m<(\varrho-n)|\frac{1}{p}-\frac{1}{2}|,$ satisfies all the conditions of Theorem \ref{weighted boundedness for true amplitudes with power weights} or;
  \item [(b)] $T_{a,\varphi}$ with $a\in S^{m}_{\varrho, \delta},$ $0\leq \varrho \leq 1$, $0\leq \delta\leq 1,$ $\lambda= \min(0, n(\varrho-\delta))$ and $\varphi(x,\xi)$ is a strongly non-degenerate phase function with $\varphi(x,\xi)-\langle x,\xi\rangle \in \Phi^1,$ where in the range $1<p\leq 2,$
$$ m<n(\varrho -1)\bigg (\frac{2}{p}-1\bigg)+\big(n-1\big)\bigg(\frac{1}{2}-\frac{1}{p}\bigg)+ \lambda\bigg(1-\frac{1}{p}\bigg);$$ and in the range $2 \leq p <\infty$
                      $$ m<n(\varrho -1)\bigg (\frac{1}{2}-\frac{1}{p}\bigg)+ (n-1)\bigg(\frac{1}{p}-\frac{1}{2}\bigg) +\frac{\lambda}{p};$$ or
\item [(c)] $T_{a,\varphi}$ with $a\in L^{\infty}S^{m}_{\varrho},$ $0\leq \varrho \leq 1$ and $\varphi$ is a strongly non-degenerate phase function with $\varphi(x,\xi) -\langle x, \xi\rangle \in \Phi^1,$ where in the range $1<p\leq 2,$
$$ m<\frac{n}{p}(\varrho -1)+\big(n-1\big)\bigg(\frac{1}{2}-\frac{1}{p}\bigg),$$
and for the range $2 \leq p <\infty$
                      $$ m<\frac{n}{2}(\varrho -1)+ (n-1)\bigg(\frac{1}{p}-\frac{1}{2}\bigg).$$
\end{enumerate}
Then for $b\in \mathrm{BMO}$, the commutators $[b, T]$ and $[b, T_{a,\varphi}]$ are bounded on $L^p$ with $1<p<\infty.$ Here we like to mention that once again, the global $L^p$ bounded in Theorem \textbf{A} above is used in the proof of the $L^p$ boundedness of the BMO commutators.
Finally, the weighted norm inequalities with weights in all $A_p$ classes have the advantage of implying weighted boundedness of repeated commutators, namely one has
\subsection*{F. Weighted $L^{p}$ boundedness of k-th BMO commutators of Fourier integral operators.}
\noindent Let $a(x,\xi)\in L^{\infty}S^{-\frac{n+1}{2}\varrho+n(\varrho -1)}_{\varrho}$ and $\varrho \in[0,1].$ Suppose that either
\begin{enumerate}
\item[(a)] $a(x,\xi)$ is compactly supported in the $x$ variable and the phase function $\varphi(x,\xi)\in C^{\infty} (\mathbf{R}^ n \times \mathbf{R}^{n}\setminus 0)$, is positively homogeneous of degree $1$ in $\xi$ and satisfies, $\det\partial^2_{x\xi}\varphi(x,\xi) \neq 0$ as well as $\mathrm{rank}\,\partial^2_{\xi\xi}\varphi(x,\xi) =n-1$; or
\item[(b)] $\varphi (x,\xi)-\langle x,\xi\rangle \in L^{\infty}\Phi^1,$ $\varphi $ satisfies the rough non-degeneracy condition as well as $|\mathrm{det}_{n-1} \partial^2_{\xi\xi}\varphi(x,\xi)|\geq c>0$.
\end{enumerate}
Then, for $b \in \mathrm{BMO}$ and $k$ a positive integer, the $k$-th commutator defined by
\begin{equation*}
T_{a, b,k} u(x):= T_{a}\big((b(x)-b(\cdot))^{k}u\big)(x)
\end{equation*}
is bounded on $L^{p}_w$ for each $w \in A_p$ and $p\in(1, \infty)$.\\
These BMO estimates have no predecessors in the literature and are useful in connection to the study of hyperbolic partial differential equations with rough coefficients.\\
In the last section of Chapter 4, we also briefly discuss global unweighted and local weighted estimates for the solutions of the Cauchy problem for $m$-th and second order hyperbolic partial differential equations.
\subsubsection*{Acknowledgements} Part of this work was undertaken while one of the authors was visiting the department of Mathematics of the Heriot-Watt University. The first author wishes to express his gratitude for the hospitality of Heriot-Watt University.

\section{Prolegomena}
In this chapter, we gather some results which will be useful in the study of boundedness of Fourier integral operators. We also illustrate some of the connections
between global boundedness results for operators with smooth phases and amplitudes and local boundedness results for operators with rough phases and amplitudes,
thus justifying a joint study of those operators.
\subsection{Definitions, notations and preliminaries} \label{prelim}
\subsubsection{Phases and amplitudes}
In our investigation of the regularity properties of Fourier integral operators, we will be concerned with both smooth and non-smooth amplitudes and phase functions. Below, we shall recall some basic definitions and fix some notations which will be used throughout the paper. Also, in the sequel we use the notation $\langle \xi\rangle$ for $(1+|\xi|^2)^{\frac{1}{2}}.$ The following definition which is due to H\"ormander \cite{H0}, yields one of the most widely used classes of smooth symbols/amplitudes.
\begin{defn}\label{defn of hormander amplitudes}
Let $m\in \mathbf{R}$, $0\leq \delta\leq 1$, $0\leq \varrho\leq 1.$ A function $a(x,\xi)\in C^{\infty}(\mathbf{R}^{n} \times\mathbf{R}^{n})$ belongs to the class $S^{m}_{\varrho,\delta}$, if for all multi-indices $\alpha, \, \beta$
   it satisfies
   \begin{align*}
      \sup_{\xi \in \R^n} \langle \xi\rangle ^{-m+\varrho\vert \alpha\vert- \delta |\beta|}
      |\partial_{\xi}^{\alpha}\partial_{x}^{\beta}a(x,\xi)|< +\infty.
   \end{align*}
\end{defn}
We shall also deal with the class $L^{\infty}S^m_{\varrho}$ of rough symbols/amplitudes introduced by Kenig and Staubach in \cite{KS}.
\begin{defn}\label{defn of amplitudes}
   Let $m\in \R$ and $0\leq\varrho \leq 1$. A function $a(x,\xi)$ which is smooth in the frequency variable $\xi$ and bounded
   measurable in the spatial variable $x$, belongs to the symbol class $L^{\infty}S^{m}_{\varrho}$, if for all multi-indices $\alpha$
   it satisfies
   \begin{align*}
      \sup_{\xi \in \R^n} \langle \xi\rangle ^{-m+\varrho\vert \alpha\vert}
      \Vert \partial_{\xi}^{\alpha}a(\cdot\,,\xi)\Vert_{L^{\infty}(\R^{n})}< +\infty.
   \end{align*}
\end{defn}
We also need to describe the type of phase functions that we will deal with. To this end, the class $\Phi^{k}$  defined below, will play a significant role in our investigations.
\begin{defn}\label{Phik phases}
A real valued function $\phi(x,\xi)$ belongs to the class $\Phi^{k}$, if $\varphi (x,\xi)\in C^{\infty}(\R^n \times\R^n \setminus 0)$, is positively homogeneous of degree $1$ in the frequency variable $\xi$, and satisfies the following condition:
For any pair of multi-indices $\alpha$ and $\beta$, satisfying $|\alpha|+|\beta|\geq k$, there exists a positive constant $C_{\alpha, \beta}$ such that
   \begin{align*}
      \sup_{(x,\,\xi) \in \R^n \times\R^n \setminus 0}  |\xi| ^{-1+\vert \alpha\vert}\vert \partial_{\xi}^{\alpha}\partial_{x}^{\beta}\phi(x,\xi)\vert
      \leq C_{\alpha , \beta}.
   \end{align*}
\end{defn}
In connection to the problem of local boundedness of Fourier integral operators, one considers phase functions $\varphi(x,\xi)$ that are positively homogeneous of degree $1$ in the frequency variable $\xi$ for which $\det [\partial^{2}_{x_{j}\xi_{k}} \varphi(x,\xi)]\neq 0.$ The latter is referred to as the {\it{non-degeneracy condition}}. However, for the purpose of proving global regularity results, we require a stronger condition than the aforementioned weak non-degeneracy condition.
\begin{defn}\label{strong non-degeneracy} $($The strong non-degeneracy condition$).$ A real valued phase $\varphi\in C^{2}(\R^n \times\R^n \setminus 0)$ satisfies the strong non-degeneracy condition, if there exists a positive constant $c$ such that
                $$ \Big|\det \frac{\partial^{2}\varphi(x,\xi)}{\partial x_j \partial \xi_k}\Big| \geq c, $$
           for all $(x,\,\xi) \in \R^n \times\R^n \setminus 0$.
\end{defn}
The phases in class $\Phi^2$ satisfying the strong non-degeneracy condition arise naturally in the study of hyperbolic partial differential equations, indeed a phase function closely related to that of the wave operator, namely $\varphi(x,\xi)= |\xi|+ \langle x, \xi\rangle$ belongs to the class $\Phi^2$ and is strongly non-degenerate.\\
We also introduce the non-smooth version of the class $\Phi^k$ which will be used throughout the paper.
\begin{defn}\label{defn of rough phases}
   A real valued function $\phi(x,\xi)$ belongs to the phase class $L^{\infty}\Phi^{k}$, if it is homogeneous of degree $1$ and smooth on $\R^n \setminus 0$
   in the frequency variable $\xi$, bounded measurable in the spatial variable $x,$ and if for all multi-indices $|\alpha|\geq k$ it satisfies
   \begin{align*}
      \sup_{\xi \in \R^n \setminus 0}  |\xi| ^{-1+\vert \alpha\vert}\Vert \partial_{\xi}^{\alpha}\phi(\cdot\,,\xi)\Vert_{L^{\infty}(\R^{n})}< +\infty.
   \end{align*}
\end{defn}
We observe that if $t(x) \in L^{\infty}$ then the phase function $\varphi(x,\xi)= t(x)|\xi|+ \langle x, \xi\rangle$ belongs to the class $L^\infty\Phi^2 ,$ hence phase functions originating from the linearisation of the maximal functions associated with averages on surfaces, can be considered as members of the $L^{\infty} \Phi^2$ class.
We will also need a rough analogue of the non-degeneracy condition, which we define below.
\begin{defn}\label{defn of rough nondegeneracy} $($The rough non-degeneracy condition$).$ A real valued phase $\varphi$ satisfies the rough non-degeneracy condition, if it is $C^1$ on $\R^n \setminus 0$
   in the frequency variable $\xi$, bounded measurable in the spatial variable $x,$ and there exists a constant $c>0$ $($depending only on the dimension$)$ such that for all $x,y\in \mathbf{R}^n$ and $\xi\in \mathbf{R}^n \setminus 0$
       \begin{equation}\label{lower bound on gradients}
        |\d_{\xi}\phi(x,\xi)-\d_{\xi}\phi(y,\xi)| \geq c |x-y|.
       \end{equation}
\end{defn}

\subsubsection{Basic notions of weighted inequalities}
Our main reference for the material in this section are \cite{G} and \cite{S}. Given $u\in L^{p}_{\mathrm{loc}}$, the $L^p$ maximal function $M_p(u)$ is defined by
\begin{equation}
M_p(u)(x) = \sup_{B\ni x} \left(\q \int_{B} \vert u(y)\vert^{p} \, \dd y\right)^{\frac{1}{p}}
\end{equation}
where the supremum is taken over balls $B$ in $\R^{n}$ containing $x$. Clearly then, the Hardy-Littlewood maximal
function is given by
\[
M(u) := M_{1}(u).
\]
An immediate consequence of H\"older's inequality is that $M(u)(x)\leq M_{p}(u)(x)$ for $p\geq 1$.
We shall use the notation
\[
u_{B}:= \q \int_{B} \vert u(y)\vert \, \dd y
\]
for the average of the function $u$ over $B$.
One can then define the class of Muckenhoupt $A_p$ weights as follows.
\begin{defn} \label{weights}
Let $w\in L^{1}_{\mathrm{loc}}$ be a positive function. One says that $w\in A_1$ if there exists a constant $C>0$ such that
\begin{equation}
 M w (x)\leq C w(x), \quad \text{ for almost all } \quad x \in \R^{n}.
\end{equation}
One says that $w\in A_p$ for $p\in(1,\infty)$ if
\begin{equation}
 \sup_{B\, \textrm{balls in}\,\, \R^{n}}\,w_{B}\big(w^{-\frac{1}{p-1}}\big)_{B}^{p-1}<\infty.
\end{equation}
The $A_p$ constants of a weight $w\in A_p$ are defined by
\begin{equation}
 [w]_{A_1}:=   \sup_{B\, \textrm{balls in}\,\, \R^{n}}\,w_{B}\Vert w^{-1}\Vert_{L^{\infty}(B)},
 \end{equation}
 and
\begin{equation}\label{ap constant}
 [w]_{A_p}:=  \sup_{B\, \textrm{balls in}\,\, \R^{n}}\,w_{B}\big(w^{-\frac{1}{p-1}}\big)_{B}^{p-1}.
\end{equation}
\end{defn}
\begin{ex}\label{examples of weights}
The function $|x|^\alpha$ is in $A_1$ if and only if $-n<\alpha\leq 0$ and is in $A_p$ with $1<p<\infty$ iff $-n<\alpha<n(p-1)$.
Also $u(x)= \log{\frac{1}{|x|}} $ when $|x|<\frac{1}{e}$ and $u(x)=1$ otherwise, is an $A_1$ weight.
\end{ex}
\subsubsection{Additional conventions}
As is common practice, we will denote constants which can be determined by known parameters in a given situation, but whose
value is not crucial to the problem at hand, by $C$. Such parameters in this paper would be, for example, $m$, $\varrho$, $p$, $n$, $[w]_{A_p}$, and the constants $C_\alpha$ in Definition \ref{defn of amplitudes}. The value of $C$ may differ from line to line, but in each instance could be estimated if necessary. We sometimes write $a\lesssim b$ as shorthand for $a\leq Cb$.
Our goal is to prove estimates of the form
   $$ \Vert Tu\Vert_{L^p} \leq C \Vert u\Vert_{L^p}, \quad u \in \mathscr{S}(\R^n) $$
when $a \in L^{\infty}S^m_{\varrho}$, $\phi \in L^{\infty}\Phi^{k}$ and $m < -\sigma \leq 0$ or equivalently
   $$ \Vert Tu\Vert_{L^p} \leq C \Vert u\Vert_{H^{s,p}}, \quad u \in \mathscr{S}(\R^n) $$
when $a \in L^{\infty}S^0_{\varrho}$ and $s > \sigma$ and $H^{s,p}:=\{u\in \mathscr{S}';\, (I-\Delta)^{\frac{s}{2}}u\in L^{p}\}$. We will use indifferently one or the other equivalent formulation and we will refer to $\sigma$ as the loss
of derivatives in the $L^p$ boundedness of $T$.
\subsection{Tools in proving $L^p$ boundedness}
\subsubsection{Semi-classical reduction and decomposition of the operators}\label{Semiclasical reduction subsec}
It is convenient to work with semi-classical estimates: let $A$ be the annulus
   $$ A=\big\{ \xi \in \R^n; \tfrac{1}{2} \leq |\xi| \leq 2 \big\} $$
and $\chi \in \D(A)$ be a cutoff function, we will prove estimates on the following semi-classical
Fourier integral operator
   $$ T_hu = (2\pi h)^{-n} \int_{\R^n} e^{\frac{i}{h} \phi(x,\xi)} \chi(\xi) a(x,\xi/h)
      \widehat{u}(\xi/h) \, \dd\xi $$
with $h \in (0,1]$. We will also need to investigate the low frequency component of the operator
   $$ T_0u = (2\pi)^{-n} \int_{\R^n} e^{i \phi(x,\xi)} \chi_{0}(\xi) a(x,\xi) \widehat{u}(\xi) \, \dd\xi $$
where $\chi_{0} \in \D(B(0,2))$. The following lemma shows how semi-classical estimates translate into classical ones.
We choose to state the result in the realm of weighted $L^p$ spaces with weights in the Muckenhoupt's $A_p$ class. This extent of generality will be needed when we deal with the weighted boundedness of Fourier integral operators.
\begin{lem}
\label{Lp:semiclassical}
      Let $a \in L^{\infty}S^m_{\varrho}$ and $\phi \in L^{\infty}\Phi^{k}$, suppose that
      for all $h \in (0,1]$ and $w\in A_p$, there exist constants $C_{1},C_{2}>0$ (only depending on the $A_p$ constants of $w$) such that the following estimates hold
         $$ \Vert T_0 u\Vert_{L_{w}^p} \leq C_{0}  \Vert u\Vert_{L_{w}^p}, \quad \Vert T_hu\Vert_{L_{w}^p} \leq C_{1} h^{-m-s} \Vert u\Vert_{L_{w}^p}, \quad u \in \S(\R^n). $$
      This  implies the bound
         $$ \Vert Tu\Vert_{L_{w}^p} \leq C_{2} \Vert u\Vert_{L_{w}^{p}}, \quad u \in \S(\R^n) $$
      provided $m<-s$.
\end{lem}
\begin{proof}
    We start by taking a dyadic partition of unity
    \begin{align*}
         \chi_{0}(\xi) +\sum_{j=1}^{+\infty}\chi_{j}(\xi)=1,
     \end{align*}
     where $\chi_{0} \in \D(B(0,2))$, $\chi_{j}(\xi) =\chi(2^{-j}\xi)$ when $j \geq 1$ with $\chi \in \D(A)$ and we decompose the operator $T$ as
     \begin{align}
     \label{Intro:Tdecomp}
         T = T  \chi_{0}(D) + \sum_{j=1}^{+\infty} T \chi_{j}(D).
     \end{align}
     The first term in \eqref{Intro:Tdecomp} is bounded from $L^p_{w}$ to itself by assumption. After a change of variables, we have
     \begin{align*}
         T \chi_{j}(D)u = (2\pi)^{-n} 2^{jn} \int_{\R^n} e^{i 2^{j}\phi(x,\xi)} \chi(\xi) a(x,2^{j}\xi) \widehat{u}(2^{j}\xi) \, \dd\xi
     \end{align*}
     therefore using the semi-classical estimate with $h=2^{-j}$ we obtain
           $$ \Vert T \chi_{j}(D)u\Vert_{L_{w}^p}\leq C_{1} 2^{(m+s)j} \Vert u\Vert_{L_{w}^p}. $$
     This finally gives
     \begin{align*}
           \Vert Tu\Vert_{L_{w}^p} \leq C_{0}  \Vert u\Vert_{L_{w}^p} + C_{1} \sum_{j=1}^{+\infty}  2^{(m+s)j} \Vert u\Vert_{L_{w}^p}
     \end{align*}
     since the series is convergent when $m<-s$. This completes the proof of our lemma.
\end{proof}
\subsubsection{Seeger-Sogge-Stein decomposition}\label{SSS decomposition}
To get useful estimates for the symbol and the phase function, one imposes a second microlocalization on the former
semi-classical operator in such a way that the annulus $A$ is partitioned into truncated cones of thickness roughly $\sqrt{h}$.
Roughly $h^{-(n-1)/2}$ such pieces are needed to cover the annulus $A$. For each $h \in (0,1]$ we fix a collection of unit
vectors $\{\xi^{\nu}\}_{1 \leq \nu \leq J}$ which satisfy:
\begin{enumerate}
     \item $\vert \xi^{\nu}-\xi^{\mu}\vert \geq h^{-\frac{1}{2}},$ if $\nu \neq \mu$, \\
     \item If $\xi \in \mathbf{S}^{n-1}$, then there exists a $\xi^{\nu}$ so that $\vert \xi -\xi^{\nu}\vert  \leq h^{\frac{1}{2}}$. \\
\end{enumerate}
Let $\Gamma^{\nu}$ denote the cone in the $\xi$ space with aperture $\sqrt{h}$ whose central direction is $\xi^{\nu}$, i.e.
\begin{equation}
     \Gamma^{\nu}=\Big\{ \xi \in \R^n;\, \Big\vert \frac{\xi}{\vert\xi\vert}-\xi^{\nu} \Big\vert \leq \sqrt{h} \Big\}.
\end{equation}
One can construct an associated partition of unity given by functions $\psi^{\nu}$, each homogeneous of degree $0$ in $\xi$
and supported in $\Gamma^{\nu}$ with
\begin{equation}
     \sum_{\nu=1}^J \psi^{\nu}(\xi)=1,\quad \text{ for all}\, \xi \neq 0
\end{equation}
and
\begin{align}
\label{Linfty:PsiBounds}
    \sup_{\xi \in \R^n}|\d^{\alpha} \psi^{\nu}(\xi)| \leq C_{\alpha} h^{-\frac{|\alpha|}{2}}.
\end{align}
We decompose the operator $T_{h}$ as
\begin{align}
     T_{h} =  \sum_{\nu=1}^J T_{h}\psi^{\nu}(D) = \sum_{\nu=1}^J T_{h}^{\nu}
\end{align}
where the kernel of the operator $T_h^{\nu}$ is given by
\begin{align}\label{kernel of Thnu}
     T_h^{\nu}(x,y)&= (2\pi h)^{-n} \int_{\R^n} e^{\frac{i}{h} \phi(x,\xi)-\frac{i}{h}\langle y,\xi \rangle} \chi(\xi)\psi^{\nu}(\xi) a(x,\xi/h) \, \dd\xi \\ \nonumber
     &= (2\pi h)^{-n} \int_{\R^n} e^{\frac{i}{h}\langle \nabla_{\xi}\phi(x,\xi^{\nu})-y,\xi \rangle} b^{\nu}(x,\xi,h) \, \dd\xi
\end{align}
with amplitude $b^{\nu}(x,\xi,h)=e^{\frac{i}{h}\langle \nabla_{\xi}\phi(x,\xi)-\nabla_{\xi}\phi(x,\xi^{\nu}),\xi \rangle} \chi(\xi)\psi^{\nu}(\xi) a(x,\xi/h)$.
We choose our coordinates on $\R^n=\R \xi^{\nu}\oplus{\xi^{\nu}}^{\perp}$ in the following way
   $$ \xi = \xi_{1} \xi^{\nu}+\xi', \quad \xi' \perp \xi_{\nu}. $$
Also it is worth noticing that the symbol $\chi(\xi) a(x,\xi/h)$ satisfies the following bound
\begin{equation}
\label{Linfty:Symbolsc}
     \sup_{\xi}\Vert\partial_{\xi}^{\alpha}\big(\chi(\xi)\, a(\cdot,\xi /h)\big)\Vert_{L^{\infty}} \leq C_{\alpha} h^{-m-|\alpha|(1-\varrho)}.
\end{equation}

\begin{lem}
\label{Linfty:bLemma}
     Let $a \in L^{\infty}S^m_{\varrho}$ and $\varphi(x,\xi)\in L^{\infty} \Phi^2.$
   Then the symbol
     \begin{align*}
          b^{\nu}(x,\xi,h)=e^{\frac{i}{h}\langle \nabla_{\xi}\phi(x,\xi)-\nabla_{\xi}\phi(x,\xi^{\nu}),\xi \rangle}\psi^{\nu}(\xi)  \chi(\xi) a(x,\xi/h)
     \end{align*}
     satisfies the estimates
     \begin{align*}
         \sup_{\xi} \big\Vert\d_{\xi}^{\alpha}b^{\nu}(\cdot,\xi,h)\big\Vert_{L^{\infty}} \leq C_{\alpha} h^{-m-|\alpha|(1-\varrho)-\frac{|\alpha'|}{2}}.
     \end{align*}
\end{lem}
\begin{proof}
    We first observe that the bounds \eqref{Linfty:PsiBounds} may be improved to
    \begin{align}
    \label{Linfty:SymbolPsi}
         \sup_{\xi \in A} \big|\d_{\xi}^{\alpha}\psi^{\nu}(\xi)\big| \leq C_{\alpha} h^{-\frac{|\alpha'|}{2}}.
     \end{align}
     This can be seen by induction on $|\alpha|$; by Euler's identity, we have
     \begin{align*}
         \d_{\xi_{1}}\d^{\alpha}_{\xi}\psi^{\nu} = - \langle |\xi|^{-1}\xi-\xi^{\nu},\nabla \d^{\alpha}_{\xi}\psi^{\nu} \rangle + |\alpha| \d^{\alpha}_{\xi}\psi^{\nu}
     \end{align*}
     from which we deduce
     \begin{align*}
         |\d_{\xi_{1}}\d^{\alpha}_{\xi}\psi^{\nu}| &\leq \Big\vert \frac{\xi}{\vert\xi\vert}-\xi^{\nu} \Big\vert \, |\nabla \d^{\alpha}_{\xi}\psi^{\nu}| +|\alpha| |\d^{\alpha}_{\xi}\psi^{\nu}|
         \\ &\lesssim h^{\frac{1}{2}} h^{-\frac{1+|\alpha'|}{2}}+h^{-\frac{|\alpha'|}{2}}.
     \end{align*}
     This ends the induction.
     Similarly we have
     \begin{align}
     \label{Linfty:SymbolPhase}
         \sup_{\xi \in A \cap \Gamma^{\nu}} \big\Vert\d_{\xi}^{\alpha}\big(e^{\frac{i}{h}\langle \nabla_{\xi}\phi(\cdot,\xi)-\nabla_{\xi}\phi(\cdot,\xi^{\nu}),\xi \rangle}\big)\big\Vert_{L^{\infty}}
         \lesssim h^{-\frac{|\alpha'|}{2}}.
     \end{align}
     To prove this bound, we proceed by induction on $|\alpha|$, we have
     \begin{multline*}
          \nabla_{\xi}\d_{\xi}^{\alpha}\Big(e^{\frac{i}{h}\langle \nabla_{\xi}\phi(x,\xi)-\nabla_{\xi}\phi(x,\xi^{\nu}),\xi \rangle}\Big)= \\
          \frac{i}{h}\d_{\xi}^{\alpha}\Big( \big(\nabla_{\xi}\phi(x,\xi)-\nabla_{\xi}\phi(x,\xi^{\nu})\big)
          e^{\frac{i}{h}\langle \nabla_{\xi}\phi(x,\xi)-\nabla_{\xi}\phi(x,\xi^{\nu}),\xi \rangle}\Big)
     \end{multline*}
     and by the Leibniz rule, it suffices to verify that for $|\beta| \leq 1$
     \begin{align*}
         &\sup_{\xi \in A \cap \Gamma^{\nu}} \big\Vert\d_{\xi}^{\beta}\big( \d_{\xi'}\phi(\cdot,\xi)-\d_{\xi'}\phi(\cdot,\xi^{\nu})\big)\big\Vert_{L^{\infty}}
         \lesssim h^{\frac{1-|\beta'|}{2}} \\
         &\sup_{\xi \in A \cap \Gamma^{\nu}} \big\Vert\d_{\xi}^{\beta}\big( \d_{\xi_{1}}\phi(\cdot,\xi)-\d_{\xi_{1}}\phi(\cdot,\xi^{\nu})\big)\big\Vert_{L^{\infty}}
         \lesssim  h^{1-\frac{|\beta'|}{2}},
     \end{align*}
     where for the case $\beta=0$ one simply uses the mean value theorem on $\nabla_{\xi}\phi(x,\xi)-\nabla_{\xi}\phi(x,\xi^{\nu})$, which due the condition $\varphi\in L^{\infty}\Phi^2$
     yields the desired estimates.
     We note that a homogeneous function which vanishes at $\xi=\xi_{\nu}$ may be written in the form
         $$ \Big(\frac{\xi}{|\xi|}-\xi_{\nu}\Big)r(x,\xi)=\mathcal{O}(\sqrt{h}) \quad \textrm{ on } A \cap \Gamma^{\nu} $$
     and this gives the first bound for $\beta_{1} \neq 1$. We also have $\d_{\xi_{1}}\d_{\xi}\phi(x,\xi_{\nu})=0$ by Euler's identity, therefore the former remark yields $\d_{\xi_{1}}\d_{\xi}\phi(x,\xi)=\mathcal{O}(\sqrt{h})$ which is the first bound for $\beta_{1}=1$ (as well as the second bound for $\beta'\neq 0$).
     It remains to prove the second bound for $\beta'=0$: by the mean value theorem and the bounds we have already obtained
         $$ |\d_{\xi_{1}}\phi(x,\xi)-\d_{\xi_{1}}\phi(x,\xi^{\nu})| \lesssim \sqrt{h} \Big|\frac{\xi}{|\xi|}-\xi_{\nu}\Big| \lesssim h.  $$
     The estimates on $b_{\nu}$ are consequences of \eqref{Linfty:Symbolsc}, \eqref{Linfty:SymbolPsi} and \eqref{Linfty:SymbolPhase} and of Leibniz's rule.
\end{proof}
\subsubsection{Phase reduction}\label{phase reduction}
In our definition of class $L^\infty \Phi^{k}$ we have only required control of those frequency derivatives of the phase function which are greater or equal to $k$. This restriction is motivated by the simple model case phase function $\varphi(x,\xi)=t(x)|\xi|+ \langle x,\xi\rangle$, $t(x)\in L^\infty$, for which the first order $\xi$-derivatives of the phase are not bounded but all the derivatives of order equal or higher than 2 are indeed bounded and so $\varphi(x,\xi)\in L^\infty \Phi^2$. However in order to deal with low frequency portions of Fourier integral operators one also needs to control the first order $\xi$ derivatives of the phase. The following phase reduction lemma will reduce the phase of the Fourier integral operators to a linear term plus a phase for which the first order frequency derivatives are bounded.
\begin{lem}
Any Fourier integral operator $T$ of the type \eqref{Intro:Fourier integral operator} with amplitude $\sigma(x,\xi)\in L^{\infty}S^{m}_{\varrho}$ and phase function $\varphi(x,\xi)\in L^{\infty}\Phi^2$, can be written as a finite sum of operators of the form
\begin{equation}\label{reduced rep of Fourier integral operator}
  \frac{1}{(2\pi)^{n}} \int a(x,\xi)\, e^{i\theta(x,\xi)+i\langle \nabla_{\xi}\varphi(x,\zeta),\xi\rangle}\, \widehat{u}(\xi) \, \dd\xi
\end{equation}
where $\zeta$ is a point on the unit sphere $\mathbf{S}^{n-1}$, $\theta(x,\xi)\in L^{\infty}\Phi^{1},$ and $a(x,\xi) \in L^{\infty} S^{m}_{\varrho}$ is localized in the $\xi$ variable around the point $\zeta$.
\end{lem}
\begin{proof}
We start by localizing the amplitude in the $\xi$ variable by introducing an open convex covering $\{U_l\}_{l=1}^{M},$ with maximum of diameters $d$, of the unit sphere $\mathbf{S}^{n-1}$.
Let $\Xi_{l}$ be a smooth partition of unity subordinate to the covering $U_l$ and set $a_{l}(x,\xi)=\sigma(x,\xi)\, \Xi_{l}(\frac{\xi}{|\xi|}).$ We set
\begin{equation}
  T_{l}u(x):= \frac{1}{(2\pi)^n} \int \, a_l(x,\xi)\, e^{i\varphi(x,\xi)}\,\widehat{u}(\xi) \, \dd\xi,
\end{equation}
and fix a point $\zeta \in U_l.$  Then for any $\xi\in U_l$, Taylor's formula and Euler's homogeneity formula yield
\begin{align*}
  \varphi(x,\xi) &= \varphi(x,\zeta) + \langle \nabla_{\xi}\varphi (x,\zeta), \xi-\zeta\rangle +\theta (x,\xi) \\
  &= \theta(x,\xi)+\langle \nabla_{\xi}\varphi(x,\zeta),\xi\rangle
\end{align*}
Furthermore, for $\xi\in U_l$, $\partial_{\xi_k} \theta(x,\xi)= \partial_{\xi_k} \varphi(x,\frac{\xi}{|\xi|})-\partial_{\xi_k} \varphi(x,\zeta)$, so the mean value theorem and the definition of class $L^{\infty}\Phi^2$ yield $|\partial_{\xi_k} \theta(x,\xi)|\leq Cd$ and for $|\alpha|\geq 2$, $|\partial^{\alpha}_{\xi} \theta(x,\xi)|\leq C |\xi|^{1-|\alpha|}.$ Here we remark in passing that in dealing with function $\theta(x,\xi),$ we only needed to control the second and higher order $\xi-$derivatives of the phase function $\varphi(x,\xi)$ and this gives a further motivation for the definition of the class $L^\infty \Phi^2 .$ We shall now extend the function $\theta(x,\xi)$ to the whole of $\mathbf{R}^{n}\times \mathbf{R}^{n}\setminus 0$, preserving its properties and we denote this extension by $\theta(x,\xi)$ again. Hence the Fourier integral operators $T_l$ defined by
\begin{equation}
  T_{l}u(x):=\frac{1}{(2\pi)^{n}} \int a_l(x,\xi)\,e^{i\theta(x,\xi)+i\langle \nabla_{\xi}\varphi(x,\zeta),\xi\rangle}\, \widehat{u}(\xi) \, \dd\xi,
\end{equation}
are the localized pieces of the original Fourier integral operator $T$ and therefore $T=\sum_{l=1}^{M} T_l$ as claimed.
\end{proof}
\subsubsection{Necessary and sufficient conditions for the non-degeneracy of smooth phase functions} \label{phase nondegeneracy}
The smoothness of phases of Fourier integral operators makes the study of boundedness considerably easier in the sense that the conditions of a phase being strongly non-degenerate and belonging to the class $\Phi^{2}$ are enough to secure $L^p$ boundedness for a wide range of rough amplitudes. The following proposition which is useful in proving global $L^2$ boundedness of Fourier integral operators, establishes a relationship between the strongly non-degenerate phases and the lower bound estimates for the gradient of the phases in question.
\begin{prop}\label{equivalence of lowerbound and non degeneracy}
Let $\varphi (x,\xi)\in C^{\infty}(\R^n \times\R^n \setminus 0)$ be a real valued phase then following statements hold true:
\begin{enumerate}
  \item Assume that
              $$ \Big|\det \frac{\partial^{2}\varphi(x,\xi)}{\partial x_j \partial \xi_k} \Big| \geq C_1, $$
           for all $(x,\,\xi) \in \R^n \times\R^n \setminus 0,$ and that
              $$ \Big\Vert\frac{\partial^{2}\varphi(x,\xi)}{\partial x \partial \xi}\Big\Vert\leq C_2, $$
         for all $(x,\,\xi) \in \R^n \times\R^n \setminus 0$ and some constant $C_2>0,$ where $\Vert \cdot\Vert$ denotes matrix norm. Then
           \begin{equation}\label{lowerbound on gradient}
    |\nabla_{\xi} \varphi(x,\xi)- \nabla_{\xi}\varphi(y,\xi)|\geq C |x-y|,
    \end{equation}
           for $x,y\in\mathbf{R}^{n}$ and $\xi\in \mathbf{R}^{n}\setminus 0$ and some $C>0.$
  \item Assume that $|\nabla_{\xi} \varphi(x,\xi)- \nabla_{\xi}(y,\xi)|\geq C |x-y|$ for $x,y\in\mathbf{R}^{n}$ and $\xi\in \mathbf{R}^{n}\setminus 0$ and some $C>0.$
          Then there exists a constant $C_1$ such that
              $$ \Big|\det \frac{\partial^{2}\varphi(x,\xi)}{\partial x_j \partial \xi_k}\Big| \geq C_1, $$
          for all $(x,\,\xi) \in \R^n \times\R^n \setminus 0$.
\end{enumerate}
\end{prop}
\begin{proof}
\setenumerate[0]{leftmargin=0pt,itemindent=20pt}
\begin{enumerate}
\item We consider the map $\mathbf{R}^{n}\ni x \to \nabla_{\xi}\varphi(x,\xi) \in\mathbf{R}^{n}$ and using our assumptions on $\varphi$, Schwartz's global inverse function theorem \cite{Sch} yields that this map is a global $C^1$-diffeomorphism whose inverse $\lambda_{\xi}$ satisfies
\begin{equation}
  |\lambda_{\xi}(z)-\lambda_{\xi}(w)|\leq \sup_{[z,w]} \Vert\lambda'_{\xi}\Vert \times |z-w|.
\end{equation}
Furthermore, $\lambda'_{\xi}(z)= [(\lambda^{-1}_{\xi})^{'}]^{-1}\circ \lambda_{\xi}(z)=[ \partial^{2}_{x,\xi} \varphi(\lambda_{\xi}(z) , \xi)]^{-1} $. Therefore using the wellknown matrix inequality $\Vert A^{-1}\Vert \leq c_{n} |\det A|^{-1} \Vert A\Vert^{n-1}$ which is valid for all $A\in \mathrm{GL}\,(n,\mathbf{R})$, we obtain using the assumption $\Vert\frac{\partial^{2}\varphi(x,\xi)}{\partial x \partial \xi}\Vert\leq C_2 $ that
\begin{equation*}
  \Vert \lambda'_{\xi}(z)\Vert \leq c_n |\det [\partial^{2}_{x,\xi} \varphi(\lambda_{\xi}(z), \xi)]|^{-1} \Vert \partial^{2}_{x,\xi} \varphi(\lambda_{\xi}(z), \xi)\Vert^{n-1}\leq \frac{c_n}{C_1} C_2\leq \frac{1}{C}.
\end{equation*}
This yields that $ |\lambda_{\xi}(z)-\lambda_{\xi}(w)|\leq C|z-w|$ and setting $z=\nabla_{\xi} \varphi(x,\xi)$ and $w=\nabla_{\xi} \varphi(y,\xi)$, we obtain \eqref{lowerbound on gradient}.
\item Given the lower bound on the difference of the gradients as in the statement of the second part of the proposition, setting $y=x+hv$ with $v\in \mathbf{R}^n$ yields,
$$\frac{|\nabla_{\xi} \varphi(x+hv,\xi)- \nabla_{\xi}\phi(x,\xi)|}{h}\geq C|v|$$
and letting $h$ tend to zero we have for any $v\in \mathbf{R}^n$
\begin{equation}\label{invertibility of hessian}
  |\partial^2_{x,\xi}\varphi(x,\xi)\cdot v|\geq C |v|.
\end{equation}
This means that $\partial^2_{x,\xi}\varphi(x,\xi)$ is invertible and $|[\partial^2_{x,\xi}\varphi(x,\xi)]^{-1}\cdot w|\leq \frac{|w|}{C}.$ Therefore, taking the supremum we obtain $\frac{1}{\Vert [\partial^2_{x,\xi}\varphi(x,\xi)]^{-1}\Vert^n} \geq \frac{1}{C^n}.$ Now using the wellknown matrix inequality
\begin{equation}\label{variant of hadamard inequality}
  \frac{1}{\gamma_{n} \Vert A^{-1}\Vert^{n}} \leq |\det A|\leq \gamma_{n} \Vert A\Vert^n,
\end{equation}
which is a consequence of the Hadamard inequality, yields for $A=\frac{\partial^{2}\varphi(x,\xi)}{\partial x_j \partial \xi_k}$
\begin{equation}
\bigg|\det\frac{\partial^{2}\varphi(x,\xi)}{\partial x_j \partial \xi_k}\bigg|\geq \frac{1}{\gamma_{n}C^n}.
\end{equation}
\end{enumerate}
This completes the proof.
\end{proof}
\begin{rem}
Proposition \ref{equivalence of lowerbound and non degeneracy} gives a motivation for our rough non-degeneracy condition in Definition \ref{defn of rough phases}, when there is no differentiability in the spatial variables.
\end{rem}
\subsubsection{Necessity of strong non-degeneracy for global regularity}\label{necessity of strong non-degeneracy}
We shall now discuss a simple example which illustrates the necessity of the strong non-degeneracy condition for the validity o global $L^p$ boundedness of Fourier integral operators. To this end, we take a smooth diffeomorphism $\kappa:\mathbf{R}^{n} \to \mathbf{R}^{n}$ with everywhere nonzero Jacobian determinant, i.e. $\det \kappa'(x)\neq 0$ for all $x\in \mathbf{R}^{n}.$ Now, if we let $\varphi(x,\xi)= \langle \kappa(x),\xi\rangle$ and take $a(x,\xi)=1 \in S^{0}_{1,0},$ then the Fourier integral operator $T_{a,\varphi} u(x)$ is nothing but the composition operator $u\circ \kappa(x).$ Therefore
\begin{equation}\label{necessity counterexample Lp bound}
\Vert T_{a, \varphi}u\Vert_{L^p} = \Vert u\circ \kappa\Vert_{L^p}=\{\int_{\mathbf{R}^{n}} |u(y)|^{p} \, |\det\kappa'^{-1}(\kappa^{-1} (y))|\, dy\}^{\frac{1}{p}},
\end{equation}
from which we see that $T_{a,\varphi}$ is $L^p$ bounded for any $p$, if and only if there exists a constant $C>0$ such that $|\det \kappa'^{-1}(x)|\leq C$ for all $x\in\mathbf{R}^{n}.$ The latter is equivalent to $|\det \kappa'(x)|\geq \frac{1}{C} >0.$ Now since $|\det \frac{\partial^{2}\varphi(x,\xi)}{\partial x \partial \xi}|=|\det \kappa'(x)|$ it follows at once that a necessary condition for the $L^p$ boundedness of the operator $T_{a, \varphi}$ is the strong non-degeneracy of the phase function $\varphi$. We observe that if we instead had chosen $a(x,\xi)$ to be equal to a smooth compactly supported function in $x,$ then the $L^p$ boundedness of $T_{a, \varphi}$ would have followed from the mere non-degeneracy condition $|\det \frac{\partial^{2}\varphi(x,\xi)}{\partial x \partial \xi}|=|\det \kappa'(x)|\neq 0.$
\subsubsection{Non smooth changes of variables}\label{nonsmooth change of variables}
In dealing with rough Fourier integral operators we would need at some point to make changes of variables when the substitution is not differentiable. This issue is problematic in general but in our setting, thanks to the rough non-degeneracy assumption on the phase, we can show that the substitution is indeed valid and furthermore has the desired boundedness properties. The discussion below is an abstract approach to the problem of non smooth substitution and we refer the reader interested in related substitution results to
De Guzman~\cite{Guz}.
\begin{lem}
\label{Lem:Substitution}
     Let $U$ be a measurable set and let $t: U \to \R^n$ be a bounded measurable map satisfying
     \begin{align}
     \label{InjectivityHyp}
          |t(x)-t(y)| \geq c |x-y|
     \end{align}
     for almost every $x,y \in U$. Then there exists a function $J_{t} \in L^{\infty}(\R^n)$ supported in $t(U)$
     such that the substitution formula
     \begin{align}\label{rough substitution formula}
          \int_{U} u\circ t(x) \, \dd x = \int u(z) J_{t}(z) \, \dd z
     \end{align}
     holds for all $u \in L^1(\R^n)$ and the Jacobian $J_t$ satisfies the estimate $\Vert J_t\Vert_{L^{\infty}} \leq \frac{2 \sqrt{n}}{c}.$
\end{lem}
\begin{rem}
\label{BijectivityRem}
    If one works with a representative $t$ in the equivalence class of functions equal almost everywhere, then possibly after replacing $U$
    with $U \setminus N$ (where $N$ is a null-set where \eqref{InjectivityHyp} does not hold), one may assume that $t$ is an injective map
    with \eqref{InjectivityHyp} holding everywhere on $U$.
\end{rem}
For the convenience of the reader, we provide a proof of this simple lemma.
\begin{proof}
     As observed in Remark \ref{BijectivityRem}, we may assume that $t$ is an injective map from $U$ to $\R^n$ for which  \eqref{InjectivityHyp}
     holds on $U$. The formula
     \begin{align*}
          \mu_{t}(f) = \int_{U} f \circ t(x) \, \dd x, \quad f \in C^0_{0}(\R^n)
     \end{align*}
     defines a non-negative Radon measure, which by the Riesz representation theorem is associated to a Borel measure. In this case, the latter measure is explicitly given by
     \begin{align*}
          \mu_{t}(A) = |t^{-1}(A) \cap U |
     \end{align*}
     on all Lebesgue measurable sets $A \subset \R^n$, where we use the notation $|\cdot|$ for the Lebesgue measure of a set. By the Lebesgue
     decomposition theorem, this measure can be split into an absolutely continuous and a singular part, i.e.
          $$ \mu_{t} = \mu_{t}^{\rm ac} + \mu_{t}^{\rm sing}.$$
Now assumption \eqref{InjectivityHyp} yields
     \begin{align*}
          t^{-1}\big(B_{\infty}(w,r)\big) \subset B_{\infty}(x,2\sqrt{n}r/c), \quad \textrm{ if } t(x) \in B_{\infty}(w,r)
     \end{align*}
     where $B_{\infty}(w,r)$ is a ball of center $w$ and radius $r$ for the supremum norm. This implies that whenever
          $$ A \cap t(U) \subset \bigcup_{k=0}^{\infty} B_{\infty}(w_{k},r_{k}) $$
     it follows that
          $$ t^{-1}(A) \cap U \subset  \bigcup_{k=0}^{\infty} B_{\infty}(x_{k},2\sqrt{n}r_{k}/c) $$
     where the centers $x_{k}$ have been chosen in $t^{-1}(B_{\infty}(w_{k},r_{k}))$ when this set is nonempty.
     Furthermore, it is wellknown that the Lebesgue measure of a set can be computed using
     \begin{align*}
          |\Omega|=\inf \bigg\{ \sum_{k=0}^{\infty} |Q_{k}|, \; \Omega \subset \bigcup_{k=0}^{\infty} Q_{k}\bigg\}
     \end{align*}
     where the infimum is taken over all possible sequences $(Q_{k})_{k \in \N}$ of cubes with faces parallel to the axes.
    Therefore
     \begin{align}
     \label{MeasureBound}
          \mu_{t}(A) \leq \frac{2 \sqrt{n}}{c}  |A \cap t(U)| \leq \frac{2 \sqrt{n}}{c} |A|
     \end{align}
     for all Lebesgue measurable sets $A$ in $\R^n$. In particular, Lebesgue null-sets are also null-sets with respect to $\mu_{t},$ which in turn implies that the measure $\mu_{t}$ is absolutely continuous with respect to the Lebesgue measure. By the Radon-Nikodym theorem, there exists
     a positive Lebesgue measurable function $J_{t} \in L^1_{\rm loc}$ such that $\mu_{t}$ has density $J_{t}$
          $$ \mu_{t}(A) = \int_{A} J_{t}(x) \, \dd x. $$
     By Lebesgue's differentiation theorem, we may compute the Jacobian function $J_{t}$ from the measure $\mu_{t}$ by a limiting
     process on balls $B$, namely
     \begin{align}\label{the jacobian of mu}
          J_{t}(x) = \lim_{B \to \{x\}} \frac{1}{|B|} \int_{B} J_{t}(y) \, \dd y = \lim_{B \to \{x\}} \frac{\mu_{t}(B)}{|B|}.
     \end{align}
     Equality \eqref{the jacobian of mu} together with the estimate \eqref{MeasureBound} yields that $J_{t}$ is bounded and
     \begin{align*}
         \Vert J_{t}\Vert_{L^{\infty}} \leq \frac{2 \sqrt{n}}{c}.
     \end{align*}
Moreover, from the definition of $\mu_{t}$ it is clear that it is supported in $t(U)$.
     Finally, \eqref{rough substitution formula} follows from
     \begin{align*}
          \int_{U} u \circ t(x) \, \dd x = \mu_{t}(u) = \int u \, \dd \mu_{t} = \int u(z) J_{t}(z) \, \dd z
     \end{align*}
     for all $u \in C^0_{0}(U)$, and this extends to functions $u \in L^1(\R^n)$.
\end{proof}
\begin{rem}
     Note that if there is a representative $t$ in the equivalence class such that \eqref{InjectivityHyp} holds everywhere on $U$ and such that
     $t(U)$ is an open subset of $\R^n$, then $t^{-1}: t(U) \to U$ is a Lipschitz bijection. Furthermore, any open subset $V \subset t(U)$ is open
     in $\R^n$ and by Brouwer's theorem on the invariance of the domain $t^{-1}(W)$ is open. This means that the map $t$ is actually continuous.
\end{rem}
\begin{cor}\label{cor:main substitution estim}
     Let $t:\R^n \to \R^n$ be a map satisfying the assumptions in $\mathrm{Lemma\, \ref{Lem:Substitution}}$ with $U=\R^n$, then $u \mapsto u \circ t$ is a bounded
     map on $L^p$ for $p\in [1,\infty]$.
\end{cor}
\begin{proof}
     This easily follows from Lemma \ref{Lem:Substitution}:
     \begin{align*}
          \int |u\circ t(x)|^p \, \dd x = \int |u(z)|^p J_{t}(z) \, \dd z \leq \Vert J_{t}\Vert_{L^{\infty}} \Vert u\Vert_{L^p}
     \end{align*}
     when $p\in [1,\infty)$. The $L^{\infty}$ estimate is similar.
\end{proof}
\subsubsection{$L^p$ boundedness of the low frequency portion of rough Fourier integral operators}\label{small frequency Lp boundedness}
Here we will prove the $L^p$ boundedness for $p\in [1,\infty]$ of Fourier integral operators whose amplitude contains a smooth compactly supported function factor, the support of which lies in a neighbourhood of the origin. There are a couple of difficulties to overcome here, the first being the singularity of the phase function in the frequency variable at the origin. The second problem is the one caused by the lack of smoothness in the spatial variables. In order to handle these problems we need the following lemma
\begin{lem}\label{main low frequency estim}
     Let $b(x,\xi)$ be a bounded function which is $C^{n+1}(\R^n_{\xi} \setminus 0)$ and compactly supported in the frequency variable $\xi$ and $L^{\infty}(\R^n_{x})$
     in the space variable $x$ satisfying
     \begin{align*}
          \sup_{\xi \in \R^n \setminus 0} |\xi|^{-1+|\alpha|} \Vert\d^{\alpha}_{\xi}b(\cdot\,,\xi)\Vert_{L^{\infty}} < +\infty, \quad |\alpha| \leq n+1.
     \end{align*}
     Then for all $0 \leq \mu<1$ we have
     \begin{align}
     \label{LowFreq:KernelEst1}
          \sup_{x,y \in \R^{2n}} \langle y \rangle^{n+\mu} \Big| \int e^{-i \langle y,\xi \rangle} b(x,\xi) \, \dd \xi \Big| < +\infty.
     \end{align}
\end{lem}
\begin{proof}
     Since $b(x,\xi)$ is assumed to be bounded, the integral in \eqref{LowFreq:KernelEst1} which we denote by $B(x,y),$ is uniformly bounded and therefore it suffices to consider the case $|y| \geq 1.$
     Integrations by parts yield
     \begin{align*}
          B(x,y)=|y|^{-2n}  \int e^{-i \langle y,\xi \rangle} \langle y,D_{\xi} \rangle^n b(x,\xi) \, \dd \xi
     \end{align*}
    and therefore we have the estimate
      \begin{align*}
          |B(x,y)| \leq C |y|^{-n}  \int_{|\xi|<M}  \frac{\dd \xi}{|\xi|^{n-1}}.
     \end{align*}
     We would like to gain an extra factor of $|y|^{-\mu};$ to this end consider the function $\beta(x,y,\xi)= |y|^{-n}\langle y,D_{\xi} \rangle^n b(x,\xi)$
     which is smooth in $\xi$ on $\R^n \setminus 0$ and satisfies
     \begin{align*}
          \sup_{\xi \in \R^n \setminus 0} |\xi|^{n+|\alpha|-1} \Vert\d_{\xi}^{\alpha}\beta(\cdot,\cdot,\xi)\Vert_{L^{\infty}} < +\infty, \quad |\alpha| \leq 1.
    \end{align*}
     Let $\chi$ be a $C^{\infty}_{0}(\R^n)$ function which is one on the unit ball and zero outside the ball of radius $2$. Taking $0<\eps \leq 1$, we have
     \begin{align*}
          |y|^nB(x,y)&=  \int e^{-i \langle y,\xi \rangle} \chi(\xi/\eps) \beta(x,y,\xi) \, \dd \xi \\ &\quad + \int e^{-i \langle y,\xi \rangle} \big(1-\chi(\xi/\eps)\big) \beta(x,y,\xi) \, \dd \xi
     \end{align*}
     the first term is bounded by a constant times $\eps$, while the second term is equal to
     \begin{align*}
          i |y|^{-2} \int e^{-i \langle y,\xi \rangle} \big(\eps^{-1}\langle y,\d_{\xi}\rangle \chi(\xi/\eps) \beta - \big(1-\chi(\xi/\eps)\big) \langle y,\d_{\xi}\rangle \beta \big) \, \dd \xi
     \end{align*}
     which may be bounded by
      \begin{align*}
          |y|^{-1} (C_{1}-C_{2} \log \eps).
     \end{align*}
     We minimize the bound $C_{0}\eps+ |y|^{-1} (C_{1}-C_{2} \log \eps)$ by taking $\eps=|y|^{-1}$, and obtain
     \begin{align*}
          |B(x,y)| \leq C |y|^{-n-1} \big(1+\log |y|\big) \leq C' |y|^{-n-\mu}
     \end{align*}
     for all $0 \leq \mu<1$. This is the desired estimate.
\end{proof}
Having this in our disposal we can show that the low frequency portion of the Fourier integral operators are $L^p$ bounded, more precisely we have
\begin{thm}\label{general low frequency boundedness for rough Fourier integral operator}
  Let $a(x,\xi)\in L^{\infty}S^{m}_{\varrho}$ with $m\in \mathbf{R}$ and $\varrho \in [0,1]$ and let the phase function $\varphi(x,\xi)\in L^{\infty}\Phi^2$ satisfy the rough
  non-degeneracy condition $($according to $\mathrm{Definition\, \ref{defn of rough nondegeneracy}}$$).$ Then for all $\chi_{0}(\xi)\in C_{0}^{\infty}$ supported around the origin,
  the Fourier integral operator
     $$T_{0} u(x)= \frac{1}{(2\pi)^{n}}\int_{\mathbf{R}^{n}}e^{i\varphi(x,\xi)}a(x,\xi) \chi_{0}(\xi) \, \widehat{u}(\xi)\, \dd\xi $$
  is bounded on $L^p$ for $p\in [1,\infty]$.
\end{thm}
\begin{proof}
In proving the $L^p$ boundedness, according to the reduction of the phase procedure in Lemma \ref{reduced rep of Fourier integral operator}, there is no loss of
generality to assume that our Fourier integral operator is of the form
     $$T_{0} u(x)=  \frac{1}{(2\pi)^{n}} \int a(x,\xi)\,\chi_{0}(\xi) e^{i\theta(x,\xi)+i\langle \nabla_{\xi}\varphi(x,\zeta),\xi\rangle}\, \widehat{u}(\xi) \, \dd\xi,$$
for some $\zeta\in \mathbf{S}^{n-1}$, $a\in L^{\infty} S^{m}_{\varrho}$ and $\theta \in L^{\infty} \Phi^1.$ In the proof of the $L^p$ boundedness of $T_0$ we only
need to analyze the kernel of the operator
    $$\int a(x,\xi) \chi_{0}(\xi) e^{i\theta(x,\xi)+i\langle \nabla_{\xi}\varphi(x,\zeta),\xi\rangle} \widehat{u}(\xi) \, \dd\xi,$$
which is given by
    $$T_0 (x,y):= \int e^{i\langle \nabla_{\xi}\varphi (x,\zeta)-y,\xi\rangle} \,e^{i\theta (x,\xi)}\, a(x,\xi) \chi_{0}(\xi)\, \dd\xi.$$
Now the estimates on the $\xi$ derivatives of $\theta(x,\xi)$ above, yield $$\sup_{|\xi|\neq 0} |\xi|^{-1+|\alpha|} |\partial_{\xi}^{\alpha} \theta(x, \xi)|<\infty,$$ for $|\alpha|\geq 1$ uniformly in $x$, and therefore setting $b(x,\xi):=a(x,\xi) \chi_{0}(\xi) e^{i\theta(x,\xi)}$ we have that $b(x,\xi)$ is bounded and $\sup_{|\xi|\neq 0} |\xi|^{-1+|\alpha|} |\partial_{\xi}^{\alpha} b(x, \xi)|<\infty,$ for $|\alpha|\geq 1$ uniformly in $x$ and using Lemma \ref{main low frequency estim}, we have for all $\mu\in [0,1)$
\begin{equation*}
  |T_{0} (x,y)|\leq C\langle \nabla_{\xi}\varphi (x,\zeta)-y\rangle ^{-n-\mu}.
\end{equation*}
From this it follows that $$\sup_{x} \int  |T_{0} (x,y)|\, dy<\infty ,$$ and using our rough non-degeneracy assumption and Corollary \ref{cor:main substitution estim} in the case $p=1,$ we also have
$$\int  |T_{0} (x,y)|\, \dd x \lesssim \int \langle \nabla_{\xi}\varphi (x,\zeta)-y\rangle ^{-n-\mu}\, \dd x \lesssim \int \langle z\rangle ^{-n-\mu}\, \dd z<\infty,$$
uniformly in $y$.
This estimate and Young's inequality yield the $L^p$ boundedness of the operator $T_0$.
\end{proof}
\subsection{Some links between nonsmoothness and global boundedness}\label{nonsmoothness and global estimates}
In this paragraph, we illustrate some of the relations between boundedness for rough Fourier integral operators and the global boundedness
of operators with smooth amplitudes and phases. Our observation is that local estimates for non-smooth Fourier integral operators imply global estimates for certain classes of Fourier integral operators. This can be done either by compactification or by using a dyadic decomposition. To see the relation between compactification and global boundedness, consider the operator
\begin{align}
     Tu(x) = (2\pi)^{-n} \int e^{i \phi(x,\xi)} a(x,\xi) \widehat{u}(\xi) \, \dd \xi.
\end{align}
Let $\chi \in \D(B(0,2))$ be equal to one on the unit ball $B(0,1)$, and $\omega = 1-\chi$ be supported away from zero. Then
    $$ T = T_{0} + T_{1}, \quad T_{0}=\chi T, \quad T_{1}=\omega T. $$
For the global continuity of $T$, we are only interested in $T_{1}$ since the amplitude of $T_{0}$ is compactly supported in the space
variable and the boundedness of that operator follows from the local theory. Concerning $T_{1}$, we make the change of variables
\begin{align}
     z = \frac{x}{|x|^{1+\frac{1}{\theta}}}, \quad x = \frac{z}{|z|^{1+\theta}}, \quad \theta \in (0,1]
\end{align}
so that
\begin{align}
     \int |T_{1}u(x)|^p \, \dd x = \theta \int \Big|T_{1}u\Big(\frac{z}{|z|^{1+\theta}}\Big)\Big|^p |z|^{-n(1+\theta)} \, \dd z.
\end{align}
Therefore it suffices to study the $L^p$ boundedness of the Fourier integral operator
\begin{align}
     \tilde{T}_{1}u(z) = (2\pi)^{-n} \int e^{i \phi(z/|z|^{1+\theta},\zeta)} \underbrace{|z|^{-\frac{n}{p}(1+\theta)} \omega a\Big(\frac{z}{|z|^{1+\theta}},\zeta\Big)}_{=\tilde{a}(z,\zeta)}
     \widehat{u}(\zeta) \, \dd \zeta.
\end{align}
The amplitude $\tilde{a}(z,\zeta)$ is compactly supported (in the unit ball), and for a suitable choice of $\theta$ belongs to $L^{\infty}S^m$ provided%
\footnote{This decay assumption is due to the singularity at $0$ of $|z|^{-n(1+\theta)/p}$ of the Jacobian. Note that any improvement on the regularity
of $\tilde{a}, \tilde{\phi}$ should translate into decay properties at infinity of the original amplitude and phase $a,\phi$.}
\begin{align}
\label{DecayAssum}
     \langle x \rangle^{s} a(x,\xi) \in L^{\infty}S^m, \quad s>\frac{n}{p}.
\end{align}
Now suppose that $\phi$ satisfies the following (global) non-degeneracy assumption:
\begin{align}
     |\nabla_{\xi}\phi(x,\xi)-\nabla_{\xi}\phi(y,\xi)| \geq c |x-y|
\end{align}
for all $x,y \in \R^n$. Then since%
\footnote{In the case of the Kelvin transform $\theta=1$, it is easy to get a better lower bound (in fact an equality):
   $$  \Big|\frac{z}{|z|^{2}}-\frac{w}{|w|^{2}} \Big|^2 = |z|^{-2}+|w|^{-2}
    - \frac{2\langle w,z\rangle}{|z|^{2}|w|^{2}}=\frac{|z-w|^2}{|z|^2|w|^2}. $$}
\begin{align}
    \Big|\frac{z}{|z|^{1+\theta}}-\frac{w}{|w|^{1+\theta}} \Big|^2 &= |z|^{-2\theta}+|w|^{-2\theta}
    - \frac{2\langle w,z\rangle}{|z|^{1+\theta}|w|^{1+\theta}}  \\ \nonumber
    &\geq \frac{1}{\max(|w|,|z|)^{1+\theta}} \, |z-w|^2,
\end{align}
the phase $\tilde{\phi}(z,\zeta)=\phi(z/|z|^{1+\theta},\zeta)$ satisfies a similar non-degeneracy condition, namely
\begin{align}
     |\nabla_{\zeta}\tilde{\phi}(z,\zeta)-\nabla_{\zeta}\tilde{\phi}(w,\zeta)|&=\Big|\nabla_{\zeta}\phi\Big(\frac{z}{|z|^{1+\theta}},\zeta\Big)
     -\nabla_{\zeta}\phi\Big(\frac{w}{|w|^{1+\theta}},\zeta\Big)\Big| \\ \nonumber
     & \geq  \frac{c}{\max(|w|,|z|)^{\frac{1+\theta}{2}}} \, |z-w| \geq c |z-w|,
\end{align}
when $|w|,|z| \leq 1$.
In order to improve the decay assumption on the amplitude \eqref{DecayAssum}, one can consider more general changes of variables
which do not affect the angular coordinate in the polar decomposition, i.e. coordinate changes of the form
     $$ z=f(|x|)\frac{x}{|x|} \textrm{ where } f:(0,\infty) \to (0,1) \textrm{ is a diffeomorphism.}  $$
Then $x=g(|z|)z/|z|$ where $g$ is the inverse function of $f,$ and the Jacobian of such a change of variables is given by
     $$ |g'(|z|)|g^{n-1}(|z|)|z|^{1-n}. $$
We would like to choose $g$ in such a way that the singularities of its Jacobian become weaker than those in the case of $g(s)=s^{-\theta}.$
One possible choice is to take
     $$ g(s) = \log(1-s) $$
for which we have
     $$ |g'(s)|g^{n-1}(s)s^{1-n}= \frac{\log^{n-1}(1-s)}{s^{n-1}(1-s)} = \mathcal{O}\big((1-s)^{-1-\theta}\big) $$
if $s \in (0,1)$. For this choice, we need the following decay
\begin{align}
     \langle x \rangle^{s} a(x,\xi) \in L^{\infty}S^m, \quad s>\frac{1}{p}.
\end{align}
Furthermore, if one assumes that $g/s$ is decreasing (or increasing) then the phase $\tilde{\phi}$ satisfies our non-degeneracy assumption, because
\begin{align}
    \Big|\frac{z}{|z|}g(|z|)&-\frac{w}{|w|}g(|w|) \Big|^2 \\ \nonumber &= g(|z|)^{2}+g(|w|)^{2}
    - \frac{2\langle w,z\rangle}{|z||w|}g(|z|)g(|w|)  \\ \nonumber
    &\geq \Big(\frac{g}{s}\Big)\big(\min(|w|,|z|)\big) \, |z-w|^2 \geq g'(0) \, |z-w|^2 .
\end{align}
Alternatively, in order to investigate global boundedness using a dyadic decomposition, one takes a Littlewood-Paley partition of unity $1=\chi(x) + \sum_{j=1}^{\infty}\psi(2^{-j} x)$, which yields
\begin{align}
     T = \chi T + \sum_{j=1}^{\infty} T_{j}, \quad T_j:=\psi(2^{-j}\cdot)T.
\end{align}
Once again we are only interested in $T_{j}$ and following a change of variables, we want to prove
\begin{align}
     \int \big|T_{j}\big(u(2^{-j}\cdot\big)(2^jz) \big|^p \, \dd z \leq C_{p}  \int |u(z)|^p \, \dd z.
\end{align}
This leads us to the study of the operator
\begin{align}
    \tilde{T}_{j}u(z) & = T_{j}\big(u(2^{-j}\cdot\big)(2^j z) \\ \nonumber
    &= (2\pi)^{-n} \int e^{i 2^{-j}\phi(2^j z,\zeta)} \underbrace{\psi(z) a\big(2^j z,2^{-j}\zeta\big)}_{=\tilde{a}_{j}(z,\zeta)}
    \widehat{u}(\zeta) \, \dd \zeta.
\end{align}
The estimate
\begin{align}
     |\d^{\alpha}_{\zeta} \tilde{a}_{j}(z,\zeta)| &\leq \underbrace{2^{-j|\alpha|}(1+2^{j}|z|)^{m}}_{\simeq 2^{j(m-|\alpha|)}}  (1+2^{-j}|\zeta|)^{m-|\alpha|}   \\
     \nonumber &\leq C_{\alpha} \, (1+|\zeta|)^{m-|\alpha|},
\end{align}
yields that the amplitude $\tilde{a}_{j}(z,\zeta)$ belongs (uniformly with respect to $j$) to the class $L^{\infty}S^m$  provided
\begin{align}
     \langle x \rangle^{-m} a(x,\xi) \in L^{\infty}S^m .
\end{align}
The phase $\tilde{\phi}_{j}(z,\zeta)=2^{-j}\phi(2^j z,\zeta)$ satisfies the non-degeneracy assumption
\begin{align}
     |\nabla_{\zeta}\tilde{\phi}_{j}(z,\zeta)-\nabla_{\zeta}\tilde{\phi}_{j}(w,\zeta)| \geq c |z-w|.
\end{align}
Therefore, once again the problem of establishing the global $L^p$ boundedness is reduced to a local problem concerning operators with rough amplitudes.

\section{Global boundedness of Fourier integral operators}
In this chapter, partly motivated by the investigation in \cite{KS} of the $L^{p}$ boundedness of the so called pseudo-pseudodifferential operators where the symbols of the aforementioned operators are only bounded and measurable in the spatial variables $x$, we consider the global and local boundedness in Lebesgue spaces of Fourier integral operators of the form
\begin{align}\label{Fourier integral operator}
   Tu(x)  = (2\pi)^{-n} \int_{\R^n} e^{i\phi(x,\xi)} a(x,\xi) \widehat{u}(\xi) \, \dd\xi,
\end{align}
in case when the phase function $\phi(x,\xi)$ is smooth and homogeneous of degree 1 in the frequency variable $\xi,$ and the amplitude $a(x,\xi)$ is either in some H\"ormander class $S^{m}_{\varrho, \delta},$ or is a $L^{\infty}$ function in the spatial variable $x$ and belongs to some $L^\infty S^{m}_\varrho$ class. We shall also investigate the $L^p$ boundedness problem for Fourier integral operators with rough phases that are $L^{\infty}$ functions in the spatial variable. In the case of the rough phase, the standard notion of non-degeneracy of the phase function has no meaning due to lack of differentiability in the $x$ variables. However, there is a non-smooth analogue of the non-degeneracy condition which has already been introduced in Definition \ref{defn of rough nondegeneracy} which will be exploited further here.\\
We start by investigating the question of $L^1$ boundedness of Fourier integral operators with rough amplitudes but smooth phase functions satisfying the strong non-degeneracy condition. Thereafter we turn to the problem of $L^2$ boundedness of the Fourier integral operators with smooth phases, but rough or smooth amplitudes. In the case of smooth amplitudes, we show the analogue of the Calder\'on-Vaillancourt's $L^2$ boundedness of pseudodifferential operators in the realm of Fourier integral operators. Next, we consider Fourier integral operators with rough amplitudes and rough phase functions and show a global and a local $L^2$ result in that context. We also give a fairly general discussion of the symplectic aspects of the $L^2$ boundedness of Fourier integral operators.\\
After concluding our investigation of the $L^2$ boundedness, we proceed by proving a sharp $L^\infty$ boundedness theorem for Fourier integral operators with rough amplitudes and rough phases in class $L^\infty\Phi^2,$ without any non-degeneracy assumption on the phase.
Finally, we close this chapter by proving $L^p-L^p$ and $L^p-L^q$ estimates for operators with smooth phase function, and smooth or rough amplitudes.
\subsection{Global $L^1$ boundedness of rough Fourier integral operators}
As will be shown below, the global $L^1$ boundedness of Fourier integral operators is a consequence of Theorem \ref{general low frequency boundedness for rough Fourier integral operator}, the Seeger-Sogge-Stein decomposition, and elementary kernel estimates.
\begin{thm}
\label{Intro:L1Thm}
   Let $T$ be a Fourier integral operator given by \eqref{Intro:Fourier integral operator} with amplitude $a \in L^{\infty}S^m_{\varrho}$
   and phase function $\phi \in L^{\infty}\Phi^2$ satisfying the rough non-degeneracy condition. Then there exists a constant $C>0$ such that
      $$ \Vert Tu\Vert_{L^{1}} \leq C \Vert u\Vert_{L^{1}}, \quad u \in \S(\R^n) $$
   provided $m<-\frac{n-1}{2} +n(\varrho-1)$ and $0\leq \varrho\leq 1$.
\end{thm}
\begin{proof}
Using semiclassical reduction of Subsection \ref{Semiclasical reduction subsec}, we decompose $T$ into low and high frequency portions $T_0$ and $T_h$. Then we use the Seeger-Sogge-Stein decomposition of Subsection \ref{SSS decomposition} to decompose $T_h$ into the sum $\sum_{\nu=1}^{J} T_{h}^{\nu}.$ The boundedness of $T_0$ follows at once from Theorem \ref{general low frequency boundedness for rough Fourier integral operator}, so it remains to establish suitable semiclassical estimates for $T_{h}^{\nu}$.
To this end we consider the following differential operator
\begin{align*}
     L = 1-\d_{\xi_{1}}^2-h \d_{\xi'}^2
\end{align*}
for which we have according to Lemma \ref{Linfty:bLemma}
\begin{equation}
     \sup_{\xi} \Vert L^{N}b^{\nu}(\cdot,\xi,h)\Vert_{L^{\infty}} \lesssim  h^{-m-2N(1-\varrho)}.
\end{equation}
Integrations by parts yield
\begin{align*}
     |T_{h}^{\nu}(x,y)| \leq (2\pi h)^{-n} \big(1+g_{1}(y-\nabla_{\xi}\phi(x,\xi^{\nu})\big)^{-N} \int |L^N b^{\nu}(x,\xi,h) | \, \dd \xi
\end{align*}
for all integers $N$, with
\begin{equation}
g(z)=h^{-2}z_{1}^2+h^{-1}|z'|^{2}.
\end{equation}
This further gives
\begin{align*}
     |T_{h}^{\nu}(x,y)| \leq C_{N} h^{-m-\frac{n+1}{2}-2N(1-\varrho)} \big(1+g(y-\nabla_{\xi}\phi(x,\xi^{\nu})\big)^{-N}
\end{align*}
since the volume of the portion of cone $|A \cap \Gamma_{\nu}|$ is of the order of $h^{(n-1)/2}$.
By interpolation, it is easy to obtain the former bound when the integer $N$ is replaced by $M/2$ where $M$ is any given positive number ;
indeed write $M/2=N+\theta$ where $N=[\frac{M}{2}]$ and $\theta \in [0,1)$ and
\begin{align}\label{T(x,y) estim}
     |T^{\nu}_{h}(x,y)| &= |T^{\nu}_{h}(x,y)|^{\theta} |T^{\nu}_{h}(x,y)|^{1-\theta} \nonumber \\
     &\leq C^{1-\theta}_{N} C^{\theta}_{N+1}h^{-m-\frac{n+1}{2}-(1-\varrho)M}  \big(1+g(y-\nabla_{\xi}\phi(x,\xi^{\nu})\big)^{-\frac{M}{2}}.
\end{align}
This implies that for any real number $M>n$
\begin{align*}
     \sup_{x} \int |T_h^{\nu}(x,y)| \, \dd y \leq C_{M} h^{-m-M(1-\varrho)}.
\end{align*}
Furthermore the rough non-degeneracy assumption on the phase function $\varphi(x,\xi)$ and Corollary \ref{cor:main substitution estim} with $p=1$ yield
\begin{align*}
     \sup_{y} \int |T_h^{\nu}(x,y)| \, \dd x &\leq C^{1-\theta}_{N} C^{\theta}_{N+1} \int\big(1+g(\nabla_{\xi}\phi(x,\xi^{\nu})\big)^{-\frac{M}{2}}\, \dd x \\ \nonumber
     &\leq C_{M} h^{-m-M(1-\varrho)}
\end{align*}
thus using Young's inequality and summing in $\nu$
    $$ \Vert T_h u\Vert_{L^{1}} \leq  \sum_{\nu=1}^J \Vert T_h^{\nu} u\Vert_{L^{1}} \leq C_{M} h^{-m-\frac{n-1}{2}-M(1-\varrho)}\Vert u\Vert_{L^{1}} $$
since $J$ is bounded (from above and below) by a constant times $h^{-\frac{n-1}{2}}$. By Lemma \ref{Lp:semiclassical} one has
    $$ \Vert T u\Vert_{L^{1}} \lesssim \Vert u\Vert_{L^{1}} $$
provided $m<-\frac{n-1}{2}-M(1-\varrho)$ and $M>n$, i.e. if $m<-\frac{n-1}{2}+n(\varrho -1)$. This completes the proof of Theorem \ref{Intro:L1Thm}
\end{proof}
\subsection{Local and global $L^2$ boundedness of Fourier integral operators}
In this section we study the local and global $L^2$ boundedness properties of Fourier integral operators. Here we complete the global $L^2$ theory of Fourier integral operators with smooth strongly non-degenerate phase functions in class $\Phi^2$ and smooth amplitudes in the H\"ormander class $S^{m}_{\varrho, \delta}$ for all ranges of $\rho$'s and $\delta$'s. As a first step we establish global $L^2$ boundedness of Fourier integral operators with smooth phases and rough amplitudes in $L^\infty S^{m}_{\varrho}$, then we proceed by investigating the $L^2$ boundedness of Fourier integral operators with smooth phases and amplitudes and finally we consider the $L^2$ regularity of the operators with rough amplitudes in $L^{\infty}S^{m}_{\varrho}$ and rough non-degenerate phase functions in $L^{\infty} \Phi^2.$
\subsubsection{$L^2$ boundedness of Fourier integral operators with phases in $\Phi^2$}
 The global $L^2$ boundedness of Fourier integral operators which we aim to prove below, yields on one hand a global version of Eskin's and H\"ormander's local $L^2$ boundedness theorem for amplitudes in $S^{0}_{1,0}$, and on the other hand generalises the global $L^2$ result of Fujiwara's for amplitudes in $S^{0}_{0,0}$ to the case of rough amplitudes. Furthermore, as we shall see later, our result is sharp.
\begin{thm}
\label{global L2 boundedness smooth phase rough amplitude}
Let $a(x,\xi)\in L^{\infty}S^{m}_{\varrho}$ and the phase $\varphi(x,\xi) \in \Phi^2$ be strongly non degenerate. Then the Fourier integral operator
    $$T_{a,\varphi}u(x)=\frac{1}{(2\pi)^n} \int a(x,\xi)\, e^{i\varphi (x,\xi)}\,\widehat{u}(\xi)\, \dd\xi$$
is a bounded operator from $L^2$ to itself provided $m<\frac{n}{2}(\varrho-1).$ The bound on $m$ is sharp.
\end{thm}
\begin{proof}
 In light of Theorem \ref{general low frequency boundedness for rough Fourier integral operator}, we can confine ourselves to deal with the high frequency component $T_{h}$ of $T_{a,\varphi},$ hence we can assume that $\xi\neq 0$ on the support of the amplitude $a(x,\xi).$ Here we shall use a $T_{h}T_{h}^*$ argument, and therefore, the kernel of the operator $S_{h}=T_{h}T_{h}^*$ reads
\begin{align*}
     S_{h}(x,y) = \frac{1}{(2\pi h)^{n}} \int e^{\frac{i}{h}(\phi(x,\xi)-\phi(y,\xi))} \chi^2(\xi) a(x,\xi/h) \overline{a}(y,\xi/h) \, \dd\xi
\end{align*}
Now the strong non degeneracy assumption on the phase and Proposition \ref{equivalence of lowerbound and non degeneracy} yield that there is a constant $C>0$ such that $|\nabla_{\xi} \varphi(x,\xi)- \nabla_{\xi}\varphi(y,\xi)|\geq C |x-y|,$ for $x,y\in\mathbf{R}^{n}$ and $\xi\in \mathbf{R}^{n}\setminus 0.$ This enables us to use the non-stationary phase estimate in \cite{H1} Theorem 7.7.1, and the smoothness of the phase function $\varphi(x,\xi)$ in the spatial variable, yield that for all integers $N$
\begin{align*}
      |S_{h}(x,y)| \leq C_N h^{-2m-n-(1-\varrho)N} \langle h^{-1}(x-y) \rangle^{-N},
\end{align*}
for some constant $C_N >0.$
Let $M$ be a positive real number, we have $M=N+\theta$ where $N$ is the integer part of $M$ and $\theta \in [0,1)$ and therefore
\begin{align}
     \nonumber |S_{h}(x,y)| &= |S_{h}(x,y)|^{\theta} |S_{h}(x,y)|^{1-\theta} \\
     &\leq C^{1-\theta}_{N} C^{\theta}_{N+1}h^{-2m-n-(1-\varrho)M} \langle h^{-1}(x-y) \rangle^{-M}.
\end{align}
This implies
\begin{align}
     \sup_{x} \int |S_{h}(x,y)| \, \dd y \leq C_{M} h^{-2m-(1-\varrho)M}
\end{align}
for all $M>n$. By Cauchy-Schwarz and Young inequalities, we obtain
\begin{align}\label{semiclassical L2 pieces}
      \Vert T_{h}^*u\Vert_{L^2}^2 \leq \Vert S_{h}u\Vert_{L^2} \Vert u\Vert_{L^2} \leq C h^{-2m-(1-\varrho)M}  \Vert u\Vert^2_{L^2}.
\end{align}
Therefore, by Lemma \ref{Lp:semiclassical} we have the $L^2$ bound
   $$ \Vert Tu\Vert_{L^2} \lesssim \Vert u\Vert_{L^2} $$
provided $m<-(1-\varrho)M/2$ and $M>n$. This completes the proof of Theorem \ref{global L2 boundedness smooth phase rough amplitude}.
For the sharpness of this result we consider the phase function $\varphi(x,\xi)= \langle x,\xi\rangle \in \Phi^2$ which is strongly non-degenerate. It was shown in \cite{Rod} that for $m=\frac{n}{2}(\varrho-1)$
there are symbols $a(x,\xi)\in S^{m}_{\varrho, 1}$ such that the pseudodifferential operator
  $$ a(x,D)u(x)= \frac{1}{(2\pi)^{n}} \int a(x,\xi) \, e^{i\langle x, \xi\rangle}\, \widehat{u}(\xi)\, \dd \xi $$
is not $L^2$ bounded. Since $S^{m}_{\varrho, 1}\subset L^{\infty}S^{m}_{\varrho}$, it turns out that there are amplitudes in $L^{\infty}S^{m}_{\varrho}$ which yield an $L^2$
unbounded operator for a non-degenerate phase function in class $\Phi^2.$ Hence the order $m$ in the theorem is sharp.
\end{proof}
As a consequence, we obtain an alternative proof for the $L^2$ boundedness of pseudo-pseudodifferential operators introduced in \cite{KS}.
More precisely we have
\begin{cor}
\label{Intro:psipsi}
    Let $a(x,D)$ be a pseudo-pseudodifferential operator, i.e. an operator defined on the Schwartz class, given by
    \begin{equation}
         a(x,D)u=\frac{1}{(2\pi)^{n}} \int_{\R^n} e^{i\langle x, \xi\rangle} a(x,\xi) \widehat{u}(\xi) \, \dd \xi,
    \end{equation}
    with symbol $a \in L^{\infty} S^{m}_{\varrho}, $ $0\leq \varrho \leq 1$. If $m<n(\varrho -1)/2$, then $a(x,D)$ extends as an $L^2$ bounded operator.
\end{cor}
Theorem \ref{global L2 boundedness smooth phase rough amplitude} can be used to show a simple local $L^2$ boundedness result for Fourier integral operators with smooth symbols in the
H\"ormander class $S^{m}_{\varrho, \delta}$ in those cases when the symbolic calculus of the Fourier integral operators, as defined in \cite{H3} breaks down (e.g. in case $\delta \geq \varrho$), more precisely we have
\begin{cor}
Let $a(x,\xi) \in S^{m}_{\varrho, \delta}$ with compact support in $x$ variable and let $\varphi(x,\xi) \in \Phi^2$ be strongly non-degenerate. Then for $m<\frac{n}{2}(\varrho -\delta-1)$ and $0\leq \varrho \leq 1,$ $0\leq \delta \leq 1.$ Then the corresponding Fourier integral operator is bounded on $L^2 .$
\end{cor}
\begin{proof}
By Sobolev embedding theorem, for a function $f(x,y)$ one has
$$\int |f(x,x)|^2 \, \dd x \leq C_{n} \sum_{|\alpha|\leq N} \iint |\partial^{\alpha}_{y} f(x,y)|^2 \,\dd x \, \dd y$$ with $N>n/2$.
Now let $f(x,y):= \int a(y,\xi)\, e^{i\varphi(x,\xi)} \, \hat{u}(\xi) \, \dd\xi .$
Since $a(x,\xi) \in S^{m}_{\varrho, \delta}$, we have that $\partial^{\alpha}_{y}a(y,\xi) \in L^{\infty}S^{m+\delta|\alpha|}_{\varrho}$. Therefore, Theorem \ref{global L2 boundedness smooth phase rough amplitude} yields
$$\int |\partial^{\alpha}_{y} f(x,y)|^2 \,\dd x \lesssim \Vert u\Vert^2_{L^2},$$
provided that $m+\delta|\alpha| < \frac{n}{2}(\varrho-1)$. Since $|\alpha|\leq N$ and $N>n/2,$ one sees that it suffices to take $m<\frac{n}{2}( \varrho-\delta-1).$ We also note that in the argument above, the integration in the $y$ variable will not cause any problem due to the compact support assumption of the amplitude.
\end{proof}
However, as was shown by D. Fujiwara in \cite{Fuji}, Fourier integral operators with phases in $\Phi^2$ and amplitudes in $S^{0}_{0,0}$ are bounded in $L^2.$ This result suggests the possibility of the existence of an analog of the Calder\'{o}n-Vaillancourt theorem \cite{CV}, concerning $L^2$ boundedness of pseudodifferential operators with symbols in $S^{0}_{\varrho, \varrho}$ with $\varrho\in [0,1),$ in the realm of smooth Fourier integral operators. That this is indeed the case will be the content of Theorem \ref{Calderon-Vaillancourt for FIOs} below. However, before proceeding with the statement of that theorem, we will need two lemmas, the first of which is a continuous version of the Cotlar-Stein lemma, due to A. Calder\'{o}n and R. Vaillancourt, see i.e. \cite{CV} for a proof.
\begin{lem}\label{calderon-vaillancourt lemma}
Let $\mathscr{H}$ be a Hilbert space, and $A(\xi)$ a family of bounded linear endomorphisms of $\mathscr{H}$ depending on $\xi \in \mathbf{R}^n .$
Assume the following three conditions hold:
\begin{enumerate}
  \item the operator norm of $A(\xi)$ is less than a number $C$ independent of $\xi.$
  \item for every $u\in \mathscr{H}$ the function $\xi\mapsto A(\xi)u$ from $\mathbf{R}^{n} \mapsto \mathscr{H}$ is continuous for the norm topology of $\mathscr{H}.$
  \item for all $\xi_{1}$ and $\xi_2$ in $\mathbf{R}^n$
  \begin{equation}\label{cotlar estimates}
    \Vert A^{\ast}(\xi_1) A(\xi_2)\Vert \leq h(\xi_1 , \xi_2)^2,\,\,\,\text{and}\,\,\,  \Vert A(\xi_1) A^{\ast}(\xi_2)\Vert \leq h(\xi_1 , \xi_2)^2 ,
  \end{equation}
  with $h(\xi_1 , \xi_2 )\geq 0$ is the kernel of a bounded linear operator on $L^2$ with norm $K$.
\end{enumerate}
Then for every $E\subset \mathbf{R}^n$, with $|E|<\infty$, the operator $A_{E}=\int_{E} A(\xi) \, \dd\xi$ defined by $\langle A_{E} u, v\rangle _{\mathscr{H}} = \int _{E} \langle A(\xi) u, v\rangle _{\mathscr{H}}\, \dd\xi,$ is a bounded linear operator on $\mathscr{H}$ with norm less than or equal to $K.$
\end{lem}
We shall also use the following useful lemma.
\begin{lem}\label{integration by parts lem}
Let
  \begin{equation}\label{definition of integration by parts}
  Lu(x):= D^{-2} (1-i s(x)\langle\nabla_{x}F, \nabla_{x}\rangle)u(x),
  \end{equation}
    with $D:= (1+ s(x)|\nabla_{x} F|^{2})^{1/2}.$ Then
  \begin{enumerate}
    \item $L (e^{iF(x)})= e^{iF(x)}$
    \item if ${}^{t}L$ denotes the formal transpose of $L,$ then for any positive integer $N,$ $({}^{t}L)^{N} u(x)$ is a finite linear combination of terms of the form
    \begin{equation}\label{mainterm}
      CD^{-k} \{\prod_{\mu=1}^{p}\partial^{\alpha_\mu}_{x} s(x)\}\{\prod_{\nu =1}^{q}\partial^{\beta_{\nu}}_{x} F(x)\}\partial^{\gamma}_{x} u(x),
    \end{equation}
    with
    \begin{multline}\label{relations for order of derivatives}
      2N\leq k\leq 4N ;\, k-2N\leq p\leq k-N ;\,
      |\alpha_{\mu}|\geq 0;\, \sum_{\mu=1}^{p} |\alpha_{\mu}|\leq N\\
       k-2N\leq q\leq k-N;\, |\beta_{\nu}|\geq 1;\,\sum_{\mu=1}^{q} |\beta_{\nu}|\leq q+N;\, |\gamma| \leq N.
    \end{multline}
  \end{enumerate}
\end{lem}
\begin{proof}
First one notes that
\begin{equation*}
  \partial_{x_j} D^{-N}= -\frac{N}{2} D^{-N-2}\sum_{k=1}^{n}\{2s(x) \,\partial_{x_k}F\, \partial^{2}_{x_{j} x_{k}}F +  \partial_{x_{j}}s\, (\partial_{x_{k}}F)^2\}.
\end{equation*}
This and Leibniz's rule yield
\begin{align*}
  {}^{t}Lu(x) &=  D^{-2} u(x)+i\sum_{j=1}^{n}\partial_{x_{j}}(D^{-2}\, s(x)\,u(x)\, \partial_{x_j}F ) \\ &=D^{-2} u(x)-iD^{-4}\sum_{k,\,j=1}^{n} u (x)\, s(x)\,\partial_{x_j}F \,\big\{2s(x) \,\partial_{x_k}F\, \partial^{2}_{x_{j} x_{k}}F  \\ &\quad +  \partial_{x_{j}}s\, (\partial_{x_{k}}F)^2 +i D^{-2} \sum_{j=1}^{n} u(x)\, \partial_{x_{j}}s(x)\,\partial_{x_{j}}F
\\ &\quad+i D^{-2} \sum_{j=1}^{n} s(x)\,u(x)\,\partial^{2}_{x_{j}}F + i D^{-2} \sum_{j=1}^{n} s(x)\,\partial_{x_{j}}u(x)\,\partial_{x_{j}}F.
\end{align*}
From this it follows that ${}^{t}L$ is a linear combination of operators of the form
\begin{equation}
  \label{type1}
  D^{-2}\times
\end{equation}
\begin{equation}
  \label{type2}
  D^{-4}s^{2}(x)\,\partial_{x_j}F \,\partial_{x_k}F\, \partial^{2}_{x_{j} x_{k}}F \times
\end{equation}
\begin{equation}
  \label{type3}
  D^{-4}s\,\partial_{x_{j}}s\,\partial_{x_j}F \, (\partial_{x_{k}}F)^2 \times
\end{equation}
\begin{equation}
  \label{type4}
  D^{-2}\partial_{x_{j}}s(x)\,\partial_{x_{j}}F\times
\end{equation}
\begin{equation}
\label{type5}
  D^{-2} s(x)\,\partial^{2}_{x_{j}}F\times
  \end{equation}
\begin{equation}
  \label{type6}
  D^{-2} s(x)\,\partial_{x_{j}}F\,\partial_{x_{j}}.
\end{equation}
If we conventionally say that the term \eqref{mainterm} is of the type
  $$\bigg(k,p, \sum_{\mu=1}^{p}|\alpha_{\mu}|, q, \sum_{\nu=1}^q |\beta_{\nu}|, |\gamma|\bigg),$$
then ${}^{t} L$ is sum of terms of the types $(2,0,0,0,0,0),$ $(4,2,0,3,4,0),$ $(4,2,1,3,3,0),$ $(2,1,1,1,1,0),$ $(2,1,0,1,2,0)$ and $(2,1,0,1,1,1).$ Now operating the operators in \eqref{type1}, \eqref{type2}, \eqref{type3}, \eqref{type4}, \eqref{type5} on a term \eqref{mainterm} of type
  $$\bigg(k,p, \sum_{\mu=1}^{p}|\alpha_{\mu}|, q, \sum_{\nu=1}^q |\beta_{\nu}|, |\gamma|\bigg),$$
increases the types by $(2,0,0,0,0,0),$ $(4,2,0,3,4,0),$ $(4,2,1,3,3,0),$ $(2,1,1,1,1,0),$ $(2,1,0,1,2,0)$ respectively. To see how operating a term of form the \ref{type6} on \eqref{mainterm} changes the type we use Leibniz rule to obtain
\begin{multline*}
   D^{-2} s(x)\,\partial_{x_{j}}F\,\partial_{x_{j}} \bigg( D^{-k} \bigg\{\prod_{\mu=1}^{p}\partial^{\alpha_\mu}_{x} s(x)\bigg\}
   \bigg\{\prod_{\nu =1}^{q}\partial^{\beta_{\nu}}_{x} F(x)\bigg\}\,\partial^{\gamma}_{x} u(x)\bigg)=\\
   -\frac{k}{2} \bigg(D^{-k-4}\sum_{l=1}^{n}\partial_{x_{j}}F\,\Big\{2s(x) \,\partial_{x_l}F\, \partial^{2}_{x_{j} x_{l}}F +
    \partial_{x_{j}}s\, (\partial_{x_{l}}F)^2\Big\}\bigg)\times\\
    \bigg\{\prod_{\mu=1}^{p}\partial^{\alpha_\mu}_{x} s(x)\bigg\}\bigg\{\prod_{\nu =1}^{q}\partial^{\beta_{\nu}}_{x} F(x)\bigg\}\,\partial^{\gamma}_{x} u(x)\\
    +D^{-k-2}\partial_{x_{j}}F \sum_{\mu'=1}^{p}\bigg\{\prod_{\mu\neq \mu'}\partial^{\alpha_\mu}_{x} s(x)\bigg\} \bigg\{\partial_{x_{j}} \partial ^{\alpha_{\mu'}}_{x} s(x)\bigg\}
    \,\prod_{\nu =1}^{q}\partial^{\beta_{\nu}}_{x} F(x)\,\partial^{\gamma}_{x} u(x)\\
    +D^{-k-2}\partial_{x_{j}}F \, \prod_{\mu=1}^{p}\partial^{\alpha_\mu}_{x} s(x)\sum_{\nu'=1}^{q}\bigg\{\prod_{\nu\neq \nu'}\partial^{\alpha_\nu}_{x} F(x)\bigg\}
    \bigg\{\partial_{x_{j}} \partial ^{\alpha_{\nu'}}_{x} F(x)\bigg\}\,\partial^{\gamma}_{x} u(x)\\
    +D^{-k-2} \bigg\{\prod_{\mu=1}^{p}\partial^{\alpha_\mu}_{x} s(x)\bigg\}\bigg\{\prod_{\nu =1}^{q}\partial^{\beta_{\nu}}_{x} F(x)\bigg\}\,\partial_{x_{j}}\,\partial^{\gamma}_{x} u(x).
\end{multline*}
Therefore, upon application of ${}^{t} L$ to \eqref{mainterm}, the types of the resulting terms increase by $(2,0,0,0,0,0),$ $(4,2,0,3,4,0),$ $(4,2,1,3,3,0),$ $(2,1,1,1,1,0),$ $(2,1,0,1,2,0)$ and $(2,1,0,1,1,1).$ Iteration of this process yields
\begin{equation*}
  ({}^{t} L)^{N} u(x)= \sum C\,D^{-k} \bigg\{\prod_{\mu=1}^{p}\partial^{\alpha_\mu}_{x} s(x)\bigg\}\bigg\{\prod_{\nu =1}^{q}\partial^{\beta_{\nu}}_{x} F(x)\bigg\}\partial^{\gamma}_{x} u(x),
\end{equation*}
where the summation is taken over all non-negative integers $N_1 ,$ $N_2 ,$ $N_3 ,$ $N_4 ,$ $N_5 ,$ $N_6$ with $\sum_{j=1}^{6} N_j =N$ and
\begin{multline}
  \bigg(k,p, \sum_{\mu=1}^{p}|\alpha_{\mu}|, q, \sum_{\nu=1}^q |\beta_{\nu}|, |\gamma|\bigg)= N_1 (2,0,0,0,0,0)+ N_2 (4,2,0,3,4,0)+\\ N_3 (4,2,1,3,3,0)+ N_4 (2,1,1,1,1,0)+ N_5 (2,1,0,1,2,0)+ N_6 (2,1,0,1,1,1).
\end{multline}
Hence,
\begin{equation}
  \label{k}
  k= 2N_1 +4 N_2 + 4 N_3 + 2N_4 +2N_5 +2N_6
\end{equation}
\begin{equation}
  \label{p}
  p= 2N_2 +2N_3 + N_4 + N_5 + N_6
\end{equation}
\begin{equation}
  \label{sum alphamu}
  \sum_{\mu=1}^{p} |\alpha_{\mu}|= N_3 + N_4
\end{equation}
\begin{equation}
  \label{q}
  q= 3N_2 + 3 N_3 + N_4 + N_5 +N_6
\end{equation}
\begin{equation}
  \label{sum betanu}
 \sum_{\nu=1}^{q} |\beta_{\nu}|=4 N_2 + 3 N_3 + N_4 + 2 N_5 + N_6
\end{equation}
\begin{equation}
  \label{gamma}
  |\gamma|= N_6 .
\end{equation}
From this it also follows that $(k,p, \sum_{\mu=1}^{p}|\alpha_{\mu}|, q, \sum_{\nu=1}^q |\beta_{\nu}|, |\gamma|)$ satisfy \eqref{relations for order of derivatives}.
\end{proof}
\begin{thm}\label{Calderon-Vaillancourt for FIOs}
If $m=\min(0,\frac{n}{2}(\varrho-\delta)),$ $0\leq \varrho\leq 1$, $0\leq \delta<1,$ $a\in S^{m}_{\varrho, \delta}$ and $\varphi\in \Phi^{2},$ satisfies the strong non-degeneracy condition, then the operator $T_{a} u(x)= \int a(x,\xi)\, e^{i\varphi(x,\xi)} \hat{u}(\xi)\, \dd\xi$ is bounded on $L^2.$
\end{thm}
\begin{proof}
First we observe that since for $\delta \leq \varrho$, $S^{0}_{\varrho, \delta} \subset S^{0}_{\varrho, \varrho},$ it is enough to show the theorem for $0\leq \varrho\leq \delta<1$ and $m=\frac{n}{2}(\varrho-\delta).$ Also, as we have done previously, we can assume without loss of generality that $a(x,\xi)=0$ when $\xi$ is in a a small neighbourhood of the origin. Using the $TT^{\ast}$ argument, it is enough to show that the operator
\begin{equation}\label{Tb amplitude presentation}
  T_{b} u(x)= \iint b(x,y,\xi)\, e^{i\varphi(x,\xi)-i\varphi(y,\xi)}\, u(y)\, \dd y \, \dd\xi,
\end{equation}
where $b$ satisfies the estimate
\begin{equation}\label{derivtives of b}
|\partial^{\alpha}_{\xi}\partial^{\beta}_{x} \partial^{\gamma}_{y} b(x,y,\xi)|\leq C_{\alpha\, \beta\,\gamma} \langle \xi \rangle ^{m_1-\varrho|\alpha|+\delta(|\beta| +|\gamma|)},
\end{equation}
 with $m_1=n(\varrho-\delta)$ and $0\leq \varrho\leq \delta<1,$ is bounded on $L^2 .$\\
\noindent We introduce a differential operator
   $$ L:= D^{-2} \Big\{1-i\langle \xi\rangle^{\varrho}\big(\langle\nabla_{\xi}\varphi(x,\xi)-\nabla_{\xi}\varphi(y,\xi), \nabla_{\xi}\rangle\big)\Big\},$$
with $D=(1+\langle \xi\rangle ^{\varrho}|\nabla_{\xi}\varphi(x,\xi)-\nabla_{\xi}\varphi(y,\xi)|^2)^{\frac{1}{2}}.$ It follows from Lemma \ref{integration by parts lem} that
 \begin{equation*}
  L (e^{i\varphi(x,\xi)-i\varphi(y,\xi)})= e^{i\varphi(x,\xi)-i\varphi(y,\xi)}
\end{equation*}
and that $({}^t L)^{N}(b(x,y,\xi))$ is a finite sum of terms of the form
\begin{equation}
D^{-k} \bigg\{\prod_{\mu=1}^{p}\partial^{\alpha_\mu}_{\xi} \langle \xi\rangle^{\varrho}\bigg\}\bigg\{\prod_{\nu =1}^{q} \big(\partial^{\beta_{\nu}}_{\xi}\varphi(x,\xi)-\partial^{\beta_{\nu}}_{\xi}\varphi(y,\xi)\big)\bigg\}\,\partial^{\gamma}_{\xi}b(x,y,\xi).
\end{equation}
Furthermore since $\varphi \in \Phi^2$ is assumed to be strongly non-degenerate, we can use Proposition \ref{equivalence of lowerbound and non degeneracy} to deduce that
\begin{equation}
\label{lowerbound for phi in x}
|\nabla_{\xi} \varphi(x,\xi)-\nabla_{\xi}\varphi(y,\xi)|\geq c_1 |x-y|
\end{equation}
 \begin{equation}
\label{lowerbound for phi in xi}
|\nabla_{z} \varphi(z,\xi_2)-\nabla_{z}\varphi(z,\xi_1)|\geq c_2|\xi_1-\xi_2|.
\end{equation}
Using \eqref{lowerbound for phi in x}, \eqref{relations for order of derivatives} and \eqref{derivtives of b}, we have
\begin{equation}\label{x derivative estimate}
  |\partial^{\sigma}_{x}({}^t L)^{N}(b(x,y,\xi))|\leq C \Lambda(\langle \xi\rangle ^{\varrho} (x-y))\langle \xi\rangle ^{m_1+\delta|\sigma|},
\end{equation}
where $\Lambda$ is an integrable function with $\int \Lambda (x) \, dx\lesssim 1.$
 Integration by parts using $L$, $N$ times, in \eqref{Tb amplitude presentation} one has
 \begin{equation}
 T_{b} u(x)= \iint c(x,y,\xi)\, e^{i\varphi(x,\xi)-i\varphi(y,\xi)}\, u(y)\, \dd y \, \dd\xi,
 \end{equation}
 with $c(x,y,\xi)=({}^t L)^{N}(b(x,y,\xi))$ and
 \begin{equation}\label{derivative estimates for c}
 |\partial^{\sigma}_{x}c(x,y,\xi)|\leq C \Lambda(\langle \xi\rangle ^{\varrho} (x-y))\langle \xi\rangle ^{m_1+\delta|\sigma|}
 \end{equation}
 and the same estimate is valid for $\partial^{\sigma}_{y}c(x,y,\xi).$
 From this we get the representation
 \begin{equation}
   T_{b} = \int A(\xi)\, \dd\xi,
 \end{equation}
where $A(\xi) u(x):= \int c(x,y,\xi)\, e^{i\varphi(x,\xi)-i\varphi(y,\xi)}\, u(y)\, \dd y.$
Noting that $A(\xi)=0$ for $\xi$ outside some compact set, we observe that condition (1) of Lemma \ref{calderon-vaillancourt lemma} follows from Young's inequality and \eqref{derivative estimates for c} with $\sigma=0,$  and condition (2) of Lemma \ref{calderon-vaillancourt lemma} follows from the assumption of the compact support of the amplitude. To verify condition (3) we confine ourselves to the estimate of $\Vert A^{\ast}(\xi_1) A(\xi_2)\Vert$, since the one for $\Vert A(\xi_1) A^{\ast}(\xi_2)\Vert$ is similar. To this end,
a calculation shows that the kernel of $A^{\ast}(\xi_1) A(\xi_2) $ is given by
\begin{multline}\label{kernel of A star A}
K(x,y,\xi_1 , \xi_2):= \\ \int \overline{c}(z,x,\xi_1)\, c(z,y,\xi_2 )\, e^{i[\varphi(z,\xi_{2})-\varphi(z,\xi_{1})+\varphi(x,\xi_{1})-\varphi(y,\xi_{2})]}\, \dd z.
\end{multline}
The estimate \eqref{derivative estimates for c} yields
\begin{multline}\label{first estimate for K}
|K(x,y,\xi_1 , \xi_2)|\\\lesssim \langle \xi_1\rangle ^{m_1}\, \langle \xi_2\rangle ^{m_1} \int   \Lambda(\langle \xi_{1}\rangle ^{\varrho} (x-z))\,
\Lambda(\langle \xi_{2}\rangle ^{\varrho} (y-z))\, \dd z.
\end{multline}
Therefore by choosing $N$ large enough, Young's inequality and using the fact that $\int \Lambda (x) \, dx \lesssim 1$ yield
\begin{equation}\label{first estim for A star A}
\Vert A^{\ast}(\xi_1) A(\xi_2)\Vert \lesssim \langle \xi_1\rangle ^{m_1 -n\varrho}\, \langle \xi_2\rangle ^{m_1 -n\varrho}.
\end{equation}
At this point we introduce another first order differential operator $M:= G^{-2} \{1-i(\langle\nabla_{z}\varphi(z,\xi_2)-\nabla_{z}\varphi(z,\xi_1),\nabla_{z}\rangle)\}$, with $G=(1+|\nabla_{z}\varphi(z,\xi_2)-\nabla_{z}\varphi(z,\xi_1)|^2)^{\frac{1}{2}}.$ Using the fact that $M e^{i(\varphi(z,\xi_{2})-\varphi(z,\xi_{1}))}= e^{i(\varphi(z,\xi_{2})-\varphi(z,\xi_{1}))},$ integration by parts in \eqref{kernel of A star A} yields
\begin{equation}
  \int ({}^{t} M)^{N'}\{\overline{c}(z,x,\xi_1)\, c(z,y,\xi_2 )\}\, e^{i[\varphi(z,\xi_{2})-\varphi(z,\xi_{1})+\varphi(x,\xi_{1})-\varphi(y,\xi_{2})]}\, \dd z.
\end{equation}
Using the second part of Lemma \ref{integration by parts lem}, we find that $({}^{t} M)^{N'}\{\overline{c}(z,x,\xi_1)\, c(z,y,\xi_2 )\}$ is a linear combination of terms of the form
\begin{equation}\label{differential operator M}
  G^{-k}\bigg\{\prod_{\nu =1}^{q}(\partial^{\beta_{\nu}}_{z}\varphi(z,\xi_2)-\partial^{\beta_{\nu}}_{\xi}\varphi(z,\xi_1))\bigg\} \partial^{\gamma_1}_{z}\overline{c}(z,x,\xi_1)\, \partial^{\gamma_2}_{z} c(z,y,\xi_2),
\end{equation}
where $k,$ $q,$ $\beta_\nu$ satisfy the inequalities in \ref{relations for order of derivatives} and $|\gamma_1|+|\gamma_2| \leq N'.$ Now, \eqref{lowerbound for phi in xi}, \eqref{derivative estimates for c} and \eqref{differential operator M}, yield the following estimate for $K(x,y,\xi_1 , \xi_2)$
\begin{align}\label{second estimate for K}
 |K(x,y,\xi_1 , \xi_2)|\lesssim \langle \xi_1\rangle ^{m_1}\, \langle \xi_2\rangle ^{m_1}(1+|\xi_1|+|\xi_2|)^{\delta N'} |\xi_1 -\xi_2|^{-N'} \\ \nonumber
 \times \int  \Lambda(\langle \xi_{1}\rangle ^{\varrho} (x-z))\, \Lambda(\langle \xi_{2}\rangle ^{\varrho} (y-z))\, \dd z.
\end{align}
Once again, choosing $N$ large enough, Young's inequality yields
\begin{equation}\label{second estim for A star A}
\Vert A^{\ast}(\xi_1) A(\xi_2)\Vert \lesssim \langle \xi_1\rangle ^{m_1-n\varrho}\, \langle \xi_2\rangle ^{m_1-n\varrho} \frac{(1+|\xi_1|+|\xi_2|)^{\delta N'}}{|\xi_1 -\xi_2|^{N'}}.
\end{equation}
Using the fact that for $x>0,$ $\inf (1,x) \sim (1+\frac{1}{x})^{-1}$, one optimizes the estimates \eqref{first estim for A star A} and \eqref{second estim for A star A} by
\begin{align}\label{optimal estim for A star A}
\Vert A^{\ast}(\xi_1) A(\xi_2)\Vert &\lesssim \langle \xi_1\rangle ^{m_1 -n\varrho}\, \langle \xi_2\rangle ^{m_1 -n\varrho}
\bigg(1+ \frac{|\xi_1 -\xi_2|^{N'}}{(1+|\xi_1|+|\xi_2|)^{\delta N'}}\bigg)^{-1}
\\\nonumber &:= h^{2}(\xi_1 , \xi_2).
\end{align}
Therefore  recalling that $m_1=n(\varrho-\delta),$ in applying Lemma \ref{calderon-vaillancourt lemma}, we need to show that
\begin{equation}
  K(\xi_1 , \xi_2)= (1+|\xi_1|)^{\frac{-n\delta}{2}} (1+|\xi_2|)^{-\frac{n\delta}{2}}  \bigg(1+ \frac{|\xi_1 -\xi_2|^{N'}}{(1+|\xi_1|+|\xi_2|)^{\delta N'}}\bigg)^{-\frac{1}{2}}
\end{equation}
is the kernel of a bounded operator in $L^2.$ At this point we use Schur's lemma, which yields the desired conclusion provided that
$$ \sup_{\xi_1} \int K(\xi_1 , \xi_2)\, \dd \xi_2, \quad \sup_{\xi_2} \int K(\xi_1 , \xi_2)\, \dd \xi_1 $$
are both finite. Due to the symmetry of the kernel, we only need to show the finiteness of one of these quantities.\\
\noindent To this end, we fix $\xi_1$ and consider the domains $\mathcal{A}=\{(\xi_1, \xi_2);\, |\xi_2| \geq 2 |\xi_1|\},$ $\mathcal{B}=\{(\xi_1, \xi_2);\, \frac{|\xi_1|}{2}\leq |\xi_2| \leq 2 |\xi_1|\},$ and $\mathcal{C}=\{(\xi_1, \xi_2);\, |\xi_2| \leq \frac{ |\xi_1|}{2}\}.$ Now we observe that on the set $\mathcal{A},$ $K(\xi_1, \xi_2)$ is dominated by
\begin{equation}\label{estimate on A}
  (1+|\xi_1|)^{-\frac{n\delta}{2}} (1+|\xi_2|)^{-\frac{n\delta}{2}+\frac{N'}{2}(\delta -1)},
\end{equation}
on $\mathcal{B},$ $K(\xi_1, \xi_2)$ is dominated by
\begin{equation}\label{estimate on A}
  (1+|\xi_1|)^{-n \delta}  \bigg(1+ \frac{|\xi_1 -\xi_2|^{N'}}{(1+|\xi_1|)^{\delta N'}}\bigg)^{-\frac{1}{2}},
\end{equation}
and on $\mathcal{C},$ $K(\xi_1, \xi_2)$ is dominated by
\begin{equation}\label{estimate on A}
   (1+|\xi_2|)^{-\frac{n\delta}{2}} (1+|\xi_1|)^{-\frac{n\delta}{2}+\frac{N'}{2}(\delta -1)}.
\end{equation}
Therefore, if $I_{\Omega}:= \int_{\Omega} K(\xi_1 , \xi_2 )\, d\xi_2,$ then choosing $\frac{N'}{2}(\delta -1)<-n ,$ which is only possible if $\delta<1, $ we have that $I_{\mathcal{A}} <\infty$ uniformly in $\xi_1$. Also,
\begin{equation}
  I_{\mathcal{C}} \leq (1+|\xi_1|)^{n-\frac{n\delta}{2}+\frac{N'}{2}(\delta -1)}\leq C,
\end{equation}
which is again possible by the fact that $\delta<1$ and a suitable choice of $N'.$
In $I_\mathcal{B}$ let us make a change of variables to set $\xi_2 -\xi_1 = (1+|\xi_1| )^{\delta} y$, then
\begin{equation}
I_{\mathcal{B}} \leq \int (1+|y|^{N'}) ^{-\frac{1}{2}} \dd y <\infty,
\end{equation}
by taking $N'$ large enough. These estimates yield the desired result and the proof of there theorem is therefore complete.
\end{proof}
\subsubsection{$L^2$ boundedness of Fourier integral operators with phases in $L^\infty \Phi^2$}
Next we shall turn to the problem of $L^2$ boundedness of Fourier integral operators with non-smooth amplitudes and phases.  As was mentioned in the introduction, a motivation for considering fully rough Fourier integral operators stems from a "linearisation" procedure which reduces certain maximal operators to Fourier integral operators with a non-smooth phase and sometimes also a non-smooth amplitude. For instance, estimates for the maximal spherical average operator
\begin{align*}
   A u(x) = \sup_{t \in [0,1]} \bigg| \int_{S^{n-1}} u(x+t\omega) \, \dd\sigma(\omega) \bigg|
\end{align*}
are related to those for the maximal wave operator
   $$ W u(x) = \sup_{t \in [0,1]} \big|e^{it\sqrt{-\Delta}}u(x)\big|, $$
and can for instance be deduced from those of the linearized operator
\begin{align}
\label{Intro:LinWave}
   e^{it(x)\sqrt{-\Delta}}u=(2\pi)^{-n} \int_{\R^n} e^{it(x)|\xi|+i\langle x,\xi \rangle} \widehat{u}(\xi) \, \dd\xi,
\end{align}
where $t(x)$ is a measurable function in $x$, with values in $[0,1]$ and the phase here belongs to the class $L^\infty \Phi^2.$
As will be demonstrated later, the validity of the results in the rough case depend on the geometric conditions (imposed on the phase functions) which are the rough analogues of the non-degeneracy and corank conditions for smooth phases. In trying to understand the subtle interrelations between boundedness, smoothness and geometric conditions, we remark that even if one assumes the phase of the linearized operator \eqref{Intro:LinWave} to be smooth, there are cases for which the canonical relation of this operator ceases to be the graph of a symplectomorphism.
Indeed, contrary to the wave operator $e^{i t\sqrt{-\Delta}}$ at fixed time $t \in [0,1]$, the phase $\phi(x,\xi)=\langle x,\xi \rangle+t(x)|\xi|$
of the linearized operator cannot be a generating function of a canonical transformation, (see \cite{D}), in certain cases since
\begin{align*}
    \frac{\d^{2}\phi}{\d x \d\xi}(x,\xi) &= {\rm Id} + \nabla t(x) \otimes \frac{\xi}{|\xi|}, \\
    \ker  \frac{\d^{2}\phi}{\d x \d\xi}(x,\xi) &= \mathop{\rm span} \nabla t(x) \quad \textrm{ when }  \langle \xi,\nabla t(x)\rangle + |\xi|=0,
\end{align*}
and this happens when $|\nabla t(x)| \geq 1$ and $\xi = \varrho(-\frac{\nabla t(x)}{|\nabla t(x)|^{2}}+\eta)$ with $\varrho \in \R^*_{+}$ and $\eta$ is a vector orthogonal to $\nabla t(x)$ of norm $(1-|\nabla t(x)|^{-2})^{1/2}$. Therefore, one can not expect $L^2$ boundedness of \eqref{Intro:LinWave}
even when the function $t(x)$ is smooth. Nevertheless,  in this case the rank of the Hessian $\d^2\phi/\d x\d \xi$ drops by one with respect
to its maximal possible value, and one could still establish $L^2$ estimates with loss of derivatives (see section \ref{subsec:symplecticL2} for more details). The operators that we intend to study will fall into this category. Before we investigate the local $L^2$ boundedness of operators based on geometric conditions on their phase, we state and prove a purely analytic global $L^2$ boundedness result which will be used later.\\
\begin{thm}
\label{Intro:L2Thm}
      Let $T$ be a Fourier integral operator given by \eqref{Intro:Fourier integral operator} with amplitude $a \in L^{\infty}S^m_{\varrho}, 0\leq \varrho \leq 1$
      and a phase function $\phi(x,\xi) \in L^{\infty}\Phi^2$ satisfying the rough non-degeneracy condition.
     Then there exists a constant $C>0$ such that
           $$ \Vert T u\Vert_{L^{2}} \lesssim \Vert u\Vert_{L^{2}} $$
      provided $m<n(\varrho -1)/2-(n-1)/4$.
\end{thm}
\begin{proof}
     Using semiclassical reduction of Subsection \ref{Semiclasical reduction subsec}, we decompose $T$ into low and high frequency portions $T_0$ and $T_h$.
     The boundedness of $T_0$ follows at once from Theorem \ref{general low frequency boundedness for rough Fourier integral operator}, so it remains to establish
     suitable semiclassical estimates for $T_{h}$. Once again we use the $TT^*$ argument. The kernel of the operator $S_{h}=T_{h}T_{h}^*$ reads
     \begin{align*}
          S_{h}(x,y) = (2\pi h)^{-n} \int e^{\frac{i}{h}(\phi(x,\xi)-\phi(y,\xi))} \chi^2(\xi) a(x,\xi/h) \overline{a}(y,\xi/h) \, \dd \xi.
     \end{align*}
     We now use the Seeger-Sogge-Stein decomposition (section \ref{SSS decomposition}) and split this operator as the sum $\sum_{j=1}^N S_h^{\nu}$
     where the kernel of $S_{h}^{\nu}$ takes the form
     \begin{align*}
          S^{\nu}_{h}(x,y) = (2\pi h)^{-n} \int e^{\frac{i}{h}\langle \nabla_{\xi}\phi(x,\xi^{\nu})-\nabla_{\xi}\phi(y,\xi^{\nu}),\xi\rangle} b_{\nu}(x,\xi,h) \overline{b_{\nu}}(y,\xi,h) \, \dd \xi.
     \end{align*}
     We consider the following differential operator
          $$ L = 1-\d_{\xi_{1}}^2-h\d_{\xi'}^2 $$
     for which we have according to Lemma \ref{Linfty:bLemma}
     \begin{equation}
          \sup_{\xi} \Vert L^{N}b^{\nu}(\cdot,\xi,h)\Vert_{L^{\infty}} \lesssim  h^{-m-2N(1-\varrho)}.
     \end{equation}
     Integration by parts yields
     \begin{multline*}
          |S_{h}^{\nu}(x,y)| \leq (2\pi h)^{-n} \big(1+g\big(\nabla_{\xi}\phi(y,\xi^{\nu})-\nabla_{\xi}\phi(x,\xi^{\nu})\big)\big)^{-N} \\
          \times \int \big|L^N \big(b^{\nu}(x,\xi,h) \overline{b^{\nu}}(y,\xi,h)\big)\big| \, \dd \xi
     \end{multline*}
     for all integers $N$, with
     \begin{equation}
          g(z)=h^{-2}z_{1}^2+h^{-1}|z'|^{2}.
     \end{equation}
     The standard interpolation trick gives the same inequality for  for all positive numbers $M>0$  and thus we have
     \begin{align*}
          |S_{h}^{\nu}(x,y)| \leq C h^{-2m-\frac{n+1}{2}-2M(1-\varrho)} \big(1+g\big(\nabla_{\xi}\phi(y,\xi^{\nu})-\nabla_{\xi}\phi(x,\xi^{\nu})\big)\big)^{-M}
     \end{align*}
     since the volume of the portion of cone $|A \cap \Gamma_{\nu}|$ is of the order of $h^{(n-1)/2}$. By the non-degeneracy assumption and Lemma \ref{Lem:Substitution}, we get
     \begin{align*}
         \int |S_{h}^{\nu}(x,y)| \, \dd y \leq  C h^{-2m-\frac{n+1}{2}-2M(1-\varrho)} \underbrace{\int \big(1+g(z)\big)^{-M} \, \dd z}_{= c h^{\frac{n+1}{2}}}.
     \end{align*}
     By Young's inequality (remembering that the kernel $S_{h}^{\nu}(x,y)$ is symmetric), we obtain
     \begin{align*}
          \|S_{h}^{\nu}u\|_{L^2} \leq C h^{-2m-2M(1-\varrho)}  \|u\|_{L^2}
     \end{align*}
     and summing the inequalities
     \begin{align*}
          \|T_{h}^*u\|_{L^2}^2  \leq \sum_{\nu=1}^J \|S_{h}^{\nu}u\|_{L^2} \|u\|_{L^2} \leq C h^{-2m+\frac{n-1}{2}-2(1-\varrho)M}  \|u\|_{L^2},
     \end{align*}
     since there are roughly $h^{-(n-1)/2}$ terms in the sum. By Lemma \ref{Lp:semiclassical}, we have the $L^2$ bound
         $$ \|Tu\|_{L^2} \lesssim \|u\|_{L^2} $$
     provided $m<-(n-1)/4+(\varrho-1)M$ and $M>n/2$, which yields the desired result.
\end{proof}
\begin{rem}
     The reason why we were led to perform the Seeger-Sogge-Stein decomposition is that under the rough non-degeneracy assumption
     (Definition \ref{defn of rough nondegeneracy}), the non-stationary phase (Theorem 7.7.1 \cite{H1}) provides the bound
     \begin{align}
          |S_{h}(x,y)| &\leq C_{N} h^{-2m-n+N} |x-y|^{2N} \\ \nonumber &\leq C_{N}h^{-2m-n}\big(1+h^{-1}|x-y|^2)^{-N}
     \end{align}
     leading, when say $\varrho=1$, to a loss of $n/4$ derivatives instead of $(n-1)/4$ derivatives in our case.
     This however can be improved to no loss of derivatives when one also assumes that there is a Lipschitz bound on the higher order derivatives
          $$ |\d^{\alpha}_{\xi}\phi(x,\xi)-\d^{\alpha}_{\xi}\phi(y,\xi)| \leq C_{\alpha}|x-y|, \quad |\alpha| \geq 2. $$
     This is indeed the case in dimension $n=1$, or if the phase can be decomposed as $\phi(x,\xi) =\phi^{\sharp}(x,\xi)+ \phi^{\flat}(x,\xi)$ where
     $ \phi^{\sharp}$ is linear in $\xi$ and $\phi^{\flat} \in \Phi^2$.
\end{rem}
Let $\pi_1$ denote the projection onto the spatial variables, i.e.
\begin{align*}
     \pi_{1}:T^*\R^n &\to \R^n \\
     (x,\xi) &\mapsto x.
\end{align*}
A geometric condition sufficient for the local $L^2$ boundedness of rough Fourier integral operators with phase functions $\varphi(x,\xi)$ and amplitudes $a(x,\xi)$ is as follows:
\begin{hyp}
\label{Rough corank condition}
     For each $x\in\pi_{1}(\supp a)$ and all $\xi\in \mathbf{S}^{n-1}$ there exists a linear subspace $V_{x,\xi}$ belonging to the Grassmannian ${\rm Gr}(n,n-k)$ varying continuously with $(x,\xi)$, and constants $c_{1},c_{2}>0$ such that if $\pi_{V_{x,\xi}}$ denotes the projection onto $V_{x,\xi}$, then
       $$ |\d_{\xi}\phi(x,\xi)-\d_{\xi}\phi(y,\xi)|+c_{1}|x-y|^2 \geq c_{2} |\pi_{V_{x,\xi}}(x-y)| $$
     for all $x,y \in \pi_1(\supp a)$.
\end{hyp}
\begin{thm}
\label{Intro:L2ThmDeg}
      Let $T$ be a Fourier integral operator given by \eqref{Intro:Fourier integral operator} with amplitude $a \in L^{\infty}S^m_\varrho$ and phase function $\phi \in L^{\infty}\Phi^2$.
      Suppose that the phase satisfies the rough corank condition $\ref{Rough corank condition}$, then  $T$ can be extended as a bounded operator from $L^{2}_{\rm comp}$
      to $L^{2}_{\rm loc}$ provided $m<-\frac{n+k-1}{4}+\frac{(n-k)(\varrho-1)}{2}.$
\end{thm}
\begin{proof}
Since we aim to prove a local $L^2$ boundedness result, we may assume that the amplitude $a$
is compactly supported in the spatial variable $x$. Then since $S_{0}=T_{0}T^*_{0}$ has a bounded compactly supported kernel, it extends to a bounded operator on $L^2$. It remains to deal with the high frequency part of the operator. Given $(x_{\mu} , \xi_{\mu})\in \R ^n \times \R^n$, $\mu= 1, \dots,\, J,$
we consider a partition of unity
       $$ \sum_{\mu=1}^J \psi^{\mu}(x,\xi)=1, \quad \xi \neq 0 $$
given by functions $\psi^{\mu}$ homogeneous of degree $0$ in the frequency variable $\xi$ supported in cones
\begin{align*}
    \Gamma^{\mu}=\Big\{ (x,\xi) \in T^*\R^n;\, |x-x^{\mu}|^2+\Big\vert \frac{\xi}{\vert\xi\vert}-\xi^{\mu} \Big\vert^2 \leq \eps^2 \Big\}
\end{align*}
where $\eps$ is yet to be chosen. We decompose the operator as
\begin{align}
     T_{h} = \sum_{\mu=1}^N T_{h}^{\mu}
\end{align}
where the kernel of $T_{h}^{\mu}$ is given by
\begin{align*}
     T_{h}^{\mu}(x,y)=(2\pi h)^{-n} \int_{\R^n} e^{\frac{i}{h} \phi(x,\xi)-\frac{i}{h}\langle y,\xi \rangle} \psi^{\mu}(x,\xi)\chi(\xi) a(x,\xi/h) \, \dd\xi.
\end{align*}
We have the direct sum
    $$ \R^n = V_{x^{\mu},\xi^{\mu}} \oplus V_{x^{\mu},\xi^{\mu}}^{\perp}, \quad \dim V_{x^{\mu},\xi^{\mu}}=n-k $$
and we decompose vectors $x=x'+x''$ (i.e. $x=(x' , x'')$) according to this sum. Assumption \ref{Rough corank condition} implies
\begin{align*}
     |\d_{\xi}\phi(x' , x'',\xi)&-\d_{\xi}\phi(y' , x'',\xi)| \\ &\geq c_{2}|\pi_{V_{x,\xi}}(x'-y')| -c_{1}|x'-y'|^2 \\
     &\geq c_{2}|x'-y'| \Big(1-\Vert\pi_{V_{x,\xi}}-\pi_{V_{x^{\mu},\xi^{\mu}}}\Vert -  \frac{c_{1}}{c_{2}}|x'-y'|\Big).
\end{align*}
Now since $(x,\xi) \mapsto \pi_{V_{x,\xi}}$ is continuous, we can choose $\eps$ in the definition of the cone $\Gamma^{\mu}$ small enough so that
\begin{align*}
    \Vert\pi_{V_{x,\xi}}-\pi_{V_{x^{\mu},\xi^{\mu}}}\Vert \leq \frac{1}{4}, \quad |x'-y'| \leq |x'-{x^{\mu}}^{'}| + |y'-{x^{\mu}}^{'}| \leq \frac{c_2}{4c_1}
\end{align*}
and therefore we have
\begin{align}
\label{L2:PartialNonDeg}
     |\d_{\xi}\phi(x', x'',\xi)-\d_{\xi}\phi(y',x'',\xi)| \geq \frac{c_{2}}{2}|x' -y'|
\end{align}
when $(x,\xi)$ and $(y,\xi)$ belong to $\Gamma^{\mu}$.
We fix the $x''$ variable and use a $TT^*$ argument on the operator acting in the $x'$ variables.
We consider
\begin{align*}
     S^{\mu}_{h}(x',x'', y') = (2\pi h)^{-n} \int e^{\frac{i}{h}(\phi(x,\xi)-\phi(y',x'',\xi))} a^{\mu}_{h}(x,\xi) \overline{a^{\mu}_{h}}(y',x'',\xi) \, \dd \xi.
\end{align*}
Because of \eqref{L2:PartialNonDeg}, performing a Seeger-Sogge-Stein decomposition and reasoning as in the proof of Theorem \ref{Intro:L2Thm} we get
\begin{multline*}
    \bigg(\int_{V_{x^{\mu},\xi^{\mu}}} \bigg| \int_{V_{x^{\mu},\xi^{\mu}}} S^{\mu}_{h}(x',x'', y') u(y') \, \dd y' \bigg|^2 \, \dd x'\bigg)^{\frac{1}{2}}
    \\ \leq C h^{-2m-\frac{n-k-1}{2}-k-2M(1-\varrho)} \bigg(\int |u(y')|^2 \, \dd y'\bigg)^{\frac{1}{2}},
\end{multline*}
with a constant $C$ that is independent of $x'',$ provided $M>\frac{n-k}{2}$ and therefore
\begin{align*}
     \int \Big| \int_{V_{x^{\mu},\xi^{\mu}}}  \overline{T^{\mu}_{h}(x',x'',y)}u(x) \, \dd x' \Big|^{2} \, \dd y
     \leq C h^{-2m-\frac{n-1}{2}-\frac{k}{2}-2M(1-\varrho)} \|u\|^2_{L^2}.
\end{align*}
Hence by Minkowski's integral inequality
\begin{align*}
    \Vert T_{h}^*u\Vert_{L^2} &\leq  \int_{V_{x^{\mu},\xi^{\mu}}^{\perp}}\bigg(\int \Big| \int_{V_{x^{\mu},\xi^{\mu}}}
    \overline{T^{\mu}_{h}(x',x'',y)}u(x) \, \dd x' \Big|^{2} \, \dd y \bigg)^{\frac{1}{2}} \dd x'' \\
    &\leq C  h^{-m-\frac{n-1}{4}-\frac{k}{4}-M(1-\varrho)}\Vert u\Vert_{L^2}
\end{align*}
provided $M>\frac{n-k}{2}$ and the amplitude is compactly supported in $x''.$ This yields the $L^2$ bound for $m<-(n-1+k)/4-(1-\varrho)M$ provided $M>\frac{n-k}{2}$,
and completes the proof of Theorem \ref{Intro:L2ThmDeg}
\end{proof}
\begin{rem}
   The phase of the linearized maximal wave operator which is $\phi(x,\xi)=t(x)|\xi| + \langle x,\xi\rangle$, satisfies the assumptions of Theorem \ref{Intro:L2ThmDeg} since it belongs to $L^{\infty}\Phi^2$ and it also satisfies the rough corank condition~\ref{Rough corank condition}. Indeed if $\xi \in \mathbf{S}^{n-1}$ we can take $V_{x,\xi}=\xi^{\perp}$ and if $\pi_{\xi},\pi_{\xi^{\perp}}$ denote the projections onto $\mathop{\rm span} \xi$
   and $V_{x,\xi}$ respectively then it is clear that
   \begin{align*}
       |\d_{\xi}\phi(x,\xi)&-\d_{\xi}\phi(y,\xi)|^2 \\ &= |t(x)-t(y)|^2+|x-y|^2+2(t(x)-t(y)) \underbrace{\langle \xi,x-y \rangle}_{=\pm |\pi_{\xi}(x-y)|} \\
       &\geq \big|\pi_{\xi^{\perp}}(x-y)\big|^2+\big| |t(x)-t(y)|- |\pi_{\xi}(x-y)| \big|^2 \\
       &\geq |\pi_{\xi^{\perp}}(x-y)|^2.
   \end{align*}
Therefore, as mentioned earlier, the Fourier integral operators under consideration include the linearized maximal wave operator.
\end{rem}
 A consequence of this is a local $L^2$ boundedness result for Fourier integral operators with smooth phase functions and rough symbols.
\begin{cor}
    Suppose that $\phi(x,\xi)$ is a smooth phase function satisfying the non-degeneracy condition
    \begin{equation}
        \mathop{\rm rank} \frac{\d^2 \phi}{\partial x_{j} \partial \xi_{k}} \geq n-k, \quad \textrm{ on } \supp a
    \end{equation}
    and the entries of the Hessian matrix  have bounded derivatives with respect to both $x$ and $\xi$ separately.
    Assume also that the symbol $a$ belongs to $L^{\infty}S^{m}_{\varrho}, 0\leq \varrho \leq 1$. Then the associated Fourier integral operator is bounded
    from $L^2_{\rm comp}$ to  $L^2_{\rm loc}$ provided $m<-\frac{k}{2}+\frac{(n-k)(\varrho -1)}{2}$.
\end{cor}
This is sharp, for example in the case $k=0$ (i.e. pseudodifferential operators), since there exists $m_0$ with $m_0 > n(\varrho-\delta)/2$ such that the pseudodifferential operator
with symbol belonging to $S^{m_0}_{\varrho,\delta}$ is not bounded from $L^2_{\text{comp}}$ to $L^2_{\text{loc}}$, see \cite{H4}. Now since the phase of a pseudodifferential operator satisfies the condition of the above corollary with $k=0$ and since obviously $m_0 \geq n(\varrho -1)/2$ and $S^{m_0}_{\varrho,\delta}\subset L^{\infty}S^{m_0}_{\varrho},$ it follows that the above $L^2$ boundedness is sharp.
\subsubsection{Symplectic aspects of the $L^2$ boundedness}
\label{subsec:symplecticL2}
Here we shall discuss the symplectic aspects of the $L^2$ boundedness of Fourier integral operators which aims to highlight the essentially geometric nature of the problem of $L^2$ regularity of Fourier integral operators. We begin by recalling some of the well known $L^{2}$ continuity results in the case of smooth phases and amplitudes.
The kernel of the Fourier integral operator
\begin{align}
\label{L2:Fourier integral operator}
     Tu(x)  = (2\pi)^{-n} \int_{\R^n} e^{i\phi(x,\xi)} a(x,\xi) \widehat{u}(\xi) \, \dd \xi
\end{align}
is an oscillatory integral whose wave front set is contained in the closed subset of $\dot{T}^*\R^{2n}=T^*\R^{2n} \setminus 0$
\begin{align}
\label{L2:WFkernel}
      \mathop{\rm WF}(T) \subset \big\{ (x,\d_{x}\phi(x,\xi),\d_{\xi}\phi(x,\xi),-\xi) : (x,\xi) \in \supp a, \, \xi \neq 0 \big\}.
\end{align}
The cotangent space $T^*\R^n$ is endowed with the symplectic form
    $$ \sigma = \sum_{j=1}^n \dd\xi_j \wedge \dd x_j. $$
A canonical relation is a Lagrangian submanifold of the product $T^*\R^n \times T^*\R^n$ endowed with the symplectic form $\sigma \oplus (-\sigma)$, this means that the aforementioned symplectic form vanishes on the canonical relation.
In particular, by rearranging the terms in the closed cone \eqref{L2:WFkernel}, one obtains a canonical relation
\begin{align*}
     \mathcal{C}_{\phi} = \big\{ (x,\d_{x}\phi(x,\xi),\d_{\xi}\phi(x,\xi),\xi) : (x,\xi) \in \supp a \big\}
\end{align*}
in $T^*\R^n \times T^*\R^n$.
If $\mathcal{C}$ is a canonical relation, we consider the two maps $\pi_{1}: (x,\xi) \mapsto (x, \d_{x}\phi)$ and $\pi_{2}: (x,\xi)\mapsto(\d_{\xi} \phi , \xi),$
\begin{align*}
 \xymatrix @!0 @C=4pc @R=3pc {
    & \ar[ld]_{\pi_1} \mathcal{C} \subset T^*\R^n \times T^*\R^n  \ar[rd]^{\pi_2} &  \\
    T^*\R^n & & T^*\R^n.
  }
\end{align*}
The canonical relation $\mathcal{C}$ is (locally) the graph of a smooth function $\chi$ if and only if $\pi_1$ is a (local) diffeomorphism, and
in this case $\chi=\pi_2 \circ \pi_1^{-1}$. This function $\chi$ is a diffeomorphism if and only if $\pi_2$ is a diffeomorphism.
Note that if this is the case, $\chi$ is a symplectomorphism because the submanifold $\mathcal{C}$ is {\it{Lagrangian}} for the symplectic form, i.e.
$\sigma \oplus (-\sigma)$
   $$ \dd \xi \wedge \dd x - \dd \eta \wedge \dd y = 0 \quad \textrm{ when } (y,\eta) = \chi(x,\xi). $$
The canonical relation $\mathcal{C}_{\phi}$ is locally the graph of a symplectomorphism in the neighbourhood of
$(x_{0},\d_{x}\phi(x_{0},\xi_{0}),\d_{\xi}\phi(x_{0},\xi_{0}),\xi_{0})$ if and only if
\begin{align}
\label{L2:NonDegenerate}
    \det \frac{\d^{2}\phi}{\d x \d \xi}(x_{0},\xi_{0}) \neq 0.
\end{align}
It is well-known that the Fourier integral operators of order $0$ whose canonical relation $\mathcal{C}_{\phi}$ is
locally the graph of a symplectic transformation $\chi$, are locally $L^2$ bounded. More precisely
\begin{thm}
\label{L2:L2Fourier integral operator}
     Let $a \in S^0_{1,0}$ and $\phi$ be a real valued function in $C^{\infty}(\R^n \times \R^n \setminus 0)$ which is homogeneous
     of degree $1$ in $\xi$. Assume that the homogeneous canonical relation $\mathcal{C}_{\phi}$ is locally the graph%
     \footnote{Or equivalently that \eqref{L2:NonDegenerate} holds on $\supp a$.} of a symplectomorphism between
     two open neighbourhoods in $\dot{T}^*\R^n=T^*\R^n \setminus 0$. Then the Fourier integral operator \eqref{L2:Fourier integral operator} defines a bounded operator from
     $L^2_{\rm comp}$ to $L^2_{\rm loc}$.
\end{thm}
\begin{proof}
This is Theorem 25.3.1 in \cite{H2}.
\end{proof}
But in fact, there are boundedness results even when $\mathcal{C}$ is not the graph of a symplectomorphism, i.e. when either the projection
$\pi_1$ or $\pi_2$ is not a diffeomorphism. There is an important instance for which this is the case and one could still prove local
$L^2$ boundedness with loss of derivatives. A suggestive example for this situation is the restriction operator to a linear subspace $H=\big\{x=(x',x'') \in \R^n=\R^{n'} \times \R^{n''} : x''=0 \big\}$
\begin{align*}
   R_H u = \langle D \rangle^{m}u(x',0) = (2\pi)^{-n} \int e^{i \langle x',\xi'\rangle} \langle \xi \rangle^m \widehat{u}(\xi) \, \dd\xi
\end{align*}
where $m \leq 0$. We know that this operator is bounded from $L^2_{\rm comp}$ to $L^2_{\rm loc}$; indeed for all $a \in C^{\infty}_{0}(\R^n)$
there exists a constant $C_{m,n}$ such that
\begin{align*}
   \Vert aR_Hu\Vert_{L^2} \leq C_{m,n} \Vert u\Vert_{L^2}
\end{align*}
provided $m \leq -\mathop{\rm codim}H/2$. The canonical relation of the Fourier integral operator $R_H$ is given by
\begin{align*}
     \xymatrix @!0 @C=5pc @R=4pc
     {
     & \ar[ld]_{\pi_1} \mathcal{C}_H=\big\{(x,\xi',0;x',0,\xi), \, (x,\xi) \in T^*\R^n \big\}  \ar[rd]^{\pi_2} &  \\
     \big\{ \xi''=0 \big\} \subset T^*\R^n & &  \big\{ x''=0 \big\} \subset T^*\R^n.
     }
\end{align*}
By $\sigma_{\mathcal{C}_{H}}$ we denote the pullback of the symplectic form $\sigma,$ by $\pi_1 ,$ to $\mathcal{C}_{H}$ (of course we could equally well consider the pullback $\pi_{2}^*\sigma$ without changing anything)
    $$ \sigma_{\mathcal{C}_{H}} = \pi_1^* \sigma = \dd\xi' \wedge \dd x'. $$
Then we have
    $$ \mathop{\rm corank} \sigma_{\mathcal{C}_{H}} = 2n''=2\mathop{\rm codim} H $$
and the condition of $L^2$ boundedness is therefore $m \leq -\mathop{\rm corank} \sigma_{\mathcal{C}_{H}}/4$.
In fact, this example models the general situation, and this is Theorem~25.3.8 in \cite{H2}.
\begin{thm}
\label{L2:L2Fourier integral operatordeg}
     Let $a \in S^m_{1,0}$ and $\phi$ be a real valued function in $C^{\infty}(\R^n \times \R^n \setminus 0)$ which is homogeneous
     of degree $1$ in $\xi$ such that
     \footnote{This ensures that $\mathcal{C}_{\phi}$ is a homogeneous canonical relation to which the radial vectors of $\dot{T}^*\R^n \times 0$
     and $0 \times \dot{T}^*\R^n$ are never tangential.}
     $d\phi \neq 0$ on $\supp a$.
     Then the Fourier integral operator \eqref{L2:Fourier integral operator} defines a bounded operator from
     $L^2_{\rm comp}$ to $L^2_{\rm loc}$ provided $m \leq - \mathop{\rm corank} \sigma_{\mathcal{C}_{\phi}}/4$.
     Here $\sigma_{\mathcal{C}_{\phi}}$ is the two form on $\mathcal{C}_{\phi}$ obtained by lifting to $C_{\phi}$ the symplectic form $\sigma$ on $\dot{T}^*\R^n$
     by one of the projections $\pi_{1}$ or $\pi_{2}$.
\end{thm}
The fact that the canonical relation is parametrised by
     $$ F: (x,\xi) \mapsto (x,\d_{x}\phi(x,\xi),\d_{\xi}\phi(x,\xi),\xi) $$
allows us to compute
\begin{align*}
   F^*(\pi_{1}^* \sigma) &= \dd(\pi_{1}\circ F)^*(\xi \, \dd x) = \dd\big(\d_{x} \phi(x,\xi) \, \dd x\big) \\
   &= \underbrace{\sum_{j,k=1}^n\d^2_{x_{j}x_{k}}\phi(x,\xi) \, \dd x_{j} \wedge \dd x_{k}}_{=0} + \sum_{j,k=1}^n\d^2_{\xi_{j}x_{k}}\phi(x,\xi) \, \dd \xi_{j} \wedge \dd x_{k}.
\end{align*}
Therefore we have
\begin{align*}
     F^*\sigma_{\mathcal{C}_{\phi}} = \sum_{j,k=1}^n\d^2_{\xi_{j}x_{k}}\phi(x,\xi) \, \dd \xi_{j} \wedge \dd x_{k}
\end{align*}
which yields
\begin{align*}
     \mathop{\rm corank} \sigma_{\mathcal{C}_{\phi}} = 2 \mathop{\rm corank} \frac{\d^{2}\phi}{\d x \d \xi}.
\end{align*}
The geometric assumption in Theorem \ref{L2:L2Fourier integral operatordeg} (which is valid for general Fourier integral operators, not necessarily of the form \eqref{L2:Fourier integral operator}) is therefore
equivalent to
\begin{align}
\label{L2:CorankHyp}
     m \leq -\frac{1}{2} \mathop{\rm corank} \frac{\d^{2}\phi}{\d x \d \xi}.
\end{align}
\begin{rem}
     If the function $t(x)$ in the linearized maximal wave operator \eqref{Intro:LinWave} were smooth, then that operator would fall into the category of Fourier integral operators
      satisfying the assumptions of Theorem \ref{L2:L2Fourier integral operatordeg}. Indeed as already noted in the introduction, the corank of $\d^2 \phi/\d x\d \xi$ when $\phi(x,\xi)=t(x)|\xi|+\langle x,\xi \rangle$
     is at most $1$. Therefore $e^{it(x)\sqrt{-\Delta}}$ defines a bounded operator from $H^{1/2}_{\rm comp}$ to $L^2_{\rm loc}$ when $t(x)$ is a smooth function on $\R^n$.
\end{rem}
Theorem \ref{Intro:L2Thm} for $\varrho=1$ is the non-smooth analogue of Theorem \ref{L2:L2Fourier integral operator} where the non-degeneracy condition \eqref{L2:NonDegenerate}
which requires smoothness in $x$ has been replaced by Definition \ref{defn of rough nondegeneracy}. Note nevertheless that Theorem \ref{Intro:L2Thm} is a \textit{global} $L^2$
result. Similarly Theorem \ref{Intro:L2ThmDeg} for $\varrho=1$ is the non-smooth analogue of Theorem \ref{L2:L2Fourier integral operatordeg} with \eqref{L2:CorankHyp} replaced by assumption  \ref{Rough corank condition}.

\subsection{Global $L^\infty$ boundedness of rough Fourier integral operators}
In this section, we establish the $L^\infty$ boundedness of Fourier integral operators. To prove the $L^\infty$ boundedness of the high frequency portion of the operator, we need to use the semiclassical estimates of Subsection \ref{SSS decomposition}. However, using only the Seeger-Sogge-Stein decomposition yields a loss of derivatives no better than $m<-\frac{n-1}{2}+ n(\varrho-1)$, and to obtain the sharp $L^\infty$ boundedness result claimed in Theorem \ref{Intro:LinftyThm}, further analysis is needed.
\begin{thm}
\label{Intro:LinftyThm}
   Let $T$ be a Fourier integral operator given by \eqref{Intro:Fourier integral operator} with amplitude $a \in L^{\infty}S^m_{\varrho}$
   and phase function $\phi \in L^{\infty}\Phi^2.$ Then there exists a constant $C>0$ such that
      $$ \Vert Tu\Vert_{L^{\infty}} \lesssim \Vert u\Vert_{L^{\infty}}, \quad u \in \S(\R^n) $$
   provided $m<-\frac{n-1}{2} +\frac{n}{2}(\varrho-1)$ and $0\leq \varrho\leq 1$. Furthermore, this result is sharp.
\end{thm}
\begin{proof}
  As a first step, we use the semiclassical reduction of Subsection \ref{Semiclasical reduction subsec} to decompose $T$ into $T_0$ and $T_h$. Thereafter we split the semiclassical piece $T_h$ further into $\sum_{\nu=1}^{J} T_{h}^{\nu}$ using the Seeger-Sogge-Stein decomposition of Subsection \ref{SSS decomposition} applied to the amplitude $a(x,\xi)$ and the phase $\varphi(x,\xi).$ Once again, the boundedness of $T_0$ follows from Theorem \ref{general low frequency boundedness for rough Fourier integral operator}, but here we don't need the rough non-degeneracy of the phase function due to the fact that we are dealing with the $L^\infty$ boundedness of $T_0$ which only requires that the integral with respect to the $y$ variable of the Schwartz kernel $T_0(x,y)$ being finite. See Theorem \ref{general low frequency boundedness for rough Fourier integral operator} for further details.

From equation \ref{kernel of Thnu} one deduces that the kernel of the semiclassical high frequency operator $T_{h}^{\nu}$ is given by
$$ T_{h}^{\nu}(x,y)=(2\pi h)^{-n} \int_{\R^n} e^{\frac{i}{h}\langle \nabla_{\xi}\varphi(x,\xi^{\nu})-y,\xi \rangle} b^{\nu}(x,\xi,h) \, \dd \xi,$$
with
\begin{equation*}
   b^{\nu}(x,\xi,h)=e^{\frac{i}{h}\langle \nabla_{\xi}\varphi(x,\xi)-\nabla_{\xi}\varphi(x,\xi^{\nu}),\xi \rangle}\psi^{\nu}(\xi)  \chi(\xi) a(x,\xi/h).
\end{equation*}
Now since
 $$ \Vert T_h^{\nu} u\Vert_{L^{\infty}} \leq  \Vert u\Vert_{L^{\infty}}\int | T_h^{\nu}(x,y)|\,\dd y, $$
it remains to show a suitable estimate for $\int | T_h^{\nu}(x,y)|\,dy$. As in the proof of $L^1$ boundedness, we use the differential operator
\begin{align*}
     L = 1-\d_{\xi_{1}}^2-h \d_{\xi'}^2
\end{align*}
for which we have according to Lemma \ref{Linfty:bLemma}
\begin{equation}\label{estimate on bnu}
     \sup_{\xi} \Vert L^{N}b^{\nu}(\cdot,\xi,h)\Vert_{L^{\infty}} \lesssim  h^{-m-2N(1-\varrho)}.
\end{equation}
Setting
\begin{equation}\label{metric g}
g(z)=h^{-2}z_{1}^2+h^{-1}|z'|^{2},
\end{equation}
we have
\begin{equation*}
     L^N\,e^{\frac{i}{h}\langle \nabla_{\xi}\varphi (x,\xi^{\nu})-y,\xi \rangle} =\big(1+g(y-\nabla_{\xi}\varphi (x,\xi^{\nu})\big)^{N}\,
     e^{\frac{i}{h}\langle \nabla_{\xi}\varphi (x,\xi^{\nu})-y,\xi \rangle}
\end{equation*}
for all integers $N$.
Now we observe that
 \begin{align*}
(2\pi h)^{\frac{n}{2}}\int\vert T_h^{\nu}(x,y)\vert\, \dd y &= (2\pi h)^{\frac{n}{2}}\int\vert T_h^{\nu}(x,y+\nabla_{\xi}\varphi (x,\xi^{\nu}))\vert\, \dd y \\ \nonumber
&= \int |\widehat{b^{\nu}} (x,y,h)| \, \dd y  \\ \nonumber&=\int_{\sqrt{g(y)}\leq h^{\varrho}}+\int_{\sqrt{g(y)}> h^{\varrho}}\vert \widehat{b^{\nu}} (x,y,h)\vert\, \dd y
:=\textbf{I}_1 + \textbf{I}_2 ,
\end{align*}
where
    $$ \widehat{b^{\nu}} (x,y,h)=(2\pi h)^{-\frac{n}{2}}\int e^{-\frac{i}{h}\langle y,\xi\rangle}\,b^{\nu}(x,\xi,h)\, \dd \xi $$
is the semiclassical Fourier transform of $b^\nu$.
To estimate $\textbf{I}_1$ we use the Cauchy-Schwarz inequality, the semiclassical Plancherel theorem, the definition of $g$ in
\eqref{metric g} and \eqref{estimate on bnu}. Hence remembering the fact that the measure of the $\xi$-support of $b^{\nu}(x,\xi, h)$ is $O(h^{\frac{(n-1)}{2}})$ we have
\begin{align*}
\textbf{I}_1 &\leq \bigg\{\int_{\sqrt{g(y)}\leq h^{\varrho}}\dd y\bigg\}^{\frac{1}{2}}\bigg\{\int\vert\widehat{b^{\nu}} (x,y,h)\vert^{2} \dd y\bigg\}^{\frac{1}{2}} \\
&\lesssim h^{\frac{n+1}{4}}\bigg\{\int_{|y| \leq h^\varrho}\dd y\bigg\}^{\frac{1}{2}}\bigg\{\int |b^{\nu} (x,\xi,h)|^2 \, \dd\xi\bigg\}^{\frac{1}{2}} \\
&\lesssim h^{\frac{n+1}{4}} h^{\frac{n\varrho}{2}} h^{-m+\frac{n-1}{4}} \lesssim h^{\frac{n}{2}} h^{-m+\frac{n\varrho}{2}}.
\end{align*}
Before we proceed with the estimate of $\textbf{I}_2$, we observe that if $l$ is a non-negative integer then the semiclassical Plancherel theorem and \eqref{estimate on bnu} yield
\begin{align}\label{integer l estimates for I2}
\bigg(\int\vert \widehat{b^{\nu}} (x,y,h)\vert^{2}(1+g(y))^{2l}\,\dd y\bigg)^{\frac{1}{2}} &\leq \bigg(\int |L^{l} b^{\nu} (x,\xi, h)|^2\, \dd \xi\bigg)^{\frac{1}{2}}
\\ \nonumber &\leq h^{-m-2l(1-\varrho)+\frac{n-1}{4}}.
\end{align}
Moreover, any positive real number $l$ which is not an integer can be written as $[l]+\{l\}$ where $[l]$ denotes the integer part of $l$ and $\{l\}$ its fractional part, which is $0<\{l\}<1$. Therefore, H\"older's inequality with conjugate exponents $\frac{1}{\{l\}}$ and $\frac{1}{1-\{l\}}$ yields
\begin{multline*}
\int \vert \widehat{b^{\nu}}(x,y,h)\vert^{2}(1+g(y))^{2l}\,\dd y \\ =
\int\vert \widehat{b^{\nu}} \vert^{2\{l\}}\, \vert \widehat{b^{\nu}}\vert^{2(1-\{l\})} (1+g(y))^{2\{l\}([l]+1)} \, (1+g(y))^{2[l](1-\{l\})} \,\dd y \\
\leq \bigg(\int \vert \widehat{b^{\nu}} \vert^{2} (1+g(y))^{2([l]+1)} \, \dd y\bigg)^{\{l\}} \bigg(\int \vert \widehat{b^{\nu}} \vert^{2}(1+g(y))^{2[l]}
\, \dd y\bigg)^{1-\{l\}}.
\end{multline*}
Therefore, using \eqref{integer l estimates for I2} we obtain
\begin{align*}
\bigg(\int \vert \widehat{b^{\nu}} &(x,y,h)\vert^{2}(1+g(y))^{2l}\, \dd y\bigg)^{\frac{1}{2}} \\ \nonumber
&\leq \bigg(\int |L^{[l]+1} b^{\nu} (x,\xi, h)|^2\, \dd \xi\bigg)^{\frac{\{l\}}{2}} \bigg(\int |L^{[l]} b^{\nu} (x,\xi, h)|^2\, \dd\xi\bigg)^{\frac{1-\{l\}}{2}} \\ \nonumber
&\leq h^{\{l\}(-m-2([l]+1)(1-\varrho)+\frac{n-1}{4})}h^{(1-\{l\})(-m-2[l](1-\varrho)+\frac{n-1}{4})} \\ \nonumber
&\leq h^{-m-2l(1-\varrho)+\frac{n-1}{4}},
\end{align*}
and hence \eqref{integer l estimates for I2} is actually valid for all non-negative real numbers $l$.
Turning now to the estimates for $\textbf{I}_2 ,$ we use the same tools as in the case of $\textbf{I}_1$ and \eqref{integer l estimates for I2} for $l\in [0,\infty)$. This yields for any $l>\frac{n}{4}$
\begin{align*}
    \textbf{I}_2 &\leq \bigg\{\int_{\sqrt{g(y)}>  h^{\varrho}}(1+g(y))^{-2l}\, \dd y\bigg\}^{\frac{1}{2}}
    \times\bigg\{\int\vert \widehat{b^{\nu}} (x,y,h)\vert^{2}(1+g(y))^{2l} \, \dd y\bigg\}^{\frac{1}{2}} \\ \nonumber
    &\lesssim h^{\frac{n+1}{4}}\bigg\{\int_{|y|> h^{\varrho}}|y|^{-4l} \, \dd y\bigg\}^{\frac{1}{2}} h^{-m-2l(1-\varrho)+\frac{n-1}{4}} \\ \nonumber
    &\lesssim h^{\frac{n+1}{4}} h^{\varrho(\frac{n}{2}-2l)} h^{-m-2l(1-\varrho)+\frac{n-1}{4}}\lesssim h^{\frac{n}{2}} h^{-m+\frac{n\varrho}{2}-2l}.
\end{align*}
Therefore
\begin{align}\label{integral of Thnu dy}
     \sup_{x} \int |T_h^{\nu}(x,y)| \, \dd y \leq C_{l}h^{-m+\frac{n\varrho}{2}-2l}
\end{align}
and summing in $\nu$ yields
    $$ \Vert T_h u\Vert_{L^{\infty}} \leq  \sum_{\nu=1}^J \Vert T_h^{\nu} u\Vert_{L^{\infty}} \leq C_{l} h^{-m+\frac{n\varrho}{2}-2l-\frac{n-1}{2}}\Vert u\Vert_{L^{\infty}}, $$
since $J$ is bounded (from above and below) by a constant times $h^{-\frac{n-1}{2}}$. By Lemma \ref{Lp:semiclassical} one has
    $$ \Vert T u\Vert_{L^{\infty}} \lesssim \Vert u\Vert_{L^{\infty}} $$
provided $m<-\frac{n-1}{2}+\frac{n\varrho}{2}-2l$ and $l>\frac{n}{4}$, i.e. if $m<-\frac{n-1}{2}+\frac{n}{2}(\varrho -1)$. This completes the proof of Theorem \ref{Intro:LinftyThm}.
\end{proof}
\subsection{Global $L^p$-$L^p$ and $L^p$-$L^q$ boundedness of Fourier integral operators}
In this section we shall state and prove our main boundedness results for Fourier integral operators. Here, we prove results both for smooth and rough operators with phases satisfying various non-degeneracy conditions. As a first step, interpolation yields the following global $L^p$ results:
\begin{thm}\label{main L^p thm for smooth Fourier integral operators}
      Let $T$ be a Fourier integral operator given by \eqref{Intro:Fourier integral operator} with amplitude $a \in S^m_{\varrho, \delta}, 0\leq \varrho \leq 1,$ $0\leq \delta \leq 1,$ and a phase function $\varphi(x,\xi) \in \Phi^2$ satisfying the strong non-degeneracy condition. Setting $\lambda:=\min(0,n(\varrho-\delta)),$ suppose that either of the following conditions hold:
      \begin{enumerate}
\item[$(a)$]  $1 \leq p \leq 2$ and
                      $$ m<n(\varrho -1)\bigg (\frac{2}{p}-1\bigg)+\big(n-1\big)\bigg(\frac{1}{2}-\frac{1}{p}\bigg)+ \lambda\bigg(1-\frac{1}{p}\bigg); $$ or
\item[$(b)$]  $2 \leq p \leq \infty$ and
                      $$ m<n(\varrho -1)\bigg (\frac{1}{2}-\frac{1}{p}\bigg)+ (n-1)\bigg(\frac{1}{p}-\frac{1}{2}\bigg) +\frac{\lambda}{p};$$ or
\item[$(c)$]   $p=2,$ $0\leq \varrho\leq 1,$ $0\leq \delta<1,$ and
$$m= \frac{\lambda}{2}.$$
\end{enumerate}
Then there exists a constant $C>0$ such that $ \Vert Tu\Vert_{L^{p}} \leq C \Vert u\Vert_{L^{p}}.$
\end{thm}
\begin{proof}
  The proof is a direct consequence of interpolation of the the $L^1$ boundedness result of Theorem \ref{Intro:L1Thm} with the $L^2$ boundedness of Theorem \ref{Calderon-Vaillancourt for FIOs} on one hand, and the interpolation of the latter with the $L^\infty$ boundedness result of Theorem \ref{Intro:LinftyThm}. The details are left to the reader.
\end{proof}
\begin{thm}\label{main L^p thm for Fourier integral operators with smooth phase and rough amplitudes}
      Let $T$ be a Fourier integral operator given by \eqref{Intro:Fourier integral operator} with amplitude $a \in L^{\infty}S^m_{\varrho}, 0\leq \varrho \leq 1$ and a strongly non-degenerate phase function $\varphi(x,\xi) \in \Phi^2.$ Suppose that either of the following conditions hold:
      \begin{enumerate}
\item[$(a)$]  $1 \leq p \leq 2$ and
                      $$ m<\frac{n}{p}(\varrho -1)+\big(n-1\big)\bigg(\frac{1}{2}-\frac{1}{p}\bigg); $$ or
\item[$(b)$]  $2 \leq p \leq \infty$ and
                      $$ m<\frac{n}{2}(\varrho -1)+ (n-1)\bigg(\frac{1}{p}-\frac{1}{2}\bigg).$$
\end{enumerate}
Then there exists a constant $C>0$ such that $ \Vert Tu\Vert_{L^{p}} \leq C \Vert u\Vert_{L^{p}}.$
\end{thm}
\begin{proof}
  The proof follows once again from interpolation of the $L^1$ boundedness result of Theorem \ref{Intro:L1Thm} with the $L^2$ boundedness of Theorem \ref{global L2 boundedness smooth phase rough amplitude} on one hand, and the interpolation of the latter with the $L^\infty$ boundedness result of Theorem \ref{Intro:LinftyThm}.
\end{proof}
As an immediate consequence of the theorem above one has
\begin{cor}\label{LinftySm1 cor}
  For a Fourier integral operator $T$ with amplitude $a \in L^{\infty}S^m_{1}$
      and a strongly non-degenerate phase function $\varphi(x,\xi) \in \Phi^2,$ one has $L^p$ boundedness for $p\in [1,\infty]$ provided $m<-(n-1)|\frac{1}{p}-\frac{1}{2}|.$
\end{cor}

Using Sobolev embedding theorem one can also show the following $L^{p}-L^{q}$ estimates for rough Fourier integral operators.

\begin{thm}\label{main LpLq thm for Fourier integral operators}
Suppose that
 \begin{enumerate}
\item [$(1)$]  $T$ is a Fourier integral operator with an amplitude $a \in S^m_{\varrho, \delta},$ $0\leq \varrho \leq 1,$ $0\leq \delta \leq 1$ and a strongly non-degenerate phase function $\varphi(x,\xi) \in \Phi^2,$ with either of the following conditions:
 \begin{enumerate}
\item [$(a)$]  $1 \leq p\leq q \leq 2$ and $$m<n(\varrho -1)\bigg (\frac{2}{q}-1\bigg)- (n-1)\bigg(\frac{1}{p}-\frac{1}{2}\bigg)+ \lambda\bigg(1-\frac{1}{q}\bigg)+\frac{1}{q}-\frac{1}{p};$$ or
\item [$(b)$]  $2 \leq p \leq q\leq \infty$ and $$m<n(\varrho -1)\bigg (\frac{1}{2}-\frac{1}{q}\bigg)+ (n-1)\bigg(\frac{2}{q}-\frac{1}{p}-\frac{1}{2}\bigg)+\frac{\lambda}{q}+\frac{1}{q}-\frac{1}{p}.$$
 \end{enumerate}
\item [$(2)$] $T$ is a Fourier integral operator with an amplitude $a \in L^{\infty}S^m_{\varrho}, 0\leq \varrho \leq 1$ and a strongly non-degenerate phase function $\varphi(x,\xi) \in \Phi^2,$ with either of the following conditions:
 \begin{enumerate}
\item [$(a)$]  $1 \leq p\leq q \leq 2$ and $$m<\frac{n}{q}(\varrho -1)- (n-1)\bigg(\frac{1}{p}-\frac{1}{2}\bigg)+\frac{1}{q}-\frac{1}{p};$$ or
\item [$(b)$] $2 \leq p \leq q\leq \infty$ and $$m<\frac{n}{2}(\varrho -1)+ (n-1)\bigg(\frac{2}{q}-\frac{1}{p}-\frac{1}{2}\bigg)+\frac{1}{q}-\frac{1}{p}.$$
\end{enumerate}
      \end{enumerate}
     Then there exists a constant $C>0$ such that $\Vert Tu\Vert _{L^{q}} \leq C\Vert u\Vert_{L^{p}}.$
\end{thm}
\begin{proof}
We give the details of the proof only for $(1)$ (a). The rest of the proof is similar to that of $(1)$(a), through the use of Theorem \ref{main L^p thm for smooth Fourier integral operators} part (b) or Theorem \ref{main L^p thm for Fourier integral operators with smooth phase and rough amplitudes}.
Condition $m<n(\varrho -1)(\frac{2}{q}-1)- (n-1)(\frac{1}{p}-\frac{1}{2})+ \lambda (1-\frac{1}{q})+\frac{1}{q}-\frac{1}{p}$ yields the existence of
    of a real number $s$ with
\begin{equation}\label{bounds on s}
n\bigg(\frac{1}{p}-\frac{1}{q}\bigg)\leq s<n(\varrho -1)\bigg (\frac{2}{q}-1\bigg)+\big(n-1\big)\bigg(\frac{1}{2}-\frac{1}{q}\bigg)+ \lambda\bigg(1-\frac{1}{q}\bigg)-m.
\end{equation}
Therefore writing $T= T(1-\Delta)^{\frac{s}{2}}(1-\Delta)^{-\frac{s}{2}}$, Leibniz rule reveals that the amplitude of $T(1-\Delta)^{\frac{s}{2}}$ belongs to $L^{\infty}S^{m+s}_{\varrho}$ and since
    $$ m+s<n(\varrho -1)\bigg (\frac{2}{q}-1\bigg)+\big(n-1\big)\bigg(\frac{1}{2}-\frac{1}{q}\bigg)+ \lambda\bigg(1-\frac{1}{q}\bigg), $$
Theorem \ref{main L^p thm for smooth Fourier integral operators} part (a) yields
\begin{align*}
\Vert T u\Vert _{L^{q}} = \Vert T(1-\Delta)^{\frac{s}{2}}(1-\Delta)^{-\frac{s}{2}} u\Vert _{L^{q}} \lesssim \Vert (1-\Delta)^{-\frac{s}{2}} u\Vert_{L^{q}} \lesssim \Vert u\Vert_{L^{p}},
\end{align*}
where the very last estimate is a direct consequence of \eqref{bounds on s} and the Sobolev embedding theorem. Hence $\Vert Tu\Vert _{L^{q}} \lesssim \Vert u\Vert_{L^{p}}$ for the above ranges of $p$, $q$ and $m$, and the proof is complete.
\end{proof}

\section{Global and local weighted $L^p$ boundedness of Fourier integral operators}
The purpose of this chapter is to establish boundedness results for a fairly wide class of Fourier integral operators on weighted $L^p$ spaces with weights belonging to Muckenhoupt's $A_p$ class. We also prove these results for Fourier integral operators whose phase functions and amplitudes are only bounded and measurable in the spatial variables and exhibit suitable symbol type behavior in the frequency variable. We will start by recalling some facts from the theory of $A_p$ weights which will be needed in this section. Thereafter we prove a couple of uniform stationary phase estimates for oscillatory integrals and then proceed with the weighted boundedness for the low frequency portions of Fourier integral operators. Before proceeding with our claims about the weighted boundedness of the high frequency part of Fourier integral operators, a discussion of a counterexample leads us to a rank condition on the phase function $\varphi(x,\xi)$ which is crucial for the validity of the weighted boundedness (with $A_p$ weights) of Fourier integral operators. Using interpolation and extrapolation, we can prove an endpoint weighted $L^p$ boundedness theorem for operators within a specific class of amplitudes and all $A_p$ weights, which is shown to be sharp in a case of particular interest and can also be invariantly formulated in the local case. Finally we show the $L^p$ boundedness of a much wider class of operators for some subclasses of the $A_p$ weights.
\subsection{Tools in proving weighted boundedness}
The following results are well-known and can be found in their order of appearance in \cite{GR}, \cite{J} and \cite{S}.
\begin{thm} \label{open}
Suppose $p > 1$ and $w \in A_p$. There exists an exponent $q < p$, which depends only on $p$ and $[w]_{A_p}$, such that $w \in A_q$.
There exists $\varepsilon > 0$, which depends only on $p$ and $[w]_{A_p}$, such that $w^{1+\varepsilon} \in A_p$.
\end{thm}
\begin{thm} \label{maxweight}
For $1 < q < \infty$, the Hardy-Littlewood maximal operator is bounded on $L^q_w$ if and only if $w \in A_q$. Consequently, for $1 \leq p <
\infty$, $M_p$ is bounded on
$L^p_w$ if and only if $w \in A_{q/p}$
\end{thm}
\begin{thm} \label{convolve}
Suppose that $\phi \colon \R^n \to \R$ is integrable non-increasing and radial. Then, for $u \in L^1$, we have
\[
\int \phi(y)u(x-y) \, \dd y \leq \Vert\phi\Vert_{L^1} Mu(x)
\]
for all $x \in \R^n$.
\end{thm}
The following result of J.Rubio de Francia is also basic in the context of weighted norm inequalities.
 \begin{thm}[Extrapolation Theorem]\label{extrapolation}
If $\Vert Tu\Vert_{L^{p_{0}}_{w}} \leq C \Vert u\Vert_{L^{p_{0}}_{w}}$ for some fixed $p_0\in (1,\infty)$ and all $w\in A_{p_0}$, then one has in fact $\Vert Tu\Vert_{L^p_{w}} \leq C \Vert u\Vert_{L^p_{w}}$ for all $p\in(1,\infty)$ and $w\in A_{p}.$
\end{thm}
\subsubsection{A pointwise uniform bound on oscillatory integrals}
Before we proceed with the main estimates we would need a stationary-phase estimate which will enable us to control certain integrals depending on various parameters uniformly with respect to those parameters. Here and in the sequel we denote the Hessian in $\xi$ of the phase function $\varphi(x,\xi)$ by $\partial^2_{\xi\xi}\varphi(x,\xi)$.
\begin{lem} \label{UniformStationaryPhase}
For $\lambda\geq 1$, let $a_{\lambda}(x,\xi)\in L^{\infty} S^{0}_{0}$ with seminorms that are uniform in $\lambda$ and let $\supp_{\xi} a_{\lambda}(x,\xi) \subset B(0, \lambda ^{\mu})$ for some $\mu\ \ge 0$. Assume that $\varphi(x,\xi)\in L^{\infty}S^{0}_{0}$ and $|\det \partial^2_{\xi\xi}\varphi(x,\xi)|\geq c>0$ for all $(x,\xi) \in \supp a_{\lambda}$. Then one has
\begin{equation}\label{uniform stationary phase}
\sup_{x\in \mathbf{R}^n} \Big| \int e^{i\lambda\varphi(x,\xi)}\, a_{\lambda}(x,\xi)\, \dd \xi \Big| \lesssim \lambda^{n\mu-\frac{n}{2}}
\end{equation}
\end{lem}
\begin{proof}
      We start with the case $\mu=0$. The matrix inequality $\Vert A^{-1}\Vert \leq C_n |\det A|^{-1} \Vert A\Vert^{n-1}$ with
      $A =\partial^2_{\xi\xi}\varphi(x,\xi)$ and the assumptions on $\varphi,$ yield the uniform bound (in $x$ and $\xi$)
      \begin{equation}\label{uniform estim for the inverse hessian}
           \big\Vert [\partial^2_{\xi\xi}\varphi(x,\xi)]^{-1} \big\Vert \leq C_n |\det \partial^2_{\xi\xi}\varphi(x,\xi)|^{-1}
           \big\Vert \partial^2_{\xi\xi}\varphi(x,\xi)\big\Vert^{n-1}\lesssim 1.
      \end{equation}
       Looking at the map $\kappa_{x}: \xi\mapsto \nabla_{\xi} \varphi(x,\xi)$, we observe that $D\kappa_{x}(\xi)=\partial^2_{\xi\xi}\varphi(x,\xi)$,
       where $D\kappa_{x}(\xi)$ denotes the Jacobian matrix of the map $\kappa_{x}$, and that $\kappa_{x}$ is a diffeomorphism due to the condition on $\varphi$ in the lemma. Therefore
           $$D\kappa^{-1}_{x}(\tilde{\xi})= \big[\partial^2_{\xi\xi}\varphi(x,\kappa^{-1}_{x}(\tilde{\xi}))\big]^{-1},$$
       which using \eqref{uniform estim for the inverse hessian} yields uniform bounds for $\Vert D\kappa^{-1}_{x}(\tilde{\xi})\Vert$, hence
           $$|\kappa^{-1}_{x}(\tilde{\xi})-\kappa^{-1}_{x}(\tilde{\eta})|\leq \big\Vert D\kappa^{-1}_{x} \big\Vert \times |\tilde{\xi}- \tilde{\eta}|
               \lesssim |\tilde{\xi}- \tilde{\eta}|.$$
       This applied to $\tilde{\xi}=\kappa_x(\xi)$, $\tilde{\eta}=\kappa_x(\eta)$ implies that
       \begin{align}
       \label{uniform lower bound}
             |\xi -\eta| \lesssim |\kappa_x(\xi)-\kappa_x(\eta)| = |\nabla_{\xi}\phi(x,\xi)-\nabla_{\xi}\phi(x,\eta)|.
       \end{align}
       We set
           $$ I(\lambda,x): =   \int e^{i\lambda\varphi(x,\xi)}\, a_{\lambda}(x,\xi)\, \dd \xi $$
       and compute
       \begin{align*}
             |I(\lambda,x)|^2 = \iint e^{i\lambda (\varphi(x,\xi)-\varphi(x,\xi+\eta))}\, a_{\lambda}(x,\xi)\,
             \overline{a_{\lambda}}(x,\xi+\eta)\, \dd \xi  \, \dd \eta.
       \end{align*}
      We decompose the integral in $\eta$ into two integrals, one on $|\eta|\leq\delta$ and the other on $|\eta|>\delta$, and this yields the estimate
       \begin{multline*}
             |I(\lambda,x)|^2  \lesssim \\ \delta^n +
            \int_{\delta}^{\infty} \int_{\mathbf{S}^{n-1}} \bigg|\int  e^{i\lambda r \frac{\varphi(x,\xi)-\varphi(x,\xi+r\theta)}{r}}
            \, a_{\lambda}(x,\xi)\, \overline{a_{\lambda}}(x,\xi+r\theta)\, \dd \xi  \bigg| \, \dd \theta \, r^{n-1}  \dd r.
       \end{multline*}
        Using the uniform lower bound on the gradient of the phase in \eqref{uniform lower bound}, we get the uniform lower bound
            $$  \bigg|\frac{\nabla_{\xi}\phi(x,\xi)-\nabla_{\xi}\phi(x,\xi+r\theta)}{r}\bigg| \gtrsim 1.$$
      Therefore, applying the non-stationary phase estimate of \cite[Theorem 7.7.1]{H1} to the right-hand side integral yields
       \begin{align*}
            |I(\lambda,x)|^2 &\lesssim \delta^n + \lambda^{-n-1}  \int_{\delta}^{\infty} r^{-n-1}  \, r^{n-1}\dd r \\
            &\lesssim \delta^n + \delta^{-1} \lambda^{-n-1}.
        \end{align*}
        We now optimize this estimate by choosing $\delta=\lambda^{-1}$ and obtain the bound
             $$  |I(\lambda,x)| \leq C \lambda^{\frac{-n}{2}}, $$
        with a constant uniform in $x$.

        In the case $\mu>0$, we cover the ball $B(0, \lambda ^{\mu})$ with balls of radius $1$ and in doing that, one would need $O(\lambda^{n\mu})$
        balls of unit radius. This will obviously provide a covering of the $\xi$ support of $a_{\lambda}$ with balls of radius 1 and we take a finite smooth
        partition of unity $\theta_{j}(\xi),$ $j=1, \, \dots,\, O(\lambda^{n\mu}),$ subordinate to this covering with
        $|\partial^{\alpha}_{\xi} \theta_{j}|\leq C_{\alpha}$. Now by the first part of this proof we have
   \begin{equation}\label{partition pieces}
       \Big|\int e^{i\lambda\varphi(x,\xi)}\, a_{\lambda}(x,\xi)\,\theta_{j}(\xi)\, \dd \xi \Big|\leq C \lambda^{-\frac{n}{2}}
   \end{equation}
   with a constant that is uniform in $x$ and $j$. Finally summing in $j$ and remembering that there are roughly $O(\lambda^{n\mu})$ terms involved
   yields the desired estimate.
\end{proof}

\subsubsection{Weighted local and global low frequency estimates}
For the low frequency portion of the Fourier integral operators studied in this section we are once again able to handle the $L^p$ boundedness for all $p\in[1,\infty]$, using Lemma \ref{main low frequency estim} and imposing suitable conditions on the phases.
\begin{prop}\label{weighted low frequency estimates}
Let $\varrho \in [0,1]$ and suppose either:
\begin{itemize}
\item[(a)] $a(x,\xi)\in L^{\infty}S^{m}_{\varrho}$ is compactly supported in the $x$ variable, $m\in\mathbf{R}$ and $\varphi (x,\xi) \in L^{\infty}\Phi^1$; or
\item[(b)] $a(x,\xi)\in L^{\infty}S^{m}_{\varrho}$, $m\in \mathbf{R}$ and $\varphi(x,\xi)-\langle x,\xi\rangle\in L^{\infty}\Phi^1$
\end{itemize}
Then for all $\chi_{0}(\xi)\in C^{\infty}_{0}$ supported near the origin, the Fourier integral operator
     $$ T_{0}u(x)=\frac{1}{(2\pi)^n} \int a(x,\xi)\,\chi_{0}(\xi)\, e^{i\varphi (x,\xi)}\,\widehat{u}(\xi)\, \dd\xi$$
is bounded on $L^{p}_{w}$ for $1<p<\infty$ and all $w\in A_p$.
\end{prop}
\begin{proof}
\setitemize[0]{leftmargin=0pt,itemindent=20pt}
\begin{itemize}
\item[(a)] The operator $T_0$ can be written as $T_0 u(x)= \int K_{0}(x,y) \, u(x-y) \, \dd y$ with
     $$ K_{0}(x,y)= \frac{1}{(2\pi)^n} \int e^{i\psi(x,\xi)-i\langle y,\xi\rangle}\chi_{0}(\xi)\, a(x,\xi) \, \dd \xi , $$
where $\psi(x,\xi):= \varphi(x,\xi)-\langle x,\xi\rangle$ satisfies the estimate
$$\sup_{|\xi|\neq 0} |\xi|^{-1+|\alpha|} |\partial_{\xi}^{\alpha} \psi(x, \xi)|\leq C_\alpha,$$
for $|\alpha|\geq 1$, on support of the amplitude $a$.
Therefore setting $b(x,\xi):=a(x,\xi) \chi_{0}(\xi) e^{i\psi(x,\xi)}$ we have that $b$ is bounded and $$\sup_{|\xi|\neq 0} |\xi|^{-1+|\alpha|} |\partial_{\xi}^{\alpha} b(x, \xi)|<\infty,$$ for $|\alpha|\geq 1$ uniformly in $x$ and using Lemma \ref{main low frequency estim}, we have for all $\mu \in [0,1)$
\begin{equation}
  |K^{l}_{0} (x,y)|\leq C_{l} \langle y\rangle ^{-n-\mu},
\end{equation}
for all $x$. From this and Theorem \ref{convolve}, it follows that $|T_0 u(x)|\lesssim Mu(x)$ and Theorem \ref{maxweight} yields the $L_{w}^{p}$ boundedness of $T_0$.
\item[(b)] The only difference from the local case, is that instead of the assumption of compact support in $x$, the assumption $\varphi(x,\xi)-\langle x,\xi\rangle\in L^{\infty}\Phi^1$ yields that $b(x,\xi)$ in the previous proof satisfies the very same estimate, whereupon the same argument will conclude the proof.
\end{itemize}
\end{proof}

\subsection{Counterexamples in the context of weighted boundedness}\label{counterexamples in weighted setting}
The following counterexample going back to \cite{KW}, shows that for smooth Fourier integral operators (smooth phases and well as amplitudes), the non-degeneracy of the phase function i.e. the non-vanishing of the determinant of the mixed hessian of the phase, is not enough to yield weighted $L^p$ boundedness, unless one is prepared to loose a rather unreasonable amount of derivatives.
\begin{cex}
Let $\varphi(x,\xi)= \langle x, \xi\rangle +\xi_1$ which is non-degenerate but $$\mathrm{rank}\, \partial^{2}_{\xi\xi}\varphi=0,$$ and let $a(x,\xi)= \langle\xi\rangle^{m}$ with $-n<m<0$. Then it has been shown in \cite{Yab} that for $1<p<\infty$ there exists $w\in A_p$ and $f\in L^{p}_{w}$ such that the Fourier integral operator $Tu(x)= (2\pi)^{-n} \int e^{i\langle x, \xi\rangle +i\xi_1} \langle\xi\rangle^{m} \widehat{u}(\xi) \, \dd\xi,$ does not belong to $L^{p}_{w}$.
\end{cex}
However, as the following proposition shows, even with a phase of the type above, one can prove weighted $L^p$ boundedness provided certain (comparatively large) loss of derivatives.
\begin{prop}\label{weighted boundedness without rank}
  Let $a(x,\xi)\in L^{\infty}S_{1}^{m}$, $m\leq -n$ and $\varphi(x,\xi)-\langle x, \xi\rangle \in L^{\infty}\Phi^1$. Then $T_{a,\varphi}u(x):=\int e^{i\varphi(x,\xi)}
  \sigma(x,\xi) \widehat{u}(\xi)\, \dd\xi$ is bounded on $L_{w}^{p}$ for $w\in A_p$ and $1<p<\infty$. This result is sharp.
\end{prop}
\begin{proof}
For the low frequency part of the Fourier integral operator we could for example use Proposition \ref{weighted low frequency estimates}. For the high frequency part we may assume that $a(x,\xi)=0$ when $\xi$ is in a neighborhood of the origin. The proof in the case $m<-n$ can be done by a simple integration by parts argument in the integral defining the Schwartz kernel of the operator. Hence the main point of the proof is to deal with the case $m=-n.$  Now the Fourier integral operator $T_{a,\varphi}$ can be written as
\begin{equation}
  T_{a,\varphi}u(x)=\int e^{i\varphi(x,\xi)} a(x,\xi) \widehat{u}(\xi)\, \dd\xi=\int \sigma(x,\xi) e^{i\langle x,\xi\rangle} \widehat{u}(\xi)\, \dd\xi
\end{equation}
with $\sigma(x,\xi)= a(x,\xi) e^{i (\varphi(x,\xi)-\langle x,\xi\rangle)}$ and we can assume that $\sigma(x,\xi)=0$ in a neighborhood of the origin. Therefore, since we have assumed that $\varphi(x,\xi)-\langle x, \xi\rangle \in L^{\infty}\Phi^1$, the operator $T_{a,\varphi}= \sigma(x,D)$ is a pseudo-pseudodifferential operator in the sense of \cite{KS}, belonging to the class $\mathrm{OP}L^{\infty}S^{-n}_{0}.$
We use the Littlewood-Paley partition of unity and decompose the symbol as
\begin{equation*}
\sigma(x,\xi)= \sum_{k=1}^{\infty}\sigma_{k}(x,\xi)
\end{equation*}
with $\sigma_k(x,\xi)= \sigma(x,\xi)\varphi_{k}(\xi)$, $k\geq 1$. We have
$$ \sigma_{k}(x,D)(u)(x)=\frac{1}{(2\pi)^{n}}\int \sigma_{k}(x,\xi)\widehat{u}(\xi)e^{i\langle x, \xi\rangle} \, \dd\xi$$
for $k\geq 1$. We note that
$\sigma_k (x,D)u(x)$ can be written as
\begin{equation*}
\sigma_{k}(x,D)u(x)=\int K_{k}(x,y)u(x-y) \, \dd y
\end{equation*}
with
\begin{equation*}
K_{k}(x,y)=\frac{1}{(2\pi)^{n}}\int \sigma_{k}(x,\xi) e^{i\langle y,\xi\rangle} \, \dd\xi= \check{\sigma}_{k}(x,y),
\end{equation*}
where $\check{\sigma}_k$ here denotes the inverse Fourier transform
of $\sigma_{k}(x,\xi)$ with respect to $\xi$. One observes that
\begin{align*}
 |\sigma_{k}(x,D)u(x)|^{p}&=
\Big|\int K_{k}(x,y)u(x-y)\,\dd y\Big|^{p} \\&=
\Big|\int K_{k}(x,y)\omega(y)\frac{1}{\omega (y)}u(x-y)\, \dd y\Big|^{p},
\end{align*}
with weight functions $\omega(y)$ which will be chosen
momentarily. Therefore, H\"older's inequality yields
\begin{equation}\label{eq2.14}
|\sigma_{k}(x,D)u(x)|^{p}\leq
\bigg\{ \int |K_{k}(x,y)|^{p'} | \omega (y)|^{p'} \, \dd y\bigg\}^{\frac{p}{p'}}
  \bigg\{ \int \frac{| u(x-y)|^{p}}{|\omega (y)|^{p}} \, \dd y\bigg\},
\end{equation}
where $\frac{1}{p} +\frac{1}{p'}=1.$ Now for an $l>\frac{n}{p}$, we
define $\omega$ by
\begin{equation*}
\omega(y)=\begin{cases}
1, &| y| \leq 1; \\
| y|^{l}, & | y| >1.
\end{cases}
\end{equation*}
By Hausdorff-Young's theorem and the symbol estimate, first
for $\alpha=0$ and then for $| \alpha|=l$, we have
\begin{align}\label{eq2.17}
\int |K_k (x,y)|^{p'} \, \dd y &\leq \bigg\{\int |\sigma_k (x,\xi)|^p\, \dd\xi\bigg\}^{\frac{p'}{p}} \lesssim \bigg\{\int_{| \xi| \sim 2^{k}}2^{-npk} \, \dd\xi\bigg\}^{\frac{p'}{p}} \\
\nonumber &\lesssim 2^{kp'(\frac{n}{p}-n)},
\end{align}
and
\begin{align}\label{eq2.18}
\int | K_k (x,y)|^{p'} |y|^{p'l}\,\dd y &\lesssim \bigg\{\int |\nabla_{\xi}^{l}\sigma_k (x,\xi)|^p \, \dd\xi\bigg\}^{\frac{p'}{p}}
\\Ê\nonumber &\lesssim \bigg\{\int_{| \xi| \sim 2^{k}}2^{-kpn}\, \dd \xi\bigg\}^{\frac{p'}{p}} \lesssim 2^{kp'(\frac{n}{p}-n)}.
\end{align}
Hence, splitting the integral into $| y| \leq 1$
and $,| y| > 1$ yields
\begin{equation*}
\bigg\{\int | K_k (x,y)|^{p'}| \omega(y)|^{p'} \, \dd y\bigg\}^{\frac{p}{p'}}
\lesssim \big\{2^{kp'(\frac{n}{p}-n)}\big\}^{\frac{p}{p'}}=2^{kp(\frac{n}{p}-n)}.
\end{equation*}
Furthermore, using Theorem 3.1.3 we have
\begin{equation*}
\int \frac{| u(x-y)|^{p}}{| \omega (y)|^{p}} \, \dd y \lesssim \big(M_{p} u(x)\big)^{p}
\end{equation*}
with a constant that only depends on the dimension $n$. Thus
\eqref{eq2.14} yields
\begin{equation}\label{pointwiseakestim}
 |\sigma_{k}(x,D) u(x)|^{p} \lesssim 2^{k(\frac{n}{p}-n)}\big(M_{p} u(x)\big)^{p}
\end{equation}
Summing in $k$ using \eqref{pointwiseakestim} we obtain
\begin{align*}
|T_{a, \varphi} u(x)|&=|\sigma(x,D) u(x)|\lesssim \bigg\{\sum_{k=1}^{\infty}|\sigma_{k}(x,D)u(x)|^{p}\bigg\}^{\frac{1}{p}} \\
&\lesssim \big(M_{p} u(x)\big)\bigg(\sum_{k=1}^{\infty}2^{k(\frac{n}{p}-n)}\bigg)^{\frac{1}{p}}
\end{align*}
We observe that the series above converges for any $p>1$ and therefore an application of Theorem \ref{maxweight} ends the proof. The sharpness of the result follows from the Counterexample 1.
\end{proof}
The above discussion suggests that without further conditions on the phase, which as it will turn out amounts to a rank condition, the weighted norm inequalities of the type considered in this paper will be false. The following counterexample shows that, even if the phase function satisfies an appropriate rank condition, there is a critical threshold on the loss of derivatives, below which the weighted norm inequalities will fail.
\begin{cex}
We consider the following operator
    $$ T_m = e^{i|D|} \langle D \rangle^{m} $$
and the following functions
    $$ w_{b}(x)=|x|^{-b}, \quad f_{\mu}(x)=\int e^{-i|\xi|+i x \cdot \xi} \langle \xi \rangle^{-\mu} \, \dd \xi. $$
As was mentioned in Example \ref{examples of weights} in Chapter 1, $w_{b} \in A_1$ if $0 \leq b <n,$ from which it also follows that $w:=w_b\chi_{|x|<2}\in A_1$ for $0\leq b<n$, were $\chi_A$ denotes the characteristic function of the set $A$.
Now if $\mu <m+n$ then we claim that $|T_m f_{\mu}(x)| \geq C_{m\mu} |x|^{\mu-m-n}$ on $|x| \leq 2$.
Indeed, we have
   $$ T_m f_{\mu}(x) = \int e^{i x \cdot \xi} \langle \xi \rangle^{m-\mu} \, \dd \xi $$
which is a radial function equal to
\begin{align*}
   |\mathbf{S}^{n-1}| \, |x|^{\mu-m-n} \int_0^{\infty} \widehat{\dd \omega}(r) \big(|x|^2+r^2\big)^{\frac{m-\mu}{2}} r^{n-1} \, \dd r.
\end{align*}
If we denote by $g_{\mu m}$ the function given by the integral, and take a dyadic partition of unity $1=\psi_{0}+\sum_{j=1}^{\infty}\psi(2^{-j}\cdot)$
then
\begin{align*}
     g_{\mu m}(s)&=\int_0^{2} \widehat{\dd \omega}(r) (s^2+r^2)^{\frac{m-\mu}{2}} r^{n-1} \psi_0(r) \, \dd r \\ &\quad + \sum_{j=1}^{\infty} 2^{jn} \int_0^{\infty} \widehat{\dd \omega}(2^jr) (s^2+2^{2j}r^2)^{\frac{m-\mu}{2}} r^{n-1} \psi(r) \, \dd r.
\end{align*}
The first term is continuous if $m-\mu+n>0$ and writing down the integral representing $\widehat{\dd \omega}(2^jr)$ and integrating by parts yields that the series in the second term of $g_{\mu m}$ is a smooth function of $s$. Moreover
\begin{align*}
    g_{\mu m}(0)= \int e^{i\langle \xi, e_{1}\rangle} |\xi|^{m-\mu} \, \dd\xi= C_n |e_1|^{-n-m+\mu}\neq 0,
\end{align*}
since the inverse Fourier transform of a radial homogeneous distribution of degree $\alpha$ is a radial homogeneous distribution of order $-n-\alpha.$
This proves the claim. From this claim it follows that \begin{align*}
     \int |T_{m}f_{\mu}|^p w \, \dd x \geq C_{\mu m} \int_{|x|\leq 2}  |x|^{(\mu-m-n)p-b} \, \dd x,
\end{align*}
and therefore $T_m f_{\mu}  \notin L^p_w$ if $(\mu-m-n)p-b \leq -n.$

Now, we also have $|f_{\mu}(x)| \leq A_{\mu} \big|1-|x|\big|^{\mu-\frac{n+1}{2}}+B_{\mu}$  on $|x| \leq 2$. This is because the function $f_{\mu}$ is radial
\begin{align*}
     f_{\mu}(x) = |\mathbf{S}^{n-1}| \int_{0}^{\infty} \widehat{\dd \omega}(r|x|) e^{-i r} (1+r^2)^{\frac{\mu}{2}} r^{n-1} \, \dd r
\end{align*}
and using the information on the Fourier transform of the measure of the sphere,
\begin{align}
     f_{\mu}(x) = |\mathbf{S}^{n-1}| \, \sum_{\pm} \int_{0}^{\infty} e^{-i r(1\pm|x|)} a_{\pm}(r|x|) (1+r^2)^{\frac{\mu}{2}} r^{n-1} \, \dd r
\end{align}
where
    $$ |\d^{\alpha}a_{\pm}(r)| \leq C_{\alpha} \langle r \rangle^{-\frac{n-1}{2}-\alpha}. $$
We now use a dyadic partition of unity $1=\psi_0+\sum_{k=1}^{\infty} \psi(2^{-k}\cdot)$ on the integral and obtain a sum of terms of the form
\begin{align*}
   2^{kn} \int_{0}^{\infty} e^{-2^ki r(1\pm|x|)} \underbrace{a_{\pm}(2^kr|x|) \psi(r) (1+2^{2k}r^2)^{\frac{\mu}{2}} r^{n-1}}_{=b^{\pm}_k(r,|x|)} \, \dd r
\end{align*}
with
 $$ |\d^{\alpha}_r b^{\pm}_k(r,|x|)| \leq C_{\alpha}
    2^{-(\frac{n-1}{2}-\mu+\alpha)k}. $$
Integration by parts yields
\begin{align*}
   |f_{\mu}| &\leq C_1 + C_2 \sum_{2^k|1-|x|| \leq 1} 2^{-(\frac{n+1}{2}-\mu)k} + C_3 \sum_{2^k|1-|x|| > 1} 2^{-(\frac{n+1}{2}-\mu+N)k} \big|1-|x|\big|^{-N} \\
   & \leq C_1 + C'_2 \big|1-|x|\big|^{\mu-\frac{n+1}{2}}.
\end{align*}
Hence one has
\begin{align*}
     \int |f_{\mu}|^p w \, \dd x \leq A_{\mu} \int_{|x| \leq 2} \big|1-|x|\big|^{\mu p-\frac{n+1}{2}p} |x|^{-b} \, \dd x+B_{\mu} \int_{|x|\leq 2} |x|^{-b} \, \dd x,
\end{align*}
which in turn yields $f_{\mu} \in L^p_w$ if $ \mu>\frac{n+1}{2}-\frac{1}{p}$ and $0 \leq b<n$.
From the estimates above it follows that if $1<p<\infty$ and $T_{m}$ is bounded on $L^p_{w}$ then
   \begin{align}\label{estim on m}
       m \leq -\frac{n-1}{2}-\frac{1}{p}.
   \end{align}
Indeed if $T_{m}$ is bounded on $L^p_{w}$ then we have
\begin{align*}
    -m>\frac{b-n}{p}+n-\mu
\end{align*}
for all $0\leq b <n$ and all $\mu>\frac{n+1}{2}-\frac{1}{p}$. Letting $\mu$ tend to $\frac{n+1}{2}-\frac{1}{p}$ we obtain
\begin{align*}
    m \leq -\frac{b-n}{p}-\frac{n-1}{2}-\frac{1}{p}
\end{align*}
for all $0 \leq b<n$, and letting $b$ tend to $n$ yields \eqref{estim on m}.\\
\noindent Now by Theorem \ref{extrapolation} if $T_{m}$ is bounded on $L^{q}_{w}$ for a fixed $q>1$ and for all $w \in A_{q},$ then by extrapolation it is bounded on all $L^p_{w}$
for all $w\in A_{p}$ and all $1<p<\infty$, therefore since $w \in A_{1}\subset A_{p}$, we conclude that $T_{m}$ is bounded on $L^p_{w}$ for all $1<p<\infty$, which implies that $m$ has to satisfy the inequality
\begin{align*}
    m\leq-\frac{n-1}{2}-\frac{1}{p},
\end{align*}
for all $1<p<\infty$. Letting $p$ tend to $1$, we obtain
\begin{align*}
    m\leq-\frac{n+1}{2},
\end{align*}
which is the desired critical threshold for the validity of the weighted $L^p$ boundedness of the class of Fourier integral operators under consideration in this paper.
\end{cex}
We end up with an example showing that in the most general situation one cannot expect \emph{global} $L^p$ weighted estimates unless the phase satisfies some slightly stronger property than a rank condition.
\begin{cex}
Let $B$ be the unit ball in $\R^n$, we consider the operator
     $$ Tu(x) = (1_{B} * u)(2x) $$
and suppose that this operator is bounded on $L^p_{w}$ with bound $C_{p}=C\big([w]_{A_{p}}\big)$ only depending on $[w]_{A_{p}}$
\begin{align*}
     \|Tu\|_{L^p_{w}} \leq C_{p} \|u\|_{L^p_{w}}, \quad u \in \S(\R^n).
\end{align*}
Note that the $A_{p}$-constant $[w]_{A_{p}}$ is scale invariant. If we apply the estimate to the function $u(\eps \cdot)$ and the weight $w(\eps \cdot)$
and scale it, we obtain
\begin{align*}
     \|T_{\eps}u\|_{L^p_{w}} \leq C_{p} \|u\|_{L^p_{w}}, \quad u \in \S(\R^n)
\end{align*}
with $T_{\eps}u=\eps^{-n} (1_{\eps B}*u)(2x)$. Since $T_{\eps}u$ tends to $u(2x)$ in $L^p_{w}$, by letting $\eps$ tend to $0$ we deduce from the former
\begin{align*}
     \|u(2\,\cdot\,)\|_{L^p_{w}} \leq  C_{p} \|u\|_{L^p_{w}}, \quad u \in \S(\R^n).
\end{align*}
After a change of variable, this would imply
\begin{align*}
    2^{-n} \int |u(x)|^p w(x/2) \, \dd x \leq  C^p_{p} \int |u(x)|^p w(x) \, \dd x
\end{align*}
for all $u \in \S(\R^n),$ hence
\begin{align*}
    w(x/2) \ \leq  C_{p}^p2^n w(x).
\end{align*}
This means that one can expect weighted $L^p$ estimates for $T$ only if $w$ satisfies a doubling property. Note that $T$ can be written as a sum of
Fourier integral operators with amplitudes in $S^{-\frac{n+1}{2}}_{1,0}$ with phases of the form
     $$ \phi_{\pm} = 2\langle x, \xi \rangle \pm |\xi| $$
which satisfy a rank condition. Nevertheless, one  has
     $$ \phi_{\pm} - \langle x,\xi \rangle \notin L^{\infty}\Phi^1. $$
In particular, one cannot generally expect global weighted estimate for Fourier integral operators with phases such that  $\phi - \langle x,\xi \rangle \notin L^{\infty}\Phi^1$
unless the weight satisfies some doubling property.
\end{cex}
\subsection{Invariant formulation in the local boundedness}
The aim of this section is to give an invariant formulation of the rank condition on the phase function, which will be used to prove our weighted norm inequalities for Fourier integral operators. In Counterexample 1 we saw that a rank condition is necessary for the validity of weighted $L^p$ estimates. The following discussion will enable us to give an invariant formulation of the local weighted norm inequalities for operators with amplitudes in $S^{\frac{-(n+1)}{2}\varrho+n(\varrho -1)}_{\varrho, 1-\varrho},$ $\varrho\in[\frac{1}{2},1].$ We refer the reader to \cite{H3} for the properties of Fourier integral operators with amplitudes in $S^{m}_{\varrho, 1-\varrho},$ $\varrho\in (\frac{1}{2}, 1]$ and the paper by A. Greenleaf and G. Uhlmann for the case when $\varrho=\frac{1}{2}.$

The central object in the theory of Fourier integral operators is the canonical
relation
\begin{align*}
     \mathcal{C}_{\phi} = \big\{ (x,\d_{x}\phi(x,\xi),\d_{\xi}\phi(x,\xi),\xi) : (x,\xi) \in \supp a \big\}
\end{align*}
in $T^*\R^n \times T^*\R^n$. We consider the following projection on the space variables
\begin{align*}
    \xymatrix @!0 @C=7pc @R=4pc {& \mathcal{C}_{\phi} \subset T^*\R^n_{x} \times T^*\R^n_{y} &
    \simeq T^{*}(\R^{2n}_{x,y})  \ar[d]^{\pi_0} &  \\
    & & \pi_{0}(\mathcal{C}) \subset \R^{2n}_{xy}&
  }
\end{align*}
with
\begin{align*}
     \pi_{0}(x,\xi,y,\eta)= (x,y).
\end{align*}
The differential of this projection is given by
\begin{align*}
     \dd \pi_{0}(t_{x},t_{\xi},t_{y},t_{\eta}) = (t_{x},t_{y}), \quad t_{\xi}&=\d^2_{x x}\phi \, t_{x} + \d^2_{x\xi}\phi \, t_{\eta} \\
     t_{y}&=\d^2_{\xi x} \phi \, t_{x} + \d^2_{\xi\xi}\phi \, t_{\eta}
\end{align*}
so that its kernel is given by
\begin{align*}
     \big\{(0,\d^2_{x\xi}\phi \, t_{\eta},0,t_{\eta}): t_{\eta} \in \ker \d^2_{\xi\xi}\phi \big\}
\end{align*}
This implies
    $$ \mathop{\rm rank} \dd \pi_{0} = \mathop{\rm codim} \ker \dd \pi_{0} = \mathop{\rm codim} \ker \d^2_{\xi \xi}\phi =  n+\mathop{\rm rank}  \d^2_{\xi \xi}\phi. $$
Our assumption on the phase $\mathop{\rm rank}  \d^2_{\xi \xi}\phi=n-1$ can be invariantly formulated as
    $$ \mathop{\rm rank} \dd \pi_{0} = 2n-1. $$
Using these facts, we will later on be able to show that if $T$ is a Fourier integral operator with amplitude in $S^{-\frac{n+1}{2}\varrho +n(\varrho -1)}_{\varrho ,1-\varrho}$ with $\varrho\in [\frac{1}{2},1]$ whose canonical relation $\mathcal{C}$ is locally the graph of a symplectomorphism, and if
           $$ \mathop{\rm rank}\dd \pi_{0} = 2n-1 $$
     everywhere on $\mathcal{C}$, with $\pi_{0} : \mathcal{C} \to \R^{2n}$ defined by $\pi_{0}(x,\xi,y,\eta)=(x,\eta)$, then there exists a constant $C>0$
     such that
     \begin{align*}
          \|Tu\|_{L^{p}_{w,{\rm loc}}} \leq \|u\|_{L^p_{w,{\rm comp}}}
     \end{align*}
     for all $w \in A_{p}$ and all $1 <p <\infty$.
However, we will actually prove local weighted $L^p$ boundedness of operators with amplitudes in the class $L^{\infty}S^{-\frac{n+1}{2}\varrho +n(\varrho -1)}_{\varrho}$ with $\varrho\in [0,1]$ for which the invariant formulation above lacks meaning, and therefore to keep the presentation of the statements as simple as possible, we will not always (with the exception of Theorem \ref{weighted boundedness for true amplitudes with power weights}) state the local boundedness theorems in an invariant form.

\subsection{Weighted local and global $L^p$ boundedness of Fourier integral operators}
We start this by showing the local weighted $L^p$ boundedness of Fourier integral operators. In Counterexample 1 which was related to Fourier integral operators with linear phases, the Hessian in the frequency variable $\xi$ of the phase function vanishes identically. This suggests that some kind of condition on the Hessian of the phase might be required. It turns out that the condition we need can be formulated in terms of the rank of the Hessian of the phase function in the frequency variable. Furthermore Counterexample 2 which was related to the wave operator, suggests a condition on the order of the amplitudes involved. It turns out that these conditions, appropriately formulated, will indeed yield weighted boundedness of a wide range of Fourier integral operators even having rough phases and amplitudes.
\begin{thm}\label{local weighted boundedness}
  Let $a(x,\xi)\in L^{\infty}S^{m}_{\varrho}$ with $m<-\frac{n+1}{2}\varrho+ n(\varrho-1)$ and $\varrho \in[0,1]$ be compactly supported in the $x$ variable. Let the phase function $\varphi(x,\xi)\in C^{\infty}(\mathbf{R}\times \mathbf{R}\setminus{0})$ homogeneous of degree $1$ in $\xi$ satisfy $\mathrm{rank}\,\partial^2_{\xi\xi}\varphi(x,\xi) =n-1$. Then the corresponding Fourier integral operator is $L^p_w$ bounded for $1<p<\infty$ and all $w\in A_p$.
   \end{thm}
\begin{proof}
The low frequency part of the Fourier integral operator is handled by Proposition \ref{weighted low frequency estimates} part (a). For the high frequency portion, once again we use a Littlewood-Paley decomposition in the frequency variables as in Subsection \ref{Semiclasical reduction subsec} to reduce the operator to its semiclassical piece
  $$T_h u(x)= (2\pi)^{-n} \iint e^{i(\varphi(x,\xi)-\langle y, \xi\rangle)}\,\chi(h\xi)\, a(x,\xi)\, u(y)\, \dd y\, \dd\xi$$
with $\chi(\xi)$ smooth and supported in the annulus $\frac{1}{2}\leq |\xi|\leq 2.$
 Now if we let $\theta(x,\xi):=\varphi(x,\xi)-\langle x, \xi\rangle$, then we have
  \begin{equation}\label{first derivative of theta}
  \vert\nabla_{\xi} \theta\vert \lesssim 1
  \end{equation}
on the support of $a$. Furthermore, if
 \begin{equation}\label{Kernel of local weihted Fourier integral operator}
 T_h (x,y)= (2\pi)^{-n} \int e^{i(\theta(x,\xi)-\langle y, \xi\rangle)}\, \chi(h\xi) a\big(x,\xi\big)\, \dd\xi,
 \end{equation}
then in order to get useful pointwise estimates for the operator $T_h$ we would need to estimate the kernel $T_h(x,y)$.
 Localising in frequency and rotating the coordinates we may assume that  $a(x,\xi)$ is supported in a small conic neighbourhood of  a $\xi_0 =e_n$.
At this point we split the modulus of $T_h$ into
\begin{align*}
|T_h u(x)| &\leq \int_{|y|> 1+2\Vert\nabla_{\xi} \theta\Vert_{L^\infty}}  + \int_{|y|\leq 1+2\Vert\nabla_{\xi} \theta\Vert_{L^\infty}}| T_h (x,y)| \, |u(x-y)| \, \dd y
\\ &:=\textbf{I} + \textbf{II}.
\end{align*}
where there are obviously no critical points on the domain of integration for~\textbf{I}. \\
\textbf{Estimate of I.}
Making the change of variables $\xi\to h^{-1} \xi$ we obtain
  \begin{align*}
  T_h u(x)&= (2\pi)^{-n} h^{-n}\iint e^{ih(\theta(x,\xi)-\langle y, \xi\rangle)}\,\chi(\xi)\, a\big(x,\xi/h\big)\, u(y)\, \dd y\, \dd \xi.
  \end{align*}
Here, since $2\Vert \nabla_{\xi} \theta\Vert_{L^\infty} < 1+ 2\Vert \nabla_{\xi} \theta\Vert_{L^\infty} < |y|$, we have
\begin{equation}\label{local weighted non stationary}
  |\nabla_{\xi} \theta(x,\xi) -y| \geq |y| -\Vert \nabla_{\xi} \theta\Vert_{L^\infty}> \frac{|y|}{2}.
\end{equation}
Also, $|\partial^{\alpha}_{\xi} (\theta(x,\xi)- \langle y, \xi\rangle)| \leq C_\alpha$ for all $|\alpha|\geq 2$ uniformly in $x$ and $y$. Therefore using the non-stationary phase estimate in \cite{H1} Theorem 7.7.1 to \eqref{Kernel of local weihted Fourier integral operator}, the derivative estimates on $a(x,\xi/h)$ and \eqref{local weighted non stationary} yield for $N>0$
\begin{align*}\label{kernel estim local weighted I}
 | T_h (x,y)| &\lesssim h^{-n} h^{N} \sum_{\alpha \leq N} \sup\Big\{\big|\partial ^{\alpha} (\chi a(x,\xi/h)\big|\,|\nabla_{\xi} \theta(x,\xi) -y|^{|\alpha|-2N}\Big\}\\ \nonumber
 &\lesssim h^{-m-n+N\varrho} |y|^{-2N}\lesssim  h^{-m-n+N\varrho} \langle y \rangle ^{-2N},
\end{align*}
where we have used the fact that $|y|>1$ in \textbf{I}. Hence taking $N>\frac{n}{2}$, Theorem \ref{convolve} yields
\begin{equation}
\textbf{I}\lesssim h^{-m-n+N \varrho} \int \langle y \rangle ^{-N} |u(x-y)|\, \dd y \lesssim h^{-m-n+N \varrho} Mu(x).
\end{equation}
\textbf{Estimate of II.}
Making a change of variables $\xi\to h^{-\varrho} \xi$ we obtain
  \begin{align*}
  T_h(x,y)&= (2\pi)^{-n} h^{-n\varrho}\int e^{ih^{-\varrho}(\theta(x,\xi)-\langle y, \xi\rangle)}\,\chi(h^{-\varrho+1}\xi)\, a\big(x,h^{-\varrho}\xi\big)\, \dd \xi
  \\ \nonumber &:=(2\pi)^{-n} h^{-n\varrho}\int   e^{ih^{-\varrho}(\varphi(x,\xi)-\langle y, \xi\rangle)}\, a_{h}\big(x,\xi\big) \, \dd\xi
  \end{align*}
where $a_{h}(x,\xi)$ is compactly supported in $x$, supported in $\xi$ in the annulus $\tfrac{1}{2}h^{\varrho-1} \leq |\xi| \leq 2 h^{\varrho-1}$
and is uniformly bounded together with all its derivatives in $\xi$, by $h^{-m}$.
Here the assumption, $\mathrm{rank}\,\partial^2_{\xi\xi}\varphi(x,\xi) =n-1$ for all $\xi$, yields that $\mathrm{ker}\,\partial^2_{\xi\xi}\varphi(x,\xi_{0}) =\mathrm{span}\, \{\xi_0\}=\mathrm{span}\, \{e_n\}.$ Therefore by the definition of $\theta(x,\xi)$
\begin{equation}
\det \partial^2_{\xi'\xi'}\theta (x,e_n)\neq 0.
\end{equation}
The assumption that $a$ has its $\xi$-support in a small conic neighborhood of $e_n$ implies that if that support is sufficiently small, then
\begin{equation}
    |\det \partial^2_{\xi'\xi'}\theta (x,\xi)| \geq 0, \quad (x,\xi) \in \supp a_h.
\end{equation}
Finally, due to the restriction $1+ 2\Vert \nabla_{\xi} \theta\Vert_{L^\infty} \geq |y|$ and \eqref{first derivative of theta}, one has
\begin{equation} \label{boundedness of derivatives of theta}
|\partial^{\alpha}_{\xi} (\theta(x,\xi)- \langle y, \xi\rangle)| \leq C_\alpha
\end{equation}
for all $|\alpha|\geq 1$ uniformly in $x$ and $y$.

Hence $\theta(x,\xi)- \langle y', \xi'\rangle$ and $h^ma_{h}$ satisfy all the assumptions of the stationary phase estimate in
Lemma \ref{UniformStationaryPhase} with $\lambda = h^{-\varrho}$ and $\lambda^{\mu}=h^{\varrho-1}$, we obtain
\begin{align*}
      \bigg|\int   e^{ih^{-\varrho}(\varphi(x,\xi)-\langle y, \xi\rangle)}\, a_{h}\big(x,\xi\big) \, \dd\xi' \bigg| \lesssim h^{-m} h^{\frac{n-1}{2}\varrho}
      h^{(n-1)(\varrho-1)}
\end{align*}
and  using the fact that the integral in $\xi_{n}$ lies on a segment of size $h^{\varrho-1}$, we get
 \begin{equation}
   |T_h(x, y)|\lesssim h^{-n\varrho} h^{-m} h^{\frac{n-1}{2}\varrho-(1-\varrho)n}\lesssim h^{-m-\frac{n+1}{2}\varrho-(1-\varrho)n}.
 \end{equation}
 This yields that
 \begin{align*}
   \textbf{II} &\lesssim h^{-m-\frac{n+1}{2}\varrho-(1-\varrho)n}  \int_{|y|\leq 1+ 2\Vert \nabla_{\xi} \theta\Vert_{L^\infty}} |u(x-y)|\, \dd y \\
   &\lesssim h^{-m-\frac{n+1}{2}\varrho-(1-\varrho)n} Mu(x)
 \end{align*}
Now adding \textbf{I} and \textbf{II}, taking $N>n$ large enough, using Lemma \ref{Lp:semiclassical}, the assumption $m<-\frac{n+1}{2}\varrho +n(\varrho-1)$ and Theorem \ref{maxweight}, we will obtain the desired result.
\end{proof}
Here  we remark that the condition on the rank of the Hessian on the metric is quite natural and is satisfied by phases like $\langle x, \xi\rangle +|\xi|$ and $\langle x, \xi\rangle +\sqrt{|\xi'|^2 -|\xi''|^2}$ where $\xi=(\xi' , \xi'')$ (with an amplitude supported in $|\xi'|>|\xi''|$), but if we put a slightly stronger condition than the rank condition on the phase, then it turns out that we would not only be able to extend the local result to a global one but also lower the regularity assumption on the phase function to the sole assumption of measurability and boundedness in the spatial variable $x$. Therefore the estimates we provide below will aim to achieve this level of generality.
Having the uniform stationary phase above in our disposal we will proceed with the main high frequency estimates, but before that let us fix a notation.
\subsubsection*{Notation.} Given an $n\times n$ matrix $M$, we denote by $\mathrm{det}_{n-1} M$ the determinant of the matrix $PMP$ where $P$ is the projection to the orthogonal complement of $\mathrm{ker} M$.
\begin{thm}\label{global weighted boundedness}
   Let $a(x,\xi)\in L^{\infty}S^{m}_{\varrho}$ with $m<-\frac{n+1}{2}\varrho+ n(\varrho-1)$ and $\varrho \in[0,1]$. Let the phase function $\varphi(x,\xi)$ satisfy $|\mathrm{det}_{n-1} \partial^2_{\xi\xi}\varphi(x,\xi)|\geq c>0$. Furthermore, suppose that $\varphi (x,\xi)-\langle x,\xi\rangle \in L^{\infty}\Phi^1$
Then the associated Fourier integral operator is bounded on $L_{w}^{p}$, for $1<p<\infty$ and $w\in A_p$.
\end{thm}
\begin{proof}
As before, the low frequency part of the Fourier integral operator is treated using Proposition \ref{weighted low frequency estimates} part (b). For the high frequency part we follow the same line of argument as in the proof of Theorem \ref{local weighted boundedness}. More specifically at the level of showing the estimate \eqref{first derivative of theta}, the lack of compact support in $x$ variable lead us to use our assumption $\varphi(x,\xi)-\langle x,\xi\rangle \in L^{\infty}\Phi^1$ instead, which yields that $\Vert \nabla_{\xi}\theta(x,\xi)\Vert_{ L^{\infty}}\lesssim 1$.
Splitting the kernel of the Fourier integral operator into the same pieces \textbf{I} and \textbf{II} as in the proof of Theorem \ref{local weighted boundedness}. We estimate the \textbf{I} piece exactly in the same way as before but for piece \textbf{II} we proceed as follows. First of all, the assumption $|\mathrm{det}_{n-1} \partial^2_{\xi\xi}\varphi(x,\xi)|\geq c>0$ for all $(x,\xi)$, yields in particular that $|\mathrm{det}_{n-1} (\partial^2_{\xi\xi}(\theta(x,\xi)- \langle y, \xi\rangle)|\geq c>0$. Due to the restriction $1+ 2\Vert \nabla_{\xi} \theta\Vert_{L^\infty} \geq |y|$ and \eqref{first derivative of theta}, one also has
$$|\partial^{\alpha}_{\xi} (\theta(x,\xi)- \langle y, \xi\rangle)| \leq C_\alpha$$
for all $|\alpha|\geq 1$ uniformly in $x$ and $y$, which yields that $\theta(x,\xi)- \langle y, \xi\rangle\in L^\infty S^{0}_{0}.$
This means that all the assumptions in Lemma \ref{UniformStationaryPhase} are satisfied and therefore
we get
\begin{multline*}
   \textbf{II} \lesssim  h^{-m-\frac{n+1}{2}\varrho-(1-\varrho)n} \\ \times
   \int_{|y|\leq 1+ 2\Vert \nabla_{\xi} \theta\Vert_{L^\infty}} |u(x-y)|\, \dd y \lesssim h^{-m-\frac{n+1}{2}\varrho-(1-\varrho)n} Mu(x)
\end{multline*}
Now adding \textbf{I} and \textbf{II}, using Lemma \ref{Lp:semiclassical} and the assumptions $N>n$,  $m<-\frac{n+1}{2}\varrho+n(\varrho-1)$ and Theorem \ref{maxweight} all together yield the desired result.
\end{proof}
\subsubsection{Endpoint weighted boundedness of Fourier integral operators}
The Following interpolation lemma is the main tool in proving the endpoint weighted boundedness of Fourier integral operators.
\begin{lem} \label{interpolgen}
Let $0 \leq \varrho \leq 1$, $1 < p < \infty$ and $m_1 < m_2$. Suppose that
\begin{enumerate}
\item [$(a)$] the Fourier integral operator $T$ with amplitude $a(x,\xi)\in L^\infty S^{m_1}_{\varrho}$ and the phase $\varphi(x,\xi)$ are bounded on $L^p_w$ for a fixed $w \in A_p$, and
\item [$(b)$] the Fourier integral operator $T$ with amplitude $a(x,\xi)\in L^\infty S^{m_2}_{\varrho}$ and the same type of phases as in $($a$)$ are bounded on $L^p$,
\end{enumerate}
where the bounds depend only on a finite number of seminorms in Definition $\ref{defn of amplitudes}.$
Then, for each $m \in (m_1,m_2)$, operators with amplitudes in $L^\infty S^{m}_{\varrho}$ are bounded on $L^p_{w^\nu}$, and
\begin{equation}
\nu = \frac{m_2 - m}{m_2 - m_1}.
\end{equation}
\end{lem}
\begin{proof}
For $a\in L^{\infty} S^{m}_{\varrho}$ we introduce a family of amplitudes $a_{z}(x,\xi):= \langle
\xi \rangle ^{z} a(x,\xi)$, where $z\in \Omega:=\{z\in\mathbb{C};\, m_1 - m\leq \text{Re}\,z\leq m_2 - m \}$.
It is easy to see that, for $| \alpha | \leq C_1$ with $C_1$ large enough and $z\in \Omega$,
\begin{equation*}
| \partial^{\alpha}_{\xi} a_z (x,\xi)| \lesssim (1+| \text{Im}\, z| )^{C_2}\langle \xi \rangle ^{\text{Re}\, z +m-\varrho | \alpha|},
\end{equation*}
for some $C_2$. We introduce the operator $$T_z u:= w^{\frac{m_2 - m - z}{p(m_2 - m_1)}} T_{a_{z}, \varphi}\big(w^{-\frac{m_2 - m - z}{p(m_2 - m_1)}} u\big).$$
First we consider the case of $p\in [1,2]$. In this case, $A_p\subset A_2$ which in turn implies that both $w^{\frac{1}{p}}$ and $w^{-\frac{1}{p}}$ belong to $L^p_{\mathrm{loc}}$ and therefore  for $z\in \Omega$,
$T_z$ is an analytic family of operators in the sense of Stein-Weiss \cite{SW}.
Now we claim that for $z_1\in \mathbb{C}$ with $\text{Re}\, z_1=m_{1} - m$, the
operator $(1+| \text{Im}\,z_1 |)^{-C_2}T_{a_{z_{1}}, \varphi}$ is bounded on $L^{p}_{w}$ with bounds uniform in $z_1$. Indeed the amplitude of this operator is $(1+| \text{Im}\,z_1 |)^{-C_2}a_{z_{1}}(x,\xi)$
which belongs to $L^{\infty}S^{m_1}_{\varrho}$ with constants uniform in $z_1$. Thus, the claim follows from assumption (a). Consequently, we have
\begin{align*}
\Vert T_{z_1} u\Vert^{p}_{L^p} & = (1+| \text{Im}\,z_1 |)^{pC_2}\Big\Vert (1+| \text{Im}\,z_1|)^{-C_2}w^{\frac{m_2 - m - z_1}{p(m_2 - m_1)}}
T_{a_{z_1},\varphi}\big(w^{-\frac{m_2 - m - z_1}{p(m_2 - m_1)}} u\big)\Big\Vert ^{p}_{L^{p}} \\ \nonumber
& \lesssim (1+| \text{Im}\,z_1 |)^{pC_2} \Big\Vert w^{-\frac{m_2 - m - z_1}{p(m_2 - m_1)}} u \Big\Vert ^{p}_{L^{p}_{w}} \\ \nonumber
& =(1+| \text{Im}\,z_1 |)^{pC_2}\Vert u\Vert^{p}_{L^p},
\end{align*}
where we have used the fact that $\big|w^{\frac{m_2 - m - z_1}{(m_2 - m_1)}} \big|= w$.

Similarly if $z_2 \in \mathbb{C}$ with $\text{Re}\,z_2 = m_2 - m$, then $\big|w^{\frac{m_2 - m - z_2}{(m_2 - m_1)}} \big|= 1$
and the amplitude of the operator $(1+| \text{Im}\,z_2 |)^{-C_2}T_{a_{z_2}}$ belongs to
$L^\infty S^{m_2}_\varrho$ with constants uniform in $z_2$. Assumption (b) therefore implies that
\begin{equation*}
\Vert T_{z_2} u\Vert^{p}_{L^p}\lesssim (1+| \text{Im}\,z_2 |)^{pC_2}\Vert u\Vert^{p}_{L^p}.
\end{equation*}
Therefore the complex interpolation of operators in \cite{BS} implies that for $z = 0$ we have
\begin{equation*}
\Vert T_{0} u\Vert^{p}_{L^p}= \Big\Vert w^{\frac{m_2 - m}{p(m_2 - m_1)}} T_{a,\varphi}\big(w^{-\frac{m_2 - m}{p(m_2 - m_1)}}u\big)\Big\Vert^{p}_{L^p}\leq C\Vert u\Vert^{p}_{L^p}.
\end{equation*}
Hence, setting $v = w^{-\frac{m_2 - m}{p(m_2 - m_1)}}u$ this reads
\begin{equation}
\Vert T_{a,\varphi}v\Vert^{p}_{L^p_{w^\nu}}\leq C\Vert u\Vert^{p}_{L^p_{w^\nu}},
\end{equation}
where $\nu = (m_2 - m)/(m_2 - m_1)$.
This ends the proof in the case $1\leq p\leq 2$.
At this point we recall the fact that if a linear operator $T$
is bounded on $L^p_w$, then its adjoint $T^{\ast}$ is bounded on $L^{p'}_{ w^{1-p'}}$. Therefore, in the case $p > 2$, we apply the above proof to
$T^{\ast}_{a, \varphi}$, with $p'\in [1,2)$ and $v= w^{1-p'}$, which yields that $T^{\ast}_{a, \varphi}$ is bounded on
$L^{p'}_ {v^\nu}$ and since  $w \in A_p$, we have $v \in A_{p'}$ and so $T_a$ is bounded on
$L^p_{v^{(1-p)\nu}} = L^p_{w^{(1-p')(1-p)\nu}}= L^p_{w^\nu}$, which concludes the proof of the theorem.
\end{proof}
Now we are ready to prove our main result concerning weighted boundedness of Fourier integral operators. This is done by combining our previous results with a method based on the properties of the $A_p$ weights and complex interpolation.
\begin{thm}\label{endpointweightedboundthm}
Let $a(x,\xi)\in L^{\infty}S^{-\frac{n+1}{2}\varrho+n(\varrho -1)}_{\varrho}$ and $\varrho \in [0,1].$ Suppose that either
\begin{enumerate}
\item[$(a)$] $a(x,\xi)$ is compactly supported in the $x$ variable and the phase function $\varphi(x,\xi)\in C^{\infty} (\mathbf{R}^ n \times \mathbf{R}^{n}\setminus 0)$, is positively homogeneous of degree $1$ in $\xi$ and satisfies, $\det\partial^2_{x\xi}\varphi(x,\xi) \neq 0$ as well as $\mathrm{rank}\,\partial^2_{\xi\xi}\varphi(x,\xi) =n-1$; or
\item[$(b)$] $\varphi (x,\xi)-\langle x,\xi\rangle \in L^{\infty}\Phi^1,$ $\varphi $ satisfies the rough non-degeneracy condition as well as $|\mathrm{det}_{n-1} \partial^2_{\xi\xi}\varphi(x,\xi)|\geq c>0$.
\end{enumerate}
Then the operator $T_{a,\varphi}$ is bounded on $L^{p}_{w}$ for $p\in (1,\infty)$ and all $w\in A_p$. Furthermore, for $\varrho=1$ this result is sharp.
\end{thm}
\begin{proof}
The sharpness of this result for $\varrho=1$ is already contained in Counterexample 2 discussed above. The key issue in the proof is that both assumptions in the statement of the theorem guarantee the weighted boundedness for $m<-\frac{n+1}{2}\varrho +n(\varrho -1).$ The rest of the argument is rather abstract and goes as follows. By the extrapolation Theorem \ref{extrapolation}, it is enough to show the boundedness of $T_{a,\varphi}$ in weighted $L^2$ spaces with weights in the class $A_2$. Let us fix $m_2$ such that $-\frac{n+1}{2}\varrho+n(\varrho -1)<m_2 < \frac{n}{2}(\varrho-1).$ By Theorem \ref{open}, given $w\in A_{2}$ choose $\varepsilon$ such that $w^{1+\varepsilon}\in A_{2}$. For this $\varepsilon$ take $m_1<-\frac{n+1}{2}\varrho+n(\varrho -1)$ in such a way that the straight line $L$ that joins points with coordinates $(m_1,1+ \varepsilon)$ and $(m_2,0)$, intersects the line $x=-\frac{n+1}{2}\varrho+n(\varrho -1)$ in the $(x,y)$ plane in a point with coordinates $(-\frac{n+1}{2}\varrho+n(\varrho -1),1)$. Clearly this procedure is possible due to the fact that we can choose the point $m_1$ on the negative $x$ axis as close as we like to the point $-\frac{n+1}{2}\varrho+n(\varrho -1).$ By Theorem \ref{local weighted boundedness}, given $\varphi(x,\xi)\in C^{\infty} (\mathbf{R}^ n \times \mathbf{R}^{n}\setminus 0)$, positively homogeneous of degree $1$ in $\xi,$ satisfying $\mathrm{rank}\,\partial^2_{\xi\xi}\varphi(x,\xi) =n-1,$ and $a\in L^{\infty} S^{m_1}_{\varrho},$ the Fourier integral operators $T_{a, \varphi}$ are bounded operators on $L^2_{w^{1+\varepsilon}}$
for $w \in A_2$ and, by Theorem \ref{global L2 boundedness smooth phase rough amplitude}, or rather its proof, the Fourier integral operators with amplitudes in $L^{\infty}S^{m_2}_{\varrho}$ compactly supported in the spatial variable, and phases that are positively homogeneous of degree $1$ in the frequency variable and satisfying the non-degeneracy condition $\det\partial^2_{x\xi}\varphi(x,\xi) \neq 0,$ are bounded on $L^2$. Therefore, by Lemma \ref{interpolgen}, the Fourier integral operators $T_{a, \varphi}$ with phases and amplitudes as in the statement of the theorem are bounded operators on $L^2_w$. The proof of part (b) is similar and uses instead the interpolation between the weighted boundedness of Theorem \ref{global weighted boundedness} and the unweighted $L^2$ boundedness result of Theorem \ref{Intro:L2Thm}. The details are left to the interested reader.
\end{proof}
If we don't insist on proving weighted boundedness results valid for all $A_p$ weights then, it is possible to improve on the number of derivatives in the estimates and push the numerology almost all the way to those that guaranty unweighted $L^p$ boundedness.  Therefore, there is the trade-off between the generality of weights and loss of derivatives as will be discussed below.
\begin{thm}\label{weighted boundedness for true amplitudes with power weights}
Let $\mathcal{C} \subset (\mathbf{R}^{n} \times \mathbf{R}^{n}\setminus {0})\times (\mathbf{R}^{n}\times \mathbf{R}^{n}\setminus {0}),$ be a conic manifold which is locally a canonical graph, see e.g. \cite{H3} for the definitions. Let $\pi: \mathcal{C}\to \mathbf{R}^{n}\times \mathbf{R}^{n}$ denote the natural projection. Suppose that for every $\lambda_{0}=(x_0, \xi _0 , y_0 , \eta_0 ) \in\mathcal {C}$ there is a conic neighborhood $U_{\lambda_{0}}\subset \mathcal{C}$ of $\lambda_{0}$ and a smooth map $\pi_{\lambda_{0}}: \mathcal{C}\cap U_{\lambda_{0}}\to \mathcal{C}$, homogeneous of degree $0$, with $\mathrm{rank}\,d\pi_{\lambda_{0}} = 2n-1,$ such that
  $$\pi= \pi\circ \pi_{\lambda_{0}}.$$
Let $T\in I^{m}_{\varrho, \mathrm{comp}}(\mathbf{R}^{n}\times \mathbf{R}^{n}; \mathcal{C})$ $($see \cite{H3}$)$ with $\frac{1}{2}\leq\varrho \leq 1$ and $m < (\varrho-n)|\frac{1}{p}-\frac{1}{2}|.$
Then for all $w \in A_p $ there exists $\alpha\in (0,1)$ depending on $m$, $\varrho$, $\delta$, $p$ and $[w]_{A_p}$ such that,
for all $\varepsilon \in[0,\alpha] ,$ the Fourier integral operator $T_{a, \varphi}$ is bounded on $L^{p}_{w^\varepsilon}$ where $1<p<\infty.$
\end{thm}
\begin{proof}
  By the equivalence of phase function theorem, which for $\frac{1}{2}< \varrho \leq 1$ is due to H\"ormander \cite{H3} and for $\varrho=\frac{1}{2}$ is due to Greenleaf-Uhlmann \cite{GU}, we reduce the study of operator $T$ to that of a finite linear combination of operators which in appropriate local coordinate systems have the form
  \begin{equation}
   T_{a}u(x)=(2\pi)^{-n} \iint e^{i\varphi(x,\xi)-i\langle y, \xi\rangle}\, a(x,\xi) \,u (y) dy\, d\xi,
  \end{equation}
  where $a\in S^{m}_{\varrho, 1-\varrho}$ with compact support in $x$ variable, and $\varphi$ homogeneous of degree 1 in $\xi,$ with $\det\partial^2_{x\xi}\varphi(x,\xi) \neq 0$ and $\mathrm{rank}\,\partial^2_{\xi\xi}\varphi(x,\xi) =n-1 .$ If $m\leq -\frac{n+1}{2}\varrho + n(\varrho -1),$ then Theorem \ref{endpointweightedboundthm} case $(a)$ yields the result, so we assume that $m> -\frac{n+1}{2}\varrho + n(\varrho -1).$ Also by assumption of the theorem we can find a $m_2 ,$ which we shall fix from now on, such that $m< m_2 <  (\varrho-n)|\frac{1}{p}-\frac{1}{2}|$ and $m_1 <-\frac{n+1}{2}\varrho + n(\varrho -1).$ Now a result of Seeger-Sogge-Stein, namely Theorem 5.1 in \cite {SSS} yields that operators $T_a$ with amplitudes compactly supported in the $x$ variable in the class $S^{m_2}_{\varrho, 1-\varrho},$ and phase functions $\varphi(x,\xi)$ satisfying $\mathrm{rank}\,\partial^2_{\xi\xi}\varphi(x,\xi) =n-1$ are bounded on $L^p.$  Furthermore by Theorem \ref{endpointweightedboundthm} case $(a),$ The operators $T_a$ with $a\in S^{m_1}_{\varrho, \delta}$ are bounded on $L^{p}_{w},$ $p\in (1,\infty).$ Therefore, Lemma \ref{interpolgen} yields the desired result.
\end{proof}
A similar result also holds for operators with amplitudes in $S^{m}_{\varrho,\delta}$ with without any rank condition on the phase function.
\begin{thm}\label{weighted boundedness with power weights}
Let $a(x,\xi)\in S^{m}_{\varrho,\delta},$  $\varphi(x,\xi)$ be a strongly non-degenerate phase function with $\varphi(x,\xi)-\langle x,\xi\rangle \in \Phi^1,$ and $\lambda:=\min(0,n(\varrho-\delta)),$  with either of the following ranges of parameters:
                    \begin{enumerate}
\item [$(1)$]  $0\leq \varrho \leq 1$, $0\leq \delta\leq 1, $ $1 \leq p \leq 2$ and
                      $$ m<n(\varrho -1)\bigg (\frac{2}{p}-1\bigg)+\big(n-1\big)\bigg(\frac{1}{2}-\frac{1}{p}\bigg)+ \lambda\bigg(1-\frac{1}{p}\bigg); $$ or
\item [$(2)$]  $0\leq \varrho \leq 1$, $0\leq \delta\leq 1, $ $2 \leq p \leq \infty$ and
                      $$ m<n(\varrho -1)\bigg (\frac{1}{2}-\frac{1}{p}\bigg)+ (n-1)\bigg(\frac{1}{p}-\frac{1}{2}\bigg) +\frac{\lambda}{p}.$$
                      \end{enumerate}
Then for all $w \in A_p $ there exists $\alpha\in (0,1)$ depending on $m$, $\varrho$, $\delta$, $p$ and $[w]_{A_p}$ such that,
for all $\varepsilon \in[0,\alpha] ,$ the Fourier integral operator $T_{a, \varphi}$ is bounded on $L^{p}_{w^\varepsilon}$ where $1<p<\infty.$
\end{thm}
\begin{proof}
  The proof is similar to that of Theorem \ref{weighted boundedness for true amplitudes with power weights} and we only consider the proof in case (1), since the other case is similar. We observe that since $\Phi^1\subset \Phi^2$ and $\langle x, \xi \rangle \in \Phi^2,$ the assumption $\varphi (x,\xi)-\langle x,\xi\rangle \in \Phi^1,$ implies that $\varphi(x,\xi) \in \Phi^2.$ To proceed with the proof we can assume that $m> -n$ because otherwise by Proposition \ref{weighted boundedness without rank} there is nothing to prove. The assumption of the theorem, enables us to find $m_2$ such that
   $$ m< m_2 <n(\varrho -1)\bigg (\frac{2}{p}-1\bigg)+\big(n-1\big)\bigg(\frac{1}{2}-\frac{1}{p}\bigg)+ \lambda\bigg(1-\frac{1}{p}\bigg) $$
and $m_1 <-n.$ Now Theorem \ref{main L^p thm for smooth Fourier integral operators} yields that operators $T_a$ with amplitudes in the class $S^{m_2}_{\varrho, \delta},$ and strongly non-degenerate phase functions $\varphi(x,\xi) \in \Phi^2$ are bounded on $L^p.$  Furthermore Proposition \ref{weighted boundedness without rank} yields that operators $T_a$ with $b\in S^{m_1}_{\varrho, \delta}$ are bounded on $L^{p}_{w}.$ Therefore, Lemma \ref{interpolgen} yields once again the desired result for the range $1<p\leq 2.$
\end{proof}

Finally, for operators with non-smooth amplitudes we can prove the following:
\begin{thm}\label{weighted boundedness with power weights nonsmooth symbols}
Let $a(x,\xi)\in L^{\infty}S^{m}_{\varrho},$ $0\leq \varrho\leq 1,$ and let $\varphi(x,\xi)-\langle x,\xi\rangle \in \Phi^1,$ with a strongly non-degenerate $\varphi$ and either of the following ranges of parameters:
                    \begin{enumerate}
\item[$(a)$]  $1 \leq p \leq 2$ and
                      $$ m<\frac{n}{p}(\varrho -1)+\big(n-1\big)\bigg(\frac{1}{2}-\frac{1}{p}\bigg); $$ or
\item[$(b)$]  $2 \leq p \leq \infty$ and
                      $$ m<\frac{n}{2}(\varrho -1)+ (n-1)\bigg(\frac{1}{p}-\frac{1}{2}\bigg).$$
                      \end{enumerate}
Then for all $w \in A_p $ there exists $\alpha\in (0,1)$ depending on $m$, $\varrho$, $p$ and $[w]_{A_p}$ such that,
for all $\varepsilon \in[0,\alpha] ,$ the Fourier integral operator $T_{a, \varphi}$ is bounded on $L^{p}_{w^\varepsilon}.$
\end{thm}
\begin{proof}
The proof is a modification of that of Theorem \ref{weighted boundedness with power weights}, where one also uses Theorem \ref{main L^p thm for Fourier integral operators with smooth phase and rough amplitudes}. The straightforward modifications are left to the interested reader.
\end{proof}
Here we remark that if in the proofs of Theorems \ref{weighted boundedness with power weights} and \ref{weighted boundedness with power weights nonsmooth symbols} we would have used Theorem \ref{endpointweightedboundthm} case (b) instead of Proposition \ref{weighted boundedness without rank} in the proof above, then we would obtain a similar result, under the condition $|\mathrm{det}_{n-1} \partial^2_{\xi\xi}\varphi(x,\xi)|\geq c>0$ on the phase, but with an improved $\alpha$ as compared to those in the statements of Theorems \ref{weighted boundedness with power weights} and \ref{weighted boundedness with power weights nonsmooth symbols}.

\section{Applications in harmonic analysis and partial differential equations}
In this chapter, we use our weighted estimates proved in the previous chapter to show the boundedness of constant coefficient Fourier integral operators in weighted Triebel-Lizorkin spaces. This is done using vector-valued inequalities for the aforementioned operators. We proceed by establishing weighted and unweighted $L^p$ boundedness of commutators of a wide class of Fourier integral operators with functions of bounded mean oscillation (BMO), where in some cases we also show the weighted boundedness of iterated commutators. The boundedness of commutators are proven using the weighted estimates of the previous chapter and a rather abstract complex analytic method. Finally in the last section, we prove global unweighted and local weighted estimates for the solutions of the Cauchy problem for $m$-th and second order hyperbolic partial differential equations on $\mathbf{R}^n .$
\subsection{Estimates in weighted Triebel-Lizorkin spaces}
In this section, we investigate the problem of the boundedness of certain classes on Fourier integral operators in weighted Triebel-Lizorkin spaces. The result obtained here can be viewed as an example of the application of weighted norm inequalities for FIO's. The main reference for this section is \cite{GR} and we will refer the reader to that monograph for the proofs of the statements concerning vector-valued inequalities.
\begin{defn}
An operator $T$ defined in $L^{p}(\mu)$ $($ this denotes $L^{p}$ spaces with measure $d\mu$$)$ is called {\it{linearizable}} if there exits a linear operator $U$ defined on $L^{p}(\mu)$ whose values are Banach space-valued functions such that
\begin{equation}
  \vert Tu(x)\vert =\Vert Uu(x)\Vert_{B},\,\,\,\, u\in L^{p}(\mu)
\end{equation}
\end{defn}
We shall use the following theorem due to Garcia-Cuerva and Rubio de Francia, whose proof can be found in \cite{GR}.
\begin{thm}\label{vectorvalued thm}
  Let $T_{j}$ be a sequence of linearizable operators and suppose that for some fixed $r>1$ and all weights $w\in A_r$
  \begin{equation}\label{vectorvalued estim}
    \int \vert T_{j} u(x) \vert ^{r} w(x) \, \dd x \leq C_{r}(w) \int \vert u(x)\vert ^r w(x) \, \dd x,
  \end{equation}
 with $C_{r}(w)$ depending on the weight $w$.
  Then for $1<p,\, q<\infty$ and $w\in A_{p}$ one has the following weighted vector-valued inequality
   \begin{equation}\label{weightedvecvaluedineq}
    \Big\Vert \Big\{\sum_{j} \vert T_{j} u_{j} \vert ^{q}\Big\}^{\frac{1}{q}}\Big\Vert_{L^{p}_w}
    \leq C_{p,q}(w)  \Big\Vert\Big \{\sum_{j} \vert u_{j} \vert ^{q}\Big\}^{\frac{1}{q}}\Big\Vert_{L^{p}_w}.
  \end{equation}
\end{thm}
Next we recall the definition of the weighted Triebel-Lizorkin spaces, see e.g. \cite{T}
\begin{defn}
 Start with a partition of unity $\sum_{j=0}^{\infty} \psi_{j}(\xi)=1$, where $\psi_{0}(\xi)$ is supported in $\vert \xi\vert\leq 2$, $\psi_{j}(\xi)$ for $j\geq 1$ is supported in $2^{j-1} \leq \vert \xi\vert \leq 2^{j+1}$ and $ |\partial^{\alpha} \psi_{j}(\xi)| \leq C_{\alpha} 2^{-j|\alpha|}$, for $j\geq 1$. Given $s\in \mathbb{R},$ $1\leq p\leq\infty,$ $1\leq q\leq \infty ,$ and $w\in A_p ,$ a tempered distribution $u$ belongs to the weighted {\it{Triebel-Lizorkin}} space $F^{s,p}_{q,\,w}$ if
\begin{equation} \label{Tribliz definition}
  \Vert u\Vert_{F^{s,p}_{q,\,w}}:= \bigg\Vert \bigg\{\sum_{j=0}^{\infty}\vert 2^{js}\psi_{j}(D) u\vert^q\bigg\}^{\frac{1}{q}}\bigg\Vert_{L^p_{w}}<\infty,
\end{equation}
\end{defn}
From this it follows that for a linear operator $T$ the estimate
\begin{equation}
  \Vert Tu\Vert_{F^{s',p}_{q,\,w}}\lesssim \Vert u\Vert_{F^{s,p}_{q,\,w}},
\end{equation}
is implied by
\begin{equation}\label{trieblizorestim}
 \bigg\Vert \bigg\{\sum_{j=0}^{\infty}\vert 2^{js'}\psi_{j}(D) Tu\vert^q\bigg\}^{\frac{1}{q}}\bigg\Vert_{L^p_{w}} \lesssim
 \bigg\Vert \bigg\{\sum_{j=0}^{\infty}\vert 2^{js}\psi_{j}(D) u\vert^q\bigg\}^{\frac{1}{q}}\bigg\Vert_{L^p_{w}}.
\end{equation}
Now if one is in the situation where $[T,\psi_j]=0,$ then \eqref{trieblizorestim} is equivalent to
\begin{equation}\label{trieblizorestimequiv}
\bigg\Vert \bigg\{\sum_{j=0}^{\infty}\vert 2^{j(s'-s)}T (2^{js}\psi_{j}(D)u)\vert^q\bigg\}^{\frac{1}{q}}\bigg\Vert_{L^p_{w}}
\lesssim \bigg\Vert \bigg\{\sum_{j=0}^{\infty}\vert 2^{js}\psi_{j}(D) u\vert^q\bigg\}^{\frac{1}{q}}\bigg\Vert_{L^p_{w}}.
\end{equation}
Therefore, setting $T_j:=2^{j(s'-s)}T$ and $u_j := 2^{js}\psi_j u$ and assuming that $s\geq s'$, \eqref{trieblizorestimequiv} has the same form as the the vector-valued inequality \eqref{weightedvecvaluedineq} and follows from \eqref{vectorvalued estim}.
Using these facts yields the following result,
\begin{thm}
Let $a(\xi)\in S_{1,0}^{-\frac{n+1}{2}}$ and $\varphi\in \Phi^{1}$ with $|\mathrm{det}_{n-1} \partial^2_{\xi\xi}\varphi(\xi)|\geq c>0$. Then for $s\geq s',$ $1<p<\infty ,$ $1<q<\infty ,$ and $w\in A_p, $ the Fourier integral operator $$Tu(x) = \frac{1}{(2\pi)^{n}}\int_{\mathbb{R}^{n}}
e^{i \varphi(\xi)+i \langle x, \xi\rangle} a(\xi)\hat{u}(\xi) \, d\xi$$ satisfies the estimate
\begin{equation}
  \Vert Tu\Vert_{F^{s',p}_{q,\,w}}\lesssim \Vert u\Vert_{F^{s,p}_{q,\,w}}
\end{equation}
\end{thm}
\begin{proof}
  We only need to check that $T_{j}= 2^{j(s'-s)}T$ satisfies \eqref{vectorvalued estim}. But this follows from the assumption $s\geq s'$ and Theorem \ref{endpointweightedboundthm} part (b) concerning the global weighted boundedness of Fourier integral operators.
\end{proof}

\subsection{Commutators with BMO functions}
In this section we show how our weighted norm inequalities can be used to derive the $L^p$ boundedness of commutators of
functions of bounded mean oscillation with a wide range of pseudodifferential operators. We start with the precise definition of a function of bounded mean oscillation.
\begin{defn} \label{BMO}
A locally integrable function $b$ is of bounded mean oscillation if
\begin{equation}
\Vert b\Vert_{\mathrm{BMO}} := \sup_{B}\q \int_{B} \vert b(x)-b_B\vert \, \dd x<\infty,
\end{equation}
where the supremum is taken over all balls in $\R^n$. We denote the set of such functions by $\mathrm{BMO}$.
\end{defn}
For $b\in \mathrm{BMO}$ it is well-known that for any $\gamma<\frac{1}{2^ n e}$, there exits a constant $C_{n,\gamma}$ so that for $u\in\text{BMO}$ and all balls $B$,
\begin{equation}\label{bmoestimate}
  \frac{1}{\vert B\vert}\int_{B} e^{\gamma\vert b(x)-b_B\vert/\Vert u\Vert_{\mathrm{BMO}}}\, \dd x \leq C_{n,\gamma}.
\end{equation}
For this see \cite[p.~528]{G}.
The following abstract lemma will enable us to prove the $L^p$ boundedness of the BMO commutators of Fourier integral operators.
\begin{lem}\label{commutatorlem}
For $1< p <\infty$, let $T$ be a linear operator which is bounded on $L^{p}_{w^\varepsilon}$ for all $w\in A_{p}$ for some fixed $\varepsilon\in (0,1]$.
Then given a function $b\in\mathrm{BMO}$, if $\Psi(z):= \int e^{zb(x)}T(e^{-z b(x)} u)(x) v(x)\, \dd x$ is holomorphic in a disc $\vert z\vert <\lambda$,
then the commutator $[b,T]$ is bounded on $L^{p}$.
\end{lem}
\begin{proof}
Without loss of generality we can assume that $\Vert b\Vert_{\mathrm{BMO}}=1$. We take $u$ and $v$ in $C^{\infty}_{0}$ with $\Vert u\Vert_{L^p} \leq1$ and $\Vert v\Vert_{L^{p'}} \leq 1$, and an application of H\"older's inequality to the holomorphic function $\Psi(z)$ together with the assumption on $v$ yield
\begin{equation*}\label{Fp}
| \Psi(z)|^p \leq \int e^{p\,\text{Re}\,z b(x)} | T(e^{-\,z\, b(x)}u) |^{p}\, \dd x .
\end{equation*}
Our first goal is to show that the function $\Psi(z)$ defined above is bounded on a disc with centre at the origin and sufficiently small radius. At this point we recall a lemma due to Chanillo \cite{Chan} which states that if $\Vert b\Vert_{\mathrm{BMO}}=1$, then for $2<s<\infty$, there is an $r_0$ depending on $s$ such that for all $r\in [-r_0 , r_0]$, $e^{rb(x)}\in A_{\frac{s}{2}}.$\\
Taking $s=2p$ in Chanillo's lemma, we see that there is some $r_1$ depending on $p$ such that for $| r| <r_1$, $e^{rb(x)}\in A_{p}$.
Then we claim that if $R:=\text{min}\, (\lambda, \frac{\varepsilon r_{1}}{p})$ and $|z|< R$ then $|\Psi(z)| \lesssim 1$. Indeed since $R< \frac{\varepsilon r_{1}}{p}$ we have $|\text{Re}\, z| < \frac{\varepsilon r_{1}}{p}$ and therefore $| \frac{p\,\text{Re}\,z}{\varepsilon}| <r_1$. Therefore Chanillo's lemma yields that for $|z|< R$, $w:=e^{\frac{p\,\text{Re}\,z}{\varepsilon}b(x)}\in A_{p}$ and since $e^{p\,\text{Re}\,z b(x)}=w^{\varepsilon}$, the assumption of weighted boundedness of $T$ and the $L^{p}$ bound on $u$, imply that for $|z|<R$
\begin{equation*}
\begin{split}
| \Psi(z)|^{p} & \leq \int e^{p\,\text{Re}\,z b(x)} | T(e^{-\,z\, b(x)}u) |^{p}\, \dd x \\
& = \int w^{\varepsilon} | T(e^{-\,z\, b(x)}u) |^{p}\, \dd x \\
& \lesssim \int w^{\varepsilon} | e^{-\,z\, b(x)}u |^{p}\, \dd x=\int w^{\varepsilon}  w^{-\varepsilon}| u |^{p}\, \dd x\lesssim 1,
\end{split}
\end{equation*}
and therefore $|\Psi(z)|\lesssim 1$ for $|z| <R$.
Finally, using the holomorphicity of $\Psi(z)$ in the disc $|z|<R$, Cauchy's integral formula applied to the circle $|z|=R'<R$, and the estimate $| \Psi(z)|\lesssim 1$, we conclude that
\begin{equation*}
| \Psi'(0)| \leq \frac{1}{2\pi} \int_{| z| =R'} \frac{| \Psi(\zeta)|}{| \zeta^2|} \,|\dd\zeta| \lesssim 1.
\end{equation*}
By construction of $\Psi(z)$, we actually have that $\Psi'(0)=\int v(x) [b,T]u(x) \, \dd x$ and the definition of the $L^p$ norm of the operator $[b, T]$ together with the assumptions on $u$ and $v$ yield at once that $[f, T]$
is a bounded operator from $L^p$ to itself for $p$.
\end{proof}
The following lemma guarantees the holomorphicity of
    $$ \Psi(z):= \int e^{zb(x)} T_{a, \varphi}(e^{-z b(x)} u)(x) v(x)\, \dd x, $$
when $T_{a,\varphi}$ is a $L^2$ bounded Fourier integral operator.
\begin{lem}\label{holomorflem}
Assume that $\varphi$ is a strongly non-degenerate phase function in the class $\Phi^2$ and suppose that either:
\begin{enumerate}
\item [$(a)$] $T_{a, \varphi}$ is a Fourier integral operator with $a\in S^{m}_{\varrho, \delta}$, $0\leq \varrho\leq 1,$ $0\leq \delta<1, $ $m=\min(0, \frac{n}{2}(\varrho-\delta))$ or
\item [$(b)$] $T_{a, \varphi}$ is a Fourier integral operator with $a\in L^{\infty} S^{m}_{\varrho},$  $0\leq \varrho \leq 1,$ $m<\frac{n}{2}(\varrho-1).$
\end{enumerate}
Then given $b\in\mathrm{BMO}$ with $\Vert b\Vert_{\mathrm{BMO}}=1$ and $u$ and $v$ in $C^{\infty}_{0}$, there exists $\lambda>0$ such that the function $\Psi(z):= \int e^{zb(x)} T_{a, \varphi} (e^{-z b(x)} u)(x) v(x)\, \dd x$ is holomorphic in the disc $\vert z\vert <\lambda$.
\end{lem}
\begin{proof}
\setenumerate[0]{leftmargin=0pt,itemindent=20pt}
\begin{enumerate}
 \item [(a)] From the explicit representation of $\Psi(z)$
 \begin{equation}\label{psi(z) representation}
\Psi(z)= \iiint a(x,\xi)\, e^{i\varphi(x,\xi)-i\langle y, \xi\rangle}\,e^{zb(x)-z b(y)}\,v(x)\,u(y) \,\dd y\, \dd\xi\, \dd x
 \end{equation}
 we can without loss of generality assume that $a(x,\xi)$ has compact $x-$support. For $f\in \mathscr{S}$ and $\varepsilon\in(0,1)$ we take $\chi(\xi)\in C^{\infty}_{0}(\mathbf{R}^{n})$ such that $\chi(0)=1$ and set
 \begin{equation}
   T_{a_{\varepsilon}, \varphi} f(x)=\int  a(x,\xi)\,\chi(\varepsilon \xi)\, e^{i\varphi(x,\xi)}\,\hat{f}(\xi)\, d\xi.
 \end{equation}
Using this and the assumption of the compact $x-$support of the amplitude, one can see that for $f\in \mathscr{S}$, $\lim_{\varepsilon \to 0}T_{a_{\varepsilon}, \varphi} f=  T_{a, \varphi}f$ in the Schwartz class $\mathscr{S}$ and also $\lim_{\varepsilon \to 0}\Vert  T_{a_{\varepsilon}, \varphi} f-  T_{a, \varphi}f\Vert _{L^2}=0.$ Since $ a(x,\xi)\,\chi(\varepsilon \xi)\in S^{m}_{\varrho, \delta}$ with seminorms that are independent of $\varepsilon,$ it follows from our assumptions on the amplitude and the phase and Theorem \ref{Calderon-Vaillancourt for FIOs} that $\Vert  T_{a_{\varepsilon}, \varphi} f\Vert_{L^2} \lesssim \Vert f\Vert_{L^2}$ with a $L^2$ bound that is independent of $\varepsilon.$ Therefore, the density of $\mathscr{S}$ in $L^2$ yields
\begin{equation}\label{L2 Convergence of T-epsilon to T}
\lim_{\varepsilon \to 0}\Vert  T_{a_{\varepsilon}, \varphi} f-  T_{a, \varphi}f\Vert _{L^2}=0,
\end{equation}
for all $f\in L^2.$ Now if we define
\begin{align}\label{defn psi epsilon}
\Psi_{\varepsilon} (z)&:=\int e^{zb(x)} T_{a_{\varepsilon}, \varphi} (e^{-z b(x)} u)(x) v(x)\, \dd x\\ \nonumber
&=\iiint a(x,\xi)\,\chi(\varepsilon \xi)\, e^{i\varphi(x,\xi)-i\langle y, \xi\rangle}\,e^{zb(x)-z b(y)}\,v(x)\,u(y) \, \dd y\, \dd\xi\, \dd x,
\end{align}
then the integrand in the last expression is a holomorphic function of $z$. Furthermore, from \eqref{bmoestimate} and the assumption $\Vert b\Vert_{\textrm{BMO}}=1$ one can deduce that for all $p\in [1, \infty)$ and $|z|<\frac{\gamma}{p}$, and all compact sets $K$ one has
\begin{equation}\label{local integrability of exponentials}
\int_{K} e^{\pm p\,\mathrm{Re}\, z\, b(x)} \, \dd x\leq C_\gamma (K).
\end{equation}
This fact shows that $u e^{-z\, b}$ and $v e^{z\, b}$ both belong to $L^{p}$ for all $p\in [1,\infty)$ provided $|z|<\frac{\gamma}{p}$. These together with the compact support in $\xi$ of the integrand defining $\Psi_{\varepsilon} (z)$ and uniform bounds on the amplitude in $x$, yield the absolute convergence of the integral in \eqref{defn psi epsilon} and therefore $\Psi_{\varepsilon}(z)$ is a holomorphic function in $|z|<1.$
Now we claim that for $\gamma$ as in \eqref{bmoestimate},
\begin{equation*}
\lim_{\varepsilon\to 0} \sup_{\vert z\vert< \frac{\gamma}{2}} | \Psi_{\varepsilon}(z)-\Psi(z)|=0.
\end{equation*}
Indeed, since $\frac{\gamma}{2}<\frac{1}{2}$, one has for $\vert z\vert <\frac{\gamma}{2}$
\begin{align*}
| \Psi_{\varepsilon}(z)-\Psi(z)|&=\left|\int  v(x) e^{z b(x)}\,\big[T_{a_{\varepsilon}, \varphi}- T_{a, \varphi}\big](e^{-z b} u)(x)\, \dd x\right|\\
&\leq \|v\, e^{zb}\|_{L^2} \big\Vert [T_{a_{\varepsilon}, \varphi}-T_{a, \varphi}]( u\, e^{-z b})\big\Vert_{L^2}\\
&\leq \|v\|_{L^{\infty}}\left \{\int_{\mathrm{supp}\,v} e^{2\mathrm{Re}\, z b(x)}\, \dd x\right \}^{\frac{1}{2}} \Vert [T_{a_{\varepsilon}, \varphi}-T_{a, \varphi}]( u\, e^{-z b})\Vert_{L^2}.
\end{align*}
Therefore, using \eqref{local integrability of exponentials} with $p=2$ and \eqref{L2 Convergence of T-epsilon to T} yield that $$\varlimsup_{\varepsilon\to 0} \sup_{\vert z\vert< \frac{\gamma}{2}}|\Psi_{\varepsilon}(z)-\Psi(z)|=0$$ and hence $\Psi(z)$ is a holomorphic function in $\vert z\vert <\lambda$ with $\lambda \in (0, \frac{\gamma}{2}).$
\item [(b)] Using the semiclassical reduction in the proof of Theorem \ref{global L2 boundedness smooth phase rough amplitude}, we decompose the operator $T_{a ,\varphi}$ into low and high frequency parts, $T_0$ and $T_h$. From this it follows that $\Psi_{0}(z):= \int e^{zb(x)} T_{0} (e^{-z b(x)} u)(x) v(x)\, \dd x$ can be written as
\begin{multline}\label{Psi0}
\Psi_{0}(z)= \\
\int \bigg\{\int\int e^{i\varphi(x,\xi) -i\langle y, \xi \rangle}\, \chi_{0}(\xi)\, a(x,\xi)\, u(y)\, e^{ -zb(y)} \, \dd y\, \dd \xi\bigg\} v(x)\,e^{zb(x)}\, \dd x
\end{multline}
and $\Psi_{h}(z):= \int e^{zb(x)} T_{h} (e^{-z b(x)} u)(x) v(x)\, \dd x$ is given by
\begin{multline}\label{Psih}
\Psi_{h}(z)= \\ h^{-n}\int \bigg\{\iint e^{\frac{i}{h}\varphi(x,\xi)-\frac{i}{h}\langle y, \xi \rangle}\, \chi(\xi)\, a(x,\xi/h)\, u(y) \,e^{ -zb(y)}\, \dd y\, \dd \xi\bigg\} \, v(x)\,e^{zb(x)} \, \dd x,
\end{multline}
Now we claim that for $\Psi_{0}(z)$ and $\Psi_h(z)$ are holomorphic in $|z|<1$. To see this, we reason in a way similar to the proof of part (a). Namely, using the compact support in $\xi$ of the integrands of \eqref{Psi0} and \eqref{Psih} and uniform bounds on the amplitude in $x$, yield the absolute convergence of the integrals in \eqref{Psi0} and \eqref{Psih} and therefore $\Psi_{0}(z)$ and $\Psi_{h}(z)$ are holomorphic functions in $|z|<1.$
Next we proceed with a uniform estimate (in $z$) for $\Psi_{h}(z)$. For this we use once again that $u e^{-z\, b}$ and $v e^{z\, b}$ both belong to $L^{2}$ provided $|z|<\frac{\gamma}{2}.$ Therefore the Cauchy-Schwartz inequality and \eqref{semiclassical L2 pieces} yield
\begin{align}\label{main Psi_h z estim}
| \Psi_{h}(z)|&=\left|\int  v(x) e^{zb(x)}T_{h}(e^{-z b} u)(x)\, \dd x\right| \\ \nonumber
&\leq \Vert  u\, e^{-z b} \Vert_{L^2} \Vert T^{\ast}_{h}(v\, e^{zb})\Vert_{L^2} \leq \Vert u\, e^{-z b} \Vert_{L^2} \Vert T_{h}T^{\ast}_{h}(v\, e^{zb})\Vert^{1/2}_{L^2}\Vert v\, e^{zb}\Vert^{1/2}_{L^2} \\ \nonumber
&\leq h^{-m-(1-\varrho)M/2} \Vert u\, e^{-zb}\Vert_{L^2} \Vert v\, e^{zb}\Vert_{L^2} \lesssim h^{-m-(1-\varrho)M/2}
\end{align}
Hence, $\vert \Psi_{h}(z)\vert \lesssim h^{-m-(1-\varrho)M/2}$ and setting $h=2^{-j}$, using $m<\frac{n}{2}(\varrho-1)$ and summing in $j$ we would have a uniformly convergent series of holomorphic functions which  therefore converges to a holomorphic function and by taking a $\lambda$ in the interval $(0,\frac{\gamma}{2})$ we conclude the holomorphicity of $\Psi(z)$ in $|z|<\lambda.$
\end{enumerate}
\end{proof}
Lemmas \ref{commutatorlem} and \ref{holomorflem} yield our main result concerning commutators with BMO functions, namely
\begin{thm}\label{Commutator estimates for FIO}
Suppose either
\begin{enumerate}
  \item [$(a)$] $T\in I^{m}_{\varrho, \mathrm{comp}}(\mathbf{R}^{n}\times \mathbf{R}^{n}; \mathcal{C})$ with $\frac{1}{2} \leq \varrho\leq 1$ and $m<(\varrho-n)|\frac{1}{p}-\frac{1}{2}|,$ satisfies all the conditions of Theorem $\ref{weighted boundedness for true amplitudes with power weights}$ or;
  \item [$(b)$] $T_{a,\varphi}$ with $a\in S^{m}_{\varrho, \delta},$ $0\leq \varrho \leq 1$, $0\leq \delta\leq 1,$ $\lambda= \min(0, n(\varrho-\delta))$ and $\varphi(x,\xi)$ is a strongly non-degenerate phase function with $\varphi(x,\xi)-\langle x,\xi\rangle \in \Phi^1,$ where in the range $1<p\leq 2,$
$$ m<n(\varrho -1)\bigg (\frac{2}{p}-1\bigg)+\big(n-1\big)\bigg(\frac{1}{2}-\frac{1}{p}\bigg)+ \lambda\bigg(1-\frac{1}{p}\bigg);$$ and in the range $2 \leq p <\infty$
                      $$ m<n(\varrho -1)\bigg (\frac{1}{2}-\frac{1}{p}\bigg)+ (n-1)\bigg(\frac{1}{p}-\frac{1}{2}\bigg) +\frac{\lambda}{p};$$ or
\item [$(c)$] $T_{a,\varphi}$ with $a\in L^{\infty}S^{m}_{\varrho},$ $0\leq \varrho \leq 1$ and $\varphi$ is a strongly non-degenerate phase function with $\varphi(x,\xi) -\langle x, \xi\rangle \in \Phi^1,$ where in the range $1<p\leq 2,$
$$ m<\frac{n}{p}(\varrho -1)+\big(n-1\big)\bigg(\frac{1}{2}-\frac{1}{p}\bigg),$$
and for the range $2 \leq p <\infty$
                      $$ m<\frac{n}{2}(\varrho -1)+ (n-1)\bigg(\frac{1}{p}-\frac{1}{2}\bigg).$$
\end{enumerate}
Then for $b\in \mathrm{BMO}$, the commutators $[b, T]$ and $[b, T_{a,\varphi}]$ are bounded on $L^p$ with $1<p<\infty.$
\end{thm}
\begin{proof}
\setenumerate[0]{leftmargin=0pt,itemindent=20pt}
\begin{enumerate}
\item [(a)] One reduces $T$ to a finite sum of operators of the form $T_a$ as in the proof of Theorem \ref{weighted boundedness for true amplitudes with power weights}. That theorem also yields the existence of an $\varepsilon \in (0,1)$ such that $T_a$ with $a\in S^{m}_{\varrho, 1-\varrho}$ and $m<(\varrho-n)|\frac{1}{p}-\frac{1}{2}|$ is $L^{p}_{w^\varepsilon}-$bounded. Moreover, since $m<(\varrho-n)|\frac{1}{p}-\frac{1}{2}|\leq 0,$ and $1-\varrho\leq \varrho,$ Theorem \ref{Calderon-Vaillancourt for FIOs} yields that $T_a$ is $L^2$ bounded. Hence, if $$\Psi(z)=\int e^{z b(x)} T_{a} (e^{-z b(x)} u)(x) v(x) \, \dd x$$ with $u$ and $v$ in $C^{\infty}_{0},$ then Lemma \ref{holomorflem} yields that $\Psi(z)$ is holomorphic in a neighbourhood of the origin. Therefore, Lemma \ref{commutatorlem} implies that the commutator $[b, T_a ]$ is bounded on $L^p$ and the linearity of the commutator in $T$ allows us to conclude the same result for a finite linear combinations of operators of the same type as $T_a$. This ends the proof of part (a).
\item [(b)] The proof of this part is similar to that of part (a). Here we observe that for any ranges of $p$ in the statement of the theorem, the order of the amplitude $m<\min(0,\frac{n}{2}(\varrho-\delta))$ and so $T_{a, \varphi}$ is $L^2$ bounded. Now, application of \ref{weighted boundedness with power weights}, Theorem \ref{Calderon-Vaillancourt for FIOs} and Lemma \ref{holomorflem} part (a), concludes the proof.
\item[(c)] The proof of this part is similar to that of part (b). For any ranges of $p,$ the order of the amplitude $m<\frac{n}{2}(\varrho-1)$ and so $T_{a, \varphi}$ is $L^2$ bounded. Therefore, Theorem \ref{weighted boundedness with power weights nonsmooth symbols}, Theorem \ref{global L2 boundedness smooth phase rough amplitude} and Lemma \ref{holomorflem} part (b), yield the desired result.
\end{enumerate}
\end{proof}
Finally, the weighted norm inequalities with weights in all $A_p$ classes have the advantage of implying weighted boundedness of repeated commutators. Namely, one has
\begin{thm}\label{k-th commutator estimates}
Let $a(x,\xi)\in L^{\infty}S^{-\frac{n+1}{2}\varrho+n(\varrho -1)}_{\varrho}$ and $\varrho \in[0,1].$ Suppose that either
\begin{enumerate}
\item[$(a)$] $a(x,\xi)$ is compactly supported in the $x$ variable and the phase function $\varphi(x,\xi)\in C^{\infty} (\mathbf{R}^ n \times \mathbf{R}^{n}\setminus 0)$, is positively homogeneous of degree $1$ in $\xi$ and satisfies, $\det\partial^2_{x\xi}\varphi(x,\xi) \neq 0$ as well as $\mathrm{rank}\,\partial^2_{\xi\xi}\varphi(x,\xi) =n-1$; or
\item[$(b)$] $\varphi (x,\xi)-\langle x,\xi\rangle \in L^{\infty}\Phi^1,$ $\varphi $ satisfies either the strong or the rough non-degeneracy condition $($depending on whether the phase is spatially smooth or not$)$, as well as $|\mathrm{det}_{n-1} \partial^2_{\xi\xi}\varphi(x,\xi)|\geq c>0$.
\end{enumerate}
Then, for $b \in \mathrm{BMO}$ and $k$ a positive integer, the $k$-th commutator defined by
\begin{equation*}
T_{a, b,k} u(x):= T_{a}\big((b(x)-b(\cdot))^{k}u\big)(x)
\end{equation*}
is bounded on $L^{p}_w$ for each $w \in A_p$ and $p\in(1, \infty)$.
\end{thm}
\begin{proof}
The claims in (a) and (b) are direct consequences of Theorem \ref{endpointweightedboundthm} and Theorem 2.13 in \cite{ABKP}.
\end{proof}
\subsection{Applications to hyperbolic partial differential equations}
It is wellknown, see e.g. \cite{D}, that the Cauchy problem for a strictly hyperbolic partial differential equation
\begin{equation}\label{hyp Cauchy prob}
  \begin{cases}
    D^{m}_{t}+ \sum_{j=1}^{m} P_{j}(x,t,D_x ) D_{t}^{m-j}, \,\,\, t\neq 0\\
    \partial_{t}^{j} u|_{t=0}=f_ j (x),\,\,\, 0\leq j\leq m-1
  \end{cases}
\end{equation}
can be solved locally in time and modulo smoothing operators by
\begin{equation}\label{FIO representation of the solution}
  u(x,t)= \sum_{j=0}^{m-1}\sum_{k=1}^{m} \int_{\mathbf{R}^n} e^{i\varphi_{k}(x,\xi,t)} a_{jk}(x,\xi,t) \, \widehat{f_{j}}(\xi)\, \dd\xi
\end{equation}
where $a_{jk}(x,\xi,t)$ are suitably chosen amplitudes depending smoothly on $t$ and belonging to $S^{-j}_{1, 0}$, and the phases $\varphi_{k}(x,\xi,t)$ also depend smoothly on $t,$ are strongly non-degenerate and belong to the class $\Phi^2.$ This yields the following:
\begin{thm}
  Let $u(x,t)$ be the solution of the hyperbolic Cauchy problem \eqref{hyp Cauchy prob} with initial data $f_j$. Let $m_{p}= (n-1)|\frac{1}{p}-\frac{1}{2}|$, for a given $p\in [1, \infty]$. If $f_j \in H^{s+m_{p}-j, p}(\mathbf{R}^n)$ and $T\in (0, \infty)$ is fixed, then for any $\varepsilon >0$ the solution $u(\cdot, t) \in H^{s-\varepsilon,p}(\mathbf{R}^n),$ satisfies the global estimate
  \begin{equation}
    \Vert u(\cdot, t)\Vert_{H^{s-\varepsilon,p}}\leq C_{T} \sum_{j=0}^{m-1}\Vert f_{j}\Vert_{H^{s+m_p-j,p}}, \,\,\, t\in [-T,T],\,\, p\in [1, \infty]
  \end{equation}
\end{thm}
\begin{proof}
The result follows at once from the Fourier integral operator representation \eqref{FIO representation of the solution} and Corollary \ref{LinftySm1 cor}.
\end{proof}
The representation formula \eqref{FIO representation of the solution} also yields the following local weighted estimate for the solution of the Cauchy problem for the second order hyperbolic equation above and in particular for variable coefficient wave equation. In this connection we recall that $H_{w}^{s}:=\{ u\in \mathscr{S}';\, (1-\Delta)^{\frac{s}{2}}u \in L_{w}^{p}, \, w\in A_{p}\}. $
\begin{thm}
  Let $u(x,t)$ be the solution of the hyperbolic Cauchy problem \eqref{hyp Cauchy prob} with $m=2$ and initial data $f_j .$ For $p\in (1,\infty),$ if $f_j \in H_{w}^{s+\frac{n+1}{2}-j, p}(\mathbf{R}^n)$ with $w\in A_p$, and if $T\in (0,\infty)$ is small enough, then the solution $u(\cdot, t)$ is in $H_{w}^{s,p}(\mathbf{R}^n)$ and satisfies the weighted estimate
  \begin{equation}
    \Vert \chi u(\cdot, t)\Vert_{H_{w}^{s,p}}\leq C_{T} \sum_{j=0}^{1}\Vert f_{j}\Vert_{H_{w}^{s+\frac{n+1}{2}-j,p}}, \,\,\, t\in [-T,T]\setminus \{0\},\,\forall w\in A_{p},
  \end{equation}
  and all $\chi \in C^{\infty}_{0}(\mathbf{R}^n).$
\end{thm}
\begin{proof}
In the case when $m=2$ then one has the important property that $$\mathrm{rank}\, \partial^{2}_{\xi \xi} \varphi(x, \xi, t)=n-1,$$ for $t \in [-T,T]\setminus \{0\}$ and small $T.$ This fact and the localization of the solution $u(x,t)$ is the spatial variable $x,$ enable us to use Theorem \ref{endpointweightedboundthm} in the case $\varrho=1,$ from which the theorem follows.
\end{proof}


\begin{thebibliography}{m}
\bibitem{ABKP}
J.~\'{A}lvarez, R.~J.~Bagby, D.~S.~Kurtz, and C.~P\'{e}rez, \textit{Weighted estimates for commutators of linear operators},
Studia Math. \textbf{104} (1993), no. 2, 195--209.
\bibitem{AF}
K.~Asada and D.~Fujiwara, \textit{On some oscillatory integral transformations in $L^{2}({\bf R}^{n})$}. Japan. J. Math. (N.S.) 4 (1978), no. 2, 299--361.
\bibitem{Be}
M.~R.~Beals, \textit{$L^{p}$ boundedness of Fourier integral operators},
Mem. Amer. Math. Soc. \textbf{38}  (1982), no. 264, viii+57 pp.
\bibitem{RBE}
R.~Beals, \textit{Spatially inhomogeneous pseudodifferential operators II},
Comm. Pure Appl. Math.  \textbf{27}  (1974), 161--205.
\bibitem{BS} C.~Bennett and R.~C.~Sharpley,
\textit{Interpolation of Operators}, Volume 129 (Pure and Applied Mathematics) Academic Press (February 28, 1988).
\bibitem{CV}
A.~P.~Calder\'{o}n and R.~Vaillancourt, \textit{On the boundedness of pseudo-differential operators},
J. Math. Soc. Japan  \textbf{23}  (1971) 374--378.
\bibitem{Chan}
S.~Chanillo, \textit{Remarks on commutators of pseudo-differential operators},
Multidimensional complex analysis and partial differential equations (S\~ao Carlos, 1995),  33--37, Contemp. Math., 205, Amer. Math. Soc., Providence, RI, 1997.
\bibitem{CNR}
E.~Cordero, F.~Nicola and L.~Rodino, \textit{On the global boundedness of Fourier integral operators},
arXiv:0804.3928v1 [math.FA] (2008).
\bibitem{CR}
S.~Coriasco and M.~Ruzhansky, \textit{On the boundedness of Fourier integral operators on $L^p (\R^n )$},
C. R. Acad. Sci. Paris, Ser I.  \textbf{348} (2010), 847--851.
\bibitem{D}
H.~Duistermaat, \textit{Fourier integral operators}, Birk\"auser (1995).
\bibitem{Esk}
G.~I.~Eskin, \textit{Degenerate elliptic pseudodifferential equations of principal type} (Russian),
Mat. Sb. (N.S.)  \textbf{82} (124)  (1970), 585--628
\bibitem{Fuji}
D.~Fujiwara, \textit{A global version of Eskin's theorem},
J. Fac. Sci. Univ. Tokyo Sect. IA Math.  \textbf{24}  (1977), no. 2, 327--339.
\bibitem{GR}
J.~Garc\'ia-Cuerva and J.~Rubio de Francia, \textit{Weighted norm inequalities and related topics},
North-Holland Mathematics Studies, 116. Notas de Matem\'atica [Mathematical Notes], 104. North-Holland Publishing Co., Amsterdam, 1985.
\bibitem{G}
L.~Grafakos, \textit{Classical and modern Fourier analysis}, Pearson Education, Inc., Upper Saddle River, NJ, 2004.
\bibitem{GU}
A.~Greenleaf and G.~Uhlmann, \textit{Estimates for singular Radon transforms and pseudodifferential operators with singular symbols},
 J. Funct. Anal.  \textbf{89}  (1990),  no. 1, 202--232.
\bibitem{Guz}
M.~De Guzman, \textit{Change of variables formula without continuity},
American Mathematical Monthly \textbf{87}, No. 9 (1980), 736--739.
\bibitem{H0}
L.~H\"ormander, \textit{Pseudo-differential operators and hypoelliptic equations}, Singular integrals (Proc. Sympos. Pure Math., Vol. X, Chicago, Ill., 1966). 138--183.
\bibitem{H1}
--------------, \textit{The analysis of linear partial differential operators} I. \textit{Distribution theory and Fourier analysis}, Springer Verlag (1985).
\bibitem{H2}
--------------, \textit{The analysis of linear partial differential operators} IV. \textit{Fourier integral operators}, Springer Verlag (1985).
\bibitem{H3}
--------------, \textit{Fourier integral operators} I,
Acta Math.  \textbf{127}  (1971), no. 1-2, 79--183.
\bibitem{H4}
--------------, \textit{On the $L\sp{2}$ continuity of pseudo-differential operators},
Comm. Pure Appl. Math. \textbf{24} (1971), 529--535.
\bibitem{J}
J-L.~Journe, \textit{Calder\'on-Zygmund operators, pseudodifferential operators and the Cauchy integral of Calderon},
Lecture Notes in Mathematics, 994, Springer-Verlag, Berlin, 1983.
\bibitem{KS}
C.~E.~Kenig, W.~Staubach, \textit{Psi-pseudodifferential operators and estimates for maximal oscillatory integrals},
(2006), to appear in Studia Mathematica.
\bibitem{KW}
D.~S.~Kurtz and R.~L.~Wheeden, \textit{Results on weighted norm inequalities for multipliers},
Trans. Amer. Math. Soc.  \textbf{255}  (1979), 343--362.
\bibitem{M}
N.~Miller, \textit{Weighted Sobolev spaces and pseudodifferential operators with smooth symbols},
Trans. Amer. Math. Soc.  \textbf{269}  (1982), no. 1, 91--109.
\bibitem{MSS1}
G.~Mockenhaupt, A.~Seeger, and C.~D.~Sogge, \textit{Wave front sets, local smoothing and Bourgain's circular maximal theorem},
Ann. of Math. (2)  \textbf{136}  (1992),  no. 1, 207--218.
\bibitem{MSS2}
------------------------------------------, \textit{Local smoothing of Fourier integral operators and Carleson-Sj\"olin estimates},
J. Amer. Math. Soc.  \textbf{6}  (1993),  no. 1, 65--130.
\bibitem{Muck}
B.~Muckenhoupt, \textit{Weighted Norm Inequalities for the Hardy Maximal Function},
Trans. Amer.  Math. Soc., \textbf{165} (1972), 207--226.
\bibitem{Rod}
L.~Rodino, \textit{On the boundedness of pseudo differential operators in the class $L^{m}_{\varrho,1}$},
Proc. Amer. Math. Soc.  \textbf{58}  (1976), 211--215.
\bibitem{Ruz 1}
M.~Ruzhansky, \textit{On local and global regularity of Fourier integral operators}, New developments in pseudo-differential operators,  1
85--200, Oper. Theory Adv. Appl., 189, Birkh\"auser, Basel, 2009.
\bibitem{Ruz 2}
M.~Ruzhansky and M.~Sugimoto, \textit{Global calculus of Fourier integral operators, weighted estimates, and applications to global analysis of hyperbolic equations}, Pseudo-differential operators and related topics,  65--78, Oper. Theory Adv. Appl., 164, Birkh\"auser, Basel, 2006.
\bibitem{Sch}
J.~T.~Schwartz, \textit{Nonlinear functional analysis}, Notes by H. Fattorini, R. Nirenberg and H. Porta, with an additional chapter by Hermann Karcher,
Notes on Mathematics and its Applications. Gordon and Breach Science Publishers, New York-London-Paris, 1969.
\bibitem{SSS}
A.~Seeger, C.~D.~Sogge, and E.~M.~Stein, \textit{Regularity properties of Fourier integral operators},
Ann. of Math. \textbf{134} (1991), no. 2, 231--251.
\bibitem{So}
C.~D.~Sogge, \textit{Fourier integral in classical analysis}, Cambridge University Press (1993).
\bibitem{So1}
-----------, \textit{Propagation of singularities and maximal functions in the plane},
Invent. Math. \textbf{104} (1991), 349--376.
\bibitem{S}
E.~M.~Stein, \textit{Harmonic analysis: real-variable methods, orthogonality, and oscillatory integrals},
Princeton Mathematical Series, 43. Princeton University Press, Princeton, NJ, (1993).
\bibitem{SW2}
E.~M.~Stein and G.~Weiss, \textit{Interpolation of operators with change of measures},
Trans. Amer. Math. Soc.  \textbf{87}  (1958), 159--172
\bibitem{SW}
------------------------, \textit{Introduction to Fourier analysis on Euclidean spaces},
Princeton Mathematical Series, No. 32. Princeton University Press, Princeton, N.J., 1971. x+297 pp.
\bibitem{T}
M.~E.~Taylor, \textit{Partial differential equations} III. \textit{Nonlinear equations}, Corrected reprint of the 1996 original, Applied Mathematical Sciences, 117. Springer-Verlag, New York, 1997.
\bibitem{Yab}
K.~Yabuta, \textit{Calder\'on-Zygmund operators and pseudodifferential operators},
Comm. Partial Differential Equations  \textbf{10}  (1985),  no. 9, 1005--1022
\end{thebibliography}
\end{document}